\documentclass[a4paper,11pt,leqno]{article}
\usepackage{amsmath,amsfonts,amsthm,amssymb}
\usepackage[left=3cm, right=3cm, top=3cm, bottom=3.5cm]{geometry}
\usepackage{graphicx}
\usepackage{mathrsfs}
\usepackage{tikz}
\usepackage{marginnote}
\numberwithin{equation}{section}
\usepackage[final]{pdfpages}

\theoremstyle{plain}
\newtheorem{theorem}{Theorem}[section]

\newtheorem{lemma}[theorem]{Lemma}
\newtheorem{proposition}[theorem]{Proposition}
\newtheorem{corollary}[theorem]{Corollary}

\theoremstyle{remark}

\theoremstyle{definition}

\newcommand{\N}{\mathbb{N}}
\newcommand{\R}{\mathbb{R}}

\newcommand{\I}{\mathbb{I}}
\newcommand{\1}{\mathbb{1}}

\newcommand{\SF}{\mathbb{S}}

\usepackage{color}

\begin{document}

\title{Connectivity of the diffuse interface and fine structure of minimizers in the Allen-Cahn theory of phase transitions. }
\author{{ Giorgio Fusco\footnote{Dipartimento di Matematica Pura ed Applicata, Universit\`a degli Studi dell'Aquila, Via Vetoio, 67010 Coppito, L'Aquila, Italy; e-mail:{\texttt{fusco@univaq.it}}}}
}
\date{}
\maketitle
\begin{abstract}
In the Allen-Cahn theory of phase transitions, minimizers partition the domain in subregions, the sets where a minimizer is near to one or to another of the zeros of the potential. These subregions that model the phases  are separated by a tiny \emph{Diffuse Interface}. Understanding the shape of this diffuse interface is an important step toward the description of the structure of minimizers.

We assume Dirichlet data and present general conditions on the domain and on the boundary datum ensuring the connectivity of the diffuse interface.

Then we restrict to the case of two dimensions and show that the phases can be separated, in a certain optimal way, by a connected \emph{network} with a well defined structure. This network is contained in the diffuse interface and is a priori unknown.

Under general assumption on the potential and on the Dirichlet datum, we show that, if we assume that the phase are connected, then we can obtain precise information on the shape of the network and in turn a detailed description of the fine structure of minimizers. In particular we can characterize the shape and the size of the various phases and also how they depend on the surface tensions.
\end{abstract}


\section{Introduction}
We study minimizers $u^\epsilon$ of the Allen-Cahn energy subjected to Dirichlet conditions:
\begin{equation}
\begin{split}
&J_\Omega^\epsilon(v)=\min_{v\in\mathscr{A}} J_\Omega^\epsilon(v)=\int_\Omega\Big(\epsilon\frac{\vert\nabla v\vert^2}{2}+\frac{1}{\epsilon}W(v)\Big)dx,\;\;0<\epsilon<<1,\\
&\mathscr{A}=\{v\in H^1(\Omega;\R^m):v\vert_{\partial\Omega}=v_0^\epsilon\},
\end{split}
\label{min}
\end{equation}
where $\Omega\subset\R^n$, $n\geq 2$ is a bounded smooth domain, $0<\epsilon<<1$ a parameter and $v_0^\epsilon:\partial\Omega\rightarrow\R^m$ is a smooth boundary datum that may also depend on $\epsilon$.

We assume
\begin{description}
\item[$h_1$] $W:\R^m\rightarrow\R$ is a smooth nonnegative potential that vanishes on a finite set
\[A=\{W=0\}=\{a_1,\ldots,a_N\},\;\;N\geq 2.\]
The hessian matrix $W_{zz}(a)$, $a\in A$ is positive definite and there is $K>0$ such that
\[W_z(z)\cdot z>0,\;\;\vert z\vert\geq K.\]
\item[$h_2$] There is a constant $M$ independent of $\epsilon>0$ such that
\[\vert v_0^\epsilon\vert+\epsilon\vert\nabla v_0^\epsilon\vert\leq M.\]
\end{description}

The study of minimizers $u^\epsilon$ of problem \eqref{min} is an interesting and challenging mathematical problem
 which is also relevant from a physical point of view. In the theory of phase transitions the functional
 $J_\Omega^\epsilon$ may be regarded as a model for the free energy of a substance which can exists in $N$
 equally preferred phases corresponding to the zeros $a_1,\ldots,a_N$ of $W$.

Assumption $h_2$ implies the existence of $v_{\mathrm{test}}^\epsilon\in\mathscr{A}$ which satisfies
\begin{equation}
J_\Omega^\epsilon(v_{\mathrm{test}}^\epsilon)\leq C,
\label{UB0}
\end{equation}
where here and in the following $C$ denotes a positive constant independent of $\epsilon$. For example we
can fix an $a\in A$ and take
 \[v_{\mathrm{test}}^\epsilon(x)=\left\{\begin{array}{l}
 \frac{1}{\epsilon}(v_0^\epsilon(y)(\epsilon-\vert x-y\vert)+\vert x-y\vert a),\;\text{if}\;\vert x-y\vert\leq\epsilon,\\
a,\;\text{if}\;\vert x-y\vert>\epsilon,
\end{array}\right.
\]
where $y$ is the orthogonal projection of $x$ on $\partial\Omega$.

 From the upper bound \eqref{UB0} and standard methods of variational calculus we have the existence of a minimizer $u^\epsilon$ of problem \eqref{min}. The minimizer is actually a classical solution of the Allen-Cahn equation
\begin{equation}
\left\{\begin{array}{l}\Delta u=W_z(u),\;x\in\Omega,\\
u=v_0^\epsilon,\;\;x\in\partial\Omega
\end{array}\right.
\label{AC}
\end{equation}
which is the Euler-Lagrange equation associated to $J_\Omega^\epsilon$. From regularity theory we also have the bound
\begin{equation}
\vert u^\epsilon\vert+\epsilon\vert\nabla u^\epsilon\vert\leq M^\prime,
\label{H1bound}
\end{equation}
for some constant $M^\prime>0$.

Once the existence of a minimizer $u^\epsilon$ is known the problem becomes to understand its fine structure. In particular to determine how $u^\epsilon$ depends on the shape of $\Omega$, on the boundary datum $v_0^\epsilon$ and on the connections among the zeros of $W$ together with the surface tensions $\sigma_{aa^\prime}$ of these connections ($\sigma_{aa^\prime}$ is the density of energy of a planar interface that separate phase $a$ from phase $a^\prime$). An object of special interest in this direction is the diffuse interface
\begin{equation}
\mathscr{I}^{\epsilon,\delta}=\{x\in\bar{\Omega}:\min_{a\in A}\vert u^\epsilon(x)-a\vert>\delta\},
\label{diffInt}
\end{equation}
where $\delta>0$ is a number that satisfies
\begin{equation}
\delta<d_0=\frac{1}{2}\min_{a\neq a^\prime\in A}\vert a-a^\prime\vert.
\label{delta<}
\end{equation}
$\mathscr{I}^{\epsilon,\delta}$ is intimately related to the fine structure of $u^\epsilon$. Indeed \eqref{delta<} implies that $\mathscr{I}^{\epsilon,\delta}$ separates the phases in the sense that  (see Figure \ref{Phases})
 \begin{equation}
 \begin{split}
 &\bar{\Omega}\setminus\mathscr{I}^{\epsilon,\delta}=\cup_{a\in A}\Omega_a^{\epsilon,\delta},\\
 &\Omega_a^{\epsilon,\delta}=\{x\in\Omega:\vert u^\epsilon(x)-a\vert\leq\delta\}\;\;a\in A,\\
 &\Omega_a^{\epsilon,\delta}\cap\Omega_{a^\prime}^{\epsilon,\delta}=\emptyset,\;\;a\neq a^\prime.
 \label{phases}
 \end{split}
 \end{equation}
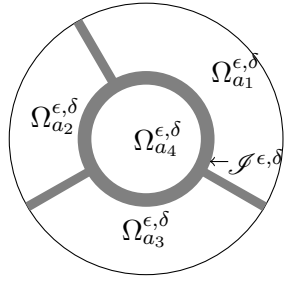
\begin{figure}
  \begin{center}
\begin{tikzpicture}[scale=.6]

\draw[] (0,0) circle [radius=3];;

\path[fill=gray] (-1.408,2.649) arc [radius=3, start angle=118, end angle= 122]
to (-.12,0)--(.12,0)--(-1.408,2.649);
\path[fill=gray] (-2.649,-1.408) arc [radius=3, start angle=208, end angle= 212]
to (.176,0)--(-.176,0)--(-2.649,-1.408);
\path[fill=gray] (2.649,-1.408) arc [radius=3, start angle=332, end angle= 328]
to (-.176,0)--(.176,0)--(2.649,-1.408);
\path[fill=gray] (0,0) circle [radius=1.5];;
\path[fill=white] (0,0) circle [radius=1.2];;
\node[right] at (1.2,1.5) {$\Omega_{a_1}^{\epsilon,\delta}$};
\node[right] at (-.5,0) {$\Omega_{a_4}^{\epsilon,\delta}$};
\node[] at (-2.,.5) {$\Omega_{a_2}^{\epsilon,\delta}$};
\node[] at (0,-2) {$\Omega_{a_3}^{\epsilon,\delta}$};
\draw[->] (1.8,-.5)--(1.4,-.5);
\node[right] at (1.55,-.5) {$\mathscr{I}^{\epsilon,\delta}$};
 \end{tikzpicture}
\end{center}
\caption{The diffuse interface $\mathscr{I}^{\epsilon,\delta}$ and the phases $\Omega_a$, $a\in A$.}
\label{Phases}
\end{figure}

 Assumption $h_1$ implies the existence of a constant $c_W>0$ such that, for $\delta\in(0,d_0]$ ,
 \begin{equation}
 \vert z\vert\leq M^\prime,\;\min_{a\in A}\vert z-a\vert\geq\delta\;\Rightarrow\;W(z)\geq\frac{1}{2}c_W^2\delta^2.
 \label{cW}
 \end{equation}
  This and the bound \eqref{UB0} yield
 \[\frac{1}{2\epsilon}c_W^2\delta^2\vert\mathscr{I}^{\epsilon,\delta}\vert\leq C\;\Rightarrow\;
 \vert\mathscr{I}^{\epsilon,\delta}\vert\leq C\frac{\epsilon}{\delta^2}.\]
 Hence, for the measure of $\mathscr{I}^{\epsilon,\delta}$, we have the estimate
 \begin{equation}
 \vert\mathscr{I}^{\epsilon,\delta}\vert\leq C\epsilon^{1-2\alpha},\;\;\text{for}\;\delta\geq\epsilon^\alpha,\;\alpha\in(0,\frac{1}{2}).
 \label{|I|}
 \end{equation}

 From \eqref{|I|} we have that, for $0<\epsilon<<1$, most of $\bar{\Omega}$ is occupied by the phases $\Omega_a^{\epsilon,\delta}$ separated by a tiny diffuse interface.

 In this paper we focus on the diffuse interface and present sufficient conditions on $\Omega$ and on the Dirichlet datum $v_0^\epsilon$ ensuring the connectivity of $\mathscr{I}^{\epsilon,\delta}$. Once we know that  $\mathscr{I}^{\epsilon,\delta}$ is connected we can ask if there exists an $n-1$ dimensional rectifiable current $\mathscr{G}^{\epsilon,\delta}\subset\mathscr{I}^{\epsilon,\delta}$ which partitions $\Omega$ in open sets corresponding to the phases. We study this question for $n=2$ and show that there exists a connected $C^1$-network $\mathscr{G}^{\epsilon,\delta}\subset\mathscr{I}^{\epsilon,\delta}$, actually infinitely many, with a precise structure which realizes the said partition.

 We introduce some notation and  give precise statements of our results.

 Consider the maps
 \begin{equation}
 \begin{split}
 &f^\epsilon:U^\epsilon\rightarrow(0,d_0),\;\;U^\epsilon=\{x\in\Omega:u^\epsilon(x)\in\cup_{a\in A}B_{d_0}(a)\},\\
 &f^\epsilon(x)=\min_{a\in A}\vert u^\epsilon-a\vert.\\
 &\text{and}\\
  &g^\epsilon:V^\epsilon\rightarrow(0,d_0),\;\;V^\epsilon=\{x\in\partial\Omega:v_0^\epsilon(x)\in\cup_{a\in A}B_{d_0}(a)\},\\
 &g^\epsilon(x)=\min_{a\in A}\vert u^\epsilon-a\vert.\\
 \end{split}
 \label{sard}
 \end{equation}
 From the definition \eqref{delta<} of $d_0$, for each $x\in U^\epsilon$ there is a unique $a\in A$ that
  realizes the minimum that defines $f^\epsilon$. It follows that $f^\epsilon$ has the same regularity
   of $u^\epsilon$. Analogous statements apply to $g^\epsilon$. We say that $\delta\in(0,d_0)$ is a regular
   value if $\delta$ is a regular value of both $f^\epsilon$ and $g^\epsilon$.

 Since $v_0^\epsilon$ is a given smooth map, $\delta$ being a regular value of $g^\epsilon$ should be
  regarded as an assumption on $v_0^\epsilon$. On the other hand $u^\epsilon\in C^k(\Omega;\R^m)$ implies
  $f^\epsilon\in C^k(U^\epsilon;\R)$. Hence, if $k\geq n$ and $\delta\in(0,d_0)$ is a regular value, by Sard lemma
  (see \cite{Dieu} pag.170, problem 2) and the implicit function theorem we can assume that
  $\partial\mathscr{I}^{\epsilon,\delta}\cap\Omega$ is a $C^2$ manifold and that
  $\bar{\Omega}\setminus\mathscr{I}^{\epsilon,\delta}$ has a finite number of connected components with mutual positive distance.
 Set
 \begin{equation}
 \Gamma_a^{\epsilon,\delta}=\{x\in\partial\Omega:\vert v_0^\epsilon(x)-a\vert\leq\delta\},\;\;a\in A.
 \label{Gammaa}
 \end{equation}
 Then we have
 \begin{theorem}
 \label{connect}
 There is $\delta_0>0$ such that if $h_1$ and $h_2$ hold and moreover
 \begin{enumerate}
 \item  $\partial\Omega$ is connected.
 \item There exists $\tilde{A}\subset A$ such that $\Gamma_a^{\epsilon,\delta}=\emptyset$ for $a\in A\setminus\tilde{A}$ and $a\in\tilde{A}$ implies
 \[\begin{split}
 &\;\;\,\Gamma_a^{\epsilon,\delta}\;\text{is an arc}\;\text{if}\;n=2,\\
 &
 \partial\Gamma_a^{\epsilon,\delta}\;\text{is connected},\;\text{if}\;n>2.
 \end{split}\]
 \item $u^\epsilon\in C^k(\Omega;\R^m)\cap C^1(\bar{\Omega};\R^M)$, $k\geq n$ and $\delta\in(0,\delta_0)$ is a regular value.
 \end{enumerate}
 Then, if $\{\omega_i^{\epsilon,\delta} \}_{i\in I}$ is a family of connected components of $\bar{\Omega}\setminus\mathscr{I}^{\epsilon,\delta}$,  $\bar{\Omega}\setminus\{\omega_i^{\epsilon,\delta} \}_{i\in I}$ is a connected set. In particular $\mathscr{I}^{\epsilon,\delta}$ and $\bar{\Omega}\setminus\Omega_a^{\epsilon,\delta}$, $a\in A$, are connected sets.
 \end{theorem}

Assumption (ii) in Theorem \ref{connect} implies that $\Omega_a^{\epsilon,\delta}\cap\partial\Omega$ has a unique connected component. This fact is essential for the validity of Theorem \ref{connect}. A situation where  $\Omega_a^{\epsilon,\delta}\cap\partial\Omega$ has two components and the diffuse interface in not connected is considered in \cite{AF+}.

An important and deep result on the structure of $\mathscr{I}^{\epsilon,\delta}$ is due to  to Sternberg and Zumbrun \cite{SZu}. They studied stable critical points of $J_\Omega^\epsilon$ subjected to a mass constraint

\begin{equation}
\frac{1}{\vert\Omega\vert}\int_\Omega vdx=\sum_{a\in A}\alpha_aa,\quad \alpha_a\geq 0,\;\sum_{a\in A}\alpha_a=1,
\label{mass}
\end{equation}
which expresses the fact that the average compositions of the material is fixed. For the scalar case $m=1$ and for potentials with two zeros $A=\{a_-,a_+\}$ in \cite{SZu} it is shown that, if $\Omega$ is strictly convex, then, for $\epsilon>0$ small and $\delta_\epsilon=\epsilon^k$ with $k>0$ sufficiently large, the diffuse interface is connected and separates the sets $\Omega_\pm=\{x\in\Omega:\vert u^\epsilon(x)-a_\pm\vert\leq\delta_\epsilon\}$. Here by a stable critical point we mean a critical point $u^\epsilon$ such that the second variation of $J_\Omega^\epsilon$ at $u^\epsilon$ is nonnegative:
\[d^2J_\Omega^\epsilon[u^\epsilon](\phi,\phi)
=\int_\Omega(\epsilon\vert\nabla\phi\vert^2+\frac{1}{\epsilon}W_{zz}(u^\epsilon)\phi\cdot\phi)dx\geq 0,\]
for all $\phi\in H^1(\Omega;\R^m)$ that satisfy \eqref{mass}.

The constraint \eqref{mass} strongly restricts the class of allowed perturbations of $u^\epsilon$. This makes the proof of the connectivity in  \cite{SZu} quite elaborate  and difficult and requires the strict convexity of $\Omega$. On the contrary assumption (i) in Theorem \ref{connect} allows for complex geometries. For example, for $n=3$, $\Omega$ can be a torus.

The relationship between the fine structure of minimizers and the diffuse interface for the constrained problem has been studied
by Baldo \cite{B} (see also \cite{M}, \cite{S}), who, using the theory of $\Gamma$-convergence, proved that if $u^\epsilon$, $\epsilon>0$, is a sequence of minimizers that converges in $L^1(\Omega;\R^m)$ to some map $u^0$, then $u^0=\sum_{a\in A}a\I_{S_a}$ where $S_{a_1},\ldots,S_{a_N}$ is a partition of $\Omega$ that minimizes the quantity $\sum_{a,a^\prime\in A}\sigma_{a,a^\prime}\partial^*S_a\cap\partial^*S_{a^\prime}$ among all partitions of $\Omega$ that satisfy $\sum_{a\in A}a\vert S_a\vert=\sum_{a\in A}\alpha_aa\vert\Omega\vert$. On the basis of this result one expects that, for $0<\epsilon<<1$, also $\mathscr{I}^{\epsilon,\delta}$  as the jump set of the limit map $u^0$, should enjoy some minimality property. The topology of $\Gamma$-convergence is too weak for establishing a rigorous link between the limit result of Baldo and the structure of $u^\epsilon$ for small $\epsilon>0$.

For the unconstrained problem with natural boundary conditions a trivial observation is that minimizers are necessarily constant and coincide with one of the zeros of $W$:
\[u^\epsilon\equiv a,\;\;\text{for some}\;a\in A.\]
This implies that to see some interesting  structure of critical points of $J_\Omega^\epsilon$ we must relax the requirements on $u^\epsilon$ and consider local  minimizers or stable critical points. Casten and Holland \cite{CH} have shown that, if $\Omega$ is convex, stable critical points are constant. On the other hand Matano \cite{Ma} has proved that, for a dumbbell shaped $\Omega$, with a sufficiently thin neck, there exist nonconstant local minimizers.

The influence of the shape of $\Omega$ on the structure of minimizers has been studied also for the constrained problem beginning with the work of Carr, Gurtin and Slemrod \cite{CGS} who showed that for the scalar problem and $A=\{-1,1\}$, if $\Omega=(0,1)$, then a minimizer of $J_{(0,1)}^\epsilon$ constrained by $\int_{(0,1)}vdx=m\in(-1,1)$ is monotone. This results has been extended to the case where $\Omega$ is a cylinder by Gurtin and Matano \cite{GM} who have considered the effect of various geometries of $\Omega$.

Next we assume $n=2$ and state our result on the existence of a network that separates the phases. We let $\tilde{N}=\sharp\tilde{A}$ be the cardinality of $\tilde{A}$. Assumption (i) in Theorem \ref{connect} and $n=2$ imply that $\partial\Omega$ is a Jordan curve. It follows
\begin{equation}
\partial\Omega\setminus\cup_{a\in\tilde{A}}\Gamma_a^{\epsilon,\delta}=\cup_{a\in\tilde{A}} I_a^{\epsilon,\delta},
\label{Bdry-int}
\end{equation}
where $I_a^{\epsilon,\delta}\subset\partial\Omega$, $a\in\tilde{A}$ is an open arc. We assume that $a$ is chosen so that the first extreme of $I_a^{\epsilon,\delta}$ coincides with the second extreme of $\Gamma_a^{\epsilon,\delta}$ with respect to a fixed orientation of $\partial\Omega$.
\begin{theorem}
\label{network}
Assume $n=2$ and let $W$, $v_0^\epsilon$ and $\delta$ be as in Theorem \ref{connect}. Also assume that $\tilde{N}\geq 2$ and that $\Omega_a^{\epsilon,\delta}\neq\emptyset$, $a\in A$. Then
\begin{enumerate}
\item There exists a connected network $\mathscr{G}^{\epsilon,\delta}\subset\mathscr{I}^{\epsilon,\delta}$ that separates the phases in the sense that
\begin{equation}
 \begin{split}
&\bar{\Omega}\setminus\mathscr{G}^{\epsilon,\delta}=\cup_{a\in A}S_a^{\epsilon,\delta},\\
&\Omega_a^{\epsilon,\delta}\subset S_a^{\epsilon,\delta},\;\;a\in A,
  \end{split}
  \label{dec}
 \end{equation}
  where $S_a^{\epsilon,\delta}$, $a\in A$ are relatively open simply connected subsets of $\bar{\Omega}$.
\item $\mathscr{G}^{\epsilon,\delta}=\cup_{i=1}^{n_s}\gamma_i([0,1])$ is the union of $n_s=3(N-1)-\tilde{N}$
rectifiable arcs, has $n_b=2(N-1)-\tilde{N}$
 branching points of triple junction type that lie in $\mathscr{I}^{\epsilon,\delta}\setminus\partial\Omega$
 and has $\tilde{N}$ end points, one for each interval $I_a^{\epsilon,\delta}$, $a\in\tilde{A}$.
\item Every arc $\gamma_i$ has no self intersections and two different arcs $\gamma_i$ and $\gamma_j$,
 $i\neq j$ can intersect only at their extremities which must then coincide with one of the $n_b$ triple points.
\item If $\gamma$ is one of the $n_s$ arcs of $\mathscr{G}^{\epsilon,\delta}$, then $\gamma$ either connects
 two distinct triple points or a triple point with an end point (if $N=\tilde{N}=2$ there
  is a unique $\gamma$ that connects two end points).
\item $\partial S_a^{\epsilon,\delta}$, $a\in A$, is a Jordan curve which (assuming counterclockwise orientation) satisfies
 \begin{equation}
 \begin{split}
 &\mathrm{index}(x,\partial S_a^{\epsilon,\delta})=1,\;\;x\in\Omega_a^{\epsilon,\delta},\\
 &\mathrm{index}(x,\partial S_a^{\epsilon,\delta})=0,\;\;x\in\Omega_{a^\prime}^{\epsilon,\delta},\;a^\prime\neq a.
 \end{split}
 \label{index}
 \end{equation}
  Moreover, for each $\gamma\in\mathscr{G}^{\epsilon,\delta}$, there
  are $a_\gamma\neq a_\gamma^\prime\in A$ such that
\begin{equation}
\gamma\subset\partial S_{a_\gamma}^{\epsilon,\delta}\cap\partial S_{a_\gamma^\prime}^{\epsilon,\delta}.
\label{gamma-in}
\end{equation}
\item If $\bar{\Omega}\setminus\mathscr{I}^{\epsilon,\delta}$ has $N^\prime$ connected components, then $\mathscr{I}^{\epsilon,\delta}$
is homotopically equivalent to a closed ball $\bar{B}$ deprived of $N^\prime$ points, $\tilde{N}$ of which
on $\partial B$. $N^\prime=N$ if $\Omega_a^{\epsilon,\delta}$, $a\in A$ is connected.
 \end{enumerate}
\end{theorem}

By the boundary conditions, $\Omega_a^{\epsilon,\delta}\neq\emptyset$, for $a\in\tilde{A}$, but there
 are example where, in spite of the boundary conditions, it results  $\Omega_a^{\epsilon,\delta}\neq\emptyset$
 for some $a\in A\setminus\tilde{A}$ (see \cite{f4} and Section \ref{app} below). A comment on the assumption $\tilde{N}\geq 2$ is in order.
  From the point of view of the complexity of $u^\epsilon$ the cases $\tilde{N}=0$ and $\tilde{N}=1$ can
   be regarded as trivial. Indeed, if $\tilde{N}=0,1$  and $0<\epsilon<<1$, aside from a tiny boundary layer
   of thickness $\mathrm{o}(1)$, $u^\epsilon$ is expected to be approximately equal to one of the zeros of $W$ in
   the whole of $\Omega$. For a results in this direction see Theorem 1.2 in \cite{AF+}.

In the following, to simplify the notation, if no confusion is possible, we omit to indicate the dependence
on $\epsilon$ and $\delta$ and write $\mathscr{I}, \Omega_a, \mathscr{G},\ldots$ instead of  $\mathscr{I}^{\epsilon,\delta}, \Omega_a^{\epsilon,\delta}, \mathscr{G}^{\epsilon,\delta},\ldots$
We refer to $\mathscr{G}=\mathscr{G}^{\epsilon,\delta}$ both as a subset of $\R^2$ and as the collection
 of $n_s$ well determined arcs $\gamma_i:[0,1]\rightarrow\R^2$.
 We also use different representation of the arc $\gamma_i$ with different admissible parameters
 in particular with the arclength. We remark that since $\mathscr{G}$ is the union of a finite number of rectifiable
arcs there is a constant $K>0$ such that
\begin{equation}
\vert\gamma\vert<K,\;\;\gamma\in\mathscr{G},
\label{B-length}
\end{equation}
where $\vert\gamma\vert$ denotes the length of $\gamma$.

We observe that there are infinitely many networks like $\mathscr{G}$ with the properties described in Theorem \ref{network}. From \eqref{gamma-in} it follows that to each pair $a\neq a^\prime\in A$
we can associate the number of the arcs $\gamma$ that satisfy  \eqref{gamma-in} with $a=a_\gamma$ and $a^\prime=a_\gamma^\prime$. This, since the total number of the arcs is equal to $n_s$, implies that the infinitely many networks
that exist can be classified in a finite number of types.

From \eqref{gamma-in} it follows that we can associate to $\mathscr{G}$ the quantity
 \begin{equation}
 \mathscr{F}(\mathscr{G})=\sum_{\gamma\in\mathscr{G}}\sigma_{a_\gamma a_\gamma^\prime}\vert\gamma\vert.
 \label{effe}
 \end{equation}
 This suggests the minimization of $\mathscr{F}$ as a criteria for an optimal choice of a network in the large class of networks with the properties described in Theorem \ref{network}. We address this question in Section \ref{optimal} (see Proposition \ref{opt-G}).

 For later reference we let $\mathcal{G}$ the set of networks which have the properties described in Theorem \ref{network} and satisfies the bound \eqref{B-length}. We let  $\mathcal{G}_f\supset\mathcal{G}$ be the set of network $\mathscr{G}\subset\bar{\Omega}$ which satisfy \eqref{B-length} and are defined as in Theorem \ref{network} by removing the constraint $\mathscr{G}\subset\mathscr{I}$. More precisely $\mathcal{G}_f$ is the set of networks $\mathscr{G}\subset\bar{\Omega}$ that satisfy (ii)-(v) as in Theorem \ref{network} while (i) and (vi) are replaced by

 (i)$_f$ $\bar{\Omega}\setminus\mathscr{G}$ is the union of $N$ pairwise disjoint simply connected relatively open subsets $S_a$, $a\in A$, of $\bar{\Omega}$. If $a\in\tilde{A}$, then $\partial S_a\cap\partial\Omega$ is an arc and
 \[\Gamma_a\subset\partial S_a\cap\partial\Omega,\]

 (vi)$_f$ $\bar{\Omega}\setminus\mathscr{G}$ is homotopically equivalent to a closed ball $\bar{B}$ deprived of $N$ points $\tilde{N}$ of which on $\partial B$.
 
 In Proposition \ref{opt-G} we prove the existence of an optimal network $\hat{\mathscr{G}}$ which minimizes $\mathscr{F}(\mathscr{G})$:
 \[\mathscr{F}(\hat{\mathscr{G}})=\inf_{\mathscr{G}\in\mathcal{G}}\mathscr{F}(\mathscr{G}).\]
 By definition of $\mathcal{G}$ the diffuse interface $\mathscr{I}$ is tightly attached to $\hat{\mathscr{G}}$ which can be regarded as a kind of \emph{spine} which captures the essential features of $\mathscr{I}$ conveying a lot of information on it and, in turn, on the fine structure of minimizers of problem \eqref{min}. $\hat{\mathscr{G}}$ is a one-dimensional geometric object easier to describe than the diffuse interface itself.
  One may conjecture that
  \begin{equation}
  \hat{\mathscr{G}}\approx\mathscr{G}_f,
  \label{ConJ}
  \end{equation}
  where $\mathscr{G}_f$ is a free optimal network, that is a minimizer of $\mathscr{F}(\mathscr{G})$ in the large class $\mathcal{G}_f$ of networks which are not constrained to the diffuse interface:
  \[\mathscr{F}(\mathscr{G}_f)=\inf_{\mathscr{G}\in\mathcal{G}_f}\mathscr{F}(\mathscr{G}).\]
  Obviously establishing that $\hat{\mathscr{G}}$ is actually a perturbation of $\mathscr{G}_f$ as suggested in \eqref{ConJ} would reduce, in some sense, the difficult problem of characterizing the structure of minimizers of \eqref{min} to the much easer geometric problem of determining $\mathscr{G}_f$.
  We show in Theorem \ref{fine} that this approach for the description of the fine structure of minimizers of problem \eqref{min} can actually be used provided we assume:
  \begin{description}
\item[$h_3$] The phases $\Omega_a$, $a\in A$ are connected.
\end{description}

The paper is organized as follows:  in Section \ref{2} we prove Theorem \ref{connect}, Theorem \ref{network} in Section \ref{3}. In Section \ref{optimal} we discuss the existence of the optimal network. In Section \ref{Fine}
we prove Theorem \ref{fine}. Section \ref{app} contains some basic lemmas. 

 \section{Proof of Theorem \ref{connect}.}
 \label{2}
 In this proof we drop the superscripts $\epsilon$ and $\delta$ and write simply $u$, $\mathscr{I}$, $\Omega_a\ldots$ instead of $u^\epsilon$, $\mathscr{I}^{\epsilon,\delta}$, $\Omega_a^{\epsilon,\delta}\ldots$.
  We let $E^\prime=\R^n\setminus E$ be the complement of a set $E\subset\R^n$.

  From $h_1$ and in particular from the assumption that $W_{zz}(a)$ is positive definite it follows that there is
  $\delta_0>0$ such that
  \begin{equation}
  \frac{d}{ds}W(a+s\nu)>0,\;\;s\in(0,2\delta_0),\;\nu\in\SF,\;a\in A.
  \label{increas}
  \end{equation}
  \begin{lemma}
   \label{do-conn}
 Assume the same as in Theorem \ref{connect} and moreover that $\delta_0$ is the number defined in \eqref{increas}. Let $\omega\subset\bar{\Omega}\setminus\mathscr{I}$ be a connected component of $\bar{\Omega}\setminus\mathscr{I}$. Then
  \[\partial\omega\quad\quad\text{is connected}.\]
 \end{lemma}
 \begin{proof}
 1. Since $\Omega$ is an open domain and, by assumption $\partial\Omega$ is connected, the complement $\Omega^\prime=\R^n\setminus\Omega$ of $\Omega$ is connected and therefore coincides with its unique
 unbounded connected component. This follows from a theorem of A. Czarnecki \emph{et altri} (see \cite{KL})
 which states: The boundary $\partial E$ of and open domain $E$ is connected if and only if its complement
  $E^\prime$ is connected.

  2. The interior $\omega^\circ$ of $\omega$ is an open domain contained in $\Omega$. It follows: $\Omega^\prime\subset{\omega^\circ}^\prime$. Actually, being unbounded, $\Omega^\prime$ is a subset of the unique unbounded component of ${\omega^\circ}^\prime$. Therefore, if $\tilde{{\omega^\circ}^\prime}\subset{\omega^\circ}^\prime$ is a bounded nonempty connected component of ${\omega^\circ}^\prime$, then it result
  \[\tilde{{\omega^\circ}^\prime}\subset\bar{\Omega}.\]

3.
 Since $\omega\subset\bar{\Omega}\setminus\mathscr{I}$ there exists $a\in A$ such that $\omega\subset\Omega_a$. Hence
 \begin{equation}
 \vert u(x)-a\vert\leq\delta,\;\;x\in\omega.
 \label{<a}
 \end{equation}

This and
 \begin{equation}
 \partial\tilde{{\omega^\circ}^\prime}\subset\partial{\omega^\circ}^\prime=\partial\omega^\circ,
 \label{dodoprime}
 \end{equation}

 yield

  \begin{equation}
 \vert u(x)-a\vert\leq\delta,\;\;x\in\partial\tilde{{\omega^\circ}^\prime}.
 \label{<atilde}
 \end{equation}

 From this \eqref{increas} it follows that we can invoke the maximum principle for minimizers (see Theorem 4.1 in \cite{afs}) and conclude
  \begin{equation}
 \vert u(x)-a\vert<\delta,\;\;x\in\tilde{{\omega^\circ}^\prime}.
 \label{<atilde1}
 \end{equation}

 This together with \eqref{dodoprime} is in contrast with the assumption that $\omega$ is a connected component of $\bar{\Omega}\setminus\mathscr{I}$ and we conclude that ${\omega^\circ}^\prime$ coincides with its unbounded component and the lemma follows from \cite{KL}. The proof is complete.
\end{proof}

 The assumption that $\delta>0$ is a regular value and, in particular, that $\delta$ is a regular value of $g^\epsilon$ implies that $M=\partial\mathscr{I}\cap\Omega$ is a $C^2$ manifold with boundary $\partial{M}=\cup_{a\in\tilde{A}}\partial\Gamma_a$ with transversal intersections with $\partial\Omega$:
  \begin{equation}
  \vert\mu_x\cdot\nu_x\vert<1,\; \text{for}\;x\in\partial M,
  \label{trans}
  \end{equation}
  where $\nu_x$ is  the unit interior normal to $\partial\Omega$ at $x$ and $\mu_x$ is the unit vector which is orthogonal to $M$ at $x$ and points toward the interior of $\mathscr{I}$.

 The smoothness of $\Omega$ implies the existence of $t_0>0$ such that, for each $t\in(-t_0,t_0)$, the map $p_t$ defined by
  \[\partial\Omega\ni x\stackrel{p_t}{\rightarrow} x+t\nu_x,\]
  is a diffeomorphism of $\partial\Omega$ onto its image $\partial\Omega_t=p_t(\partial\Omega)$.
  Similarly, provided $t_0>0$ is sufficiently small, for each $t\in(-t_0,t_0)$, the map $q_t(x)=x+t\mu_x$, $x\in M$, is a diffeomorphism of $M$ onto its image
  $M_t$.

 We are now in the position of completing the proof of Theorem \ref{connect}. Let $\{\omega_i\}_{i\in I}$ a family of connected components of $\bar{\Omega}\setminus\mathscr{I}$. By assumption (i) $\Omega$ is connected. Therefore, given $x_0\neq x_1\in\Omega\setminus\cup_{i\in I}\omega_i$, there is an arc $\gamma:[0,1]\rightarrow\Omega$ which connects $x_0$ to $x_1$ and there is $\tau>0$ such that
 \begin{equation}
 d(\gamma([0,1]),\partial\Omega)\geq\tau.
 \label{fin-dist}
 \end{equation}

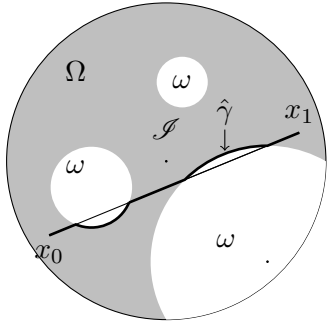
\begin{figure}
  \begin{center}
 \begin{tikzpicture}[scale=.7]
\draw[fill=lightgray] (0,0) circle [radius=3];;
\path[fill=white] (.3,1.5) circle [radius=.48];;
\path[fill=white] (3,0)  arc [radius=2.1961, start angle=60, end angle= 210]
 arc [radius=3, start angle=-90, end angle= 0] to (3,0);
\path[fill=black] (0,0) circle [radius=.025];;
\path[fill=black] (1.902,-1.902) circle [radius=.025];;
\draw[] (.349,-.349)--(1.902,.2941);

\path[fill=white] (-1.7,-1.19695)  arc [radius=.7653, start angle=247.494, end angle= 337.494]
  arc [radius=.7653, start angle=-22.506, end angle= 247.494] to (-1.7,-1.19695);
\draw[] (-2.2,-1.404)--(2.5,.542);
\draw[black,line width=1] (-2.2,-1.404)--(-1.7,-1.19695) arc [radius=.7653, start angle=247.494, end angle= 337.494](-.7,-.78285) to  (.349,-.349)
 arc [radius=2.1961, start angle=135, end angle= 90] to (2.5,.542);
 \node[] at (-1.7,1.7) {$\Omega$};
  \node[below] at (-2.2,-1.404) {$x_0$};
   \node[above] at (2.5,.542) {$x_1$};
  \node[below] at (0,1) {$\mathscr{I}$};
  \node[below] at (-1.7,.2) {$\omega$};
   \node[left] at (1.5,-1.5) {$\omega$};
   \node[] at (.3,1.5)  {$\omega$};
\draw [->] (1.1,.6)--(1.1,.2);
  \node[above] at  (1.1,.5)  {$\hat{\gamma}$};
 \end{tikzpicture}
 \end{center}
\caption{The idea of the proof of Theorem \ref{connect}.}
\label{idea}
\end{figure}

  We show that the connectivity of $\partial\omega$ established in Lemma \ref{do-conn} implies the existence of an arc $\tilde{\gamma}:[0,1]\rightarrow\bar{\Omega}\setminus\cup_{i\in I}\omega_i$ that connects $x_0$ to $x_1$ (see Figure \ref{idea} for an idea of the argument). We can assume that $\gamma((0,1))\cap\cup_{i\in I}\omega_i\neq\emptyset$. Otherwise we can identify $\tilde{\gamma}$ with $\gamma$. Then there exist: $k\geq 1$ and sequences $\{\omega_{i_1},\ldots,\omega_{i_k}\}\subset\{\omega_i\}_{i\in I}$,
$s_1\leq s_1^\prime<s_2\leq s_2^\prime<\ldots<s_k\leq s_k^\prime<1$ such that
\begin{equation}
\begin{split}
&\gamma([0,s_1))\cap\cup_{i\in I}\omega_i=\emptyset,\\
&\gamma(s_1),\gamma(s_1^\prime)\in\partial\omega_{i_1},\\
&\gamma((s_1^\prime,1])\cap\omega_{i_1}=\emptyset,\\\\
&\gamma((s_1^\prime,s_2))\cap\cup_{i\in I}\omega_i=\emptyset,\\
&\gamma(s_2),\gamma(s_2^\prime)\in\partial\omega_{i_2},\\
&\gamma((s_2^\prime,1])\cap\omega_{i_2}=\emptyset,\\
&\cdot\\
&\cdot\\
&\gamma((s_{k-1}^\prime,s_k))\cap\cup_{i\in I}\omega_i=\emptyset,\\
&\gamma(s_k),\gamma(s_k^\prime)\in\partial\omega_{i_k},\\
&\gamma((s_k^\prime,1])\cap\cup_{i\in I}\omega_i=\emptyset.
\end{split}
\label{seqs}
\end{equation}
From \eqref{fin-dist} it follows that, if $t>0$ is sufficiently small, there exist $0<s_1^-<s_1\leq s_1^\prime<s_1^+<s_2$  such that
\[\gamma(s_1^\pm)\in m_{1,t},\]
where $m_{1,t}\subset M_t$ is the diffeomorphic image of $m_1=\partial\omega_{i_1}\cap\Omega$ under the map $q_t$ introduced above. Set $x^\pm=q_t^{-1}(\gamma(s_1^\pm))$ where $q_t^{-1}$ is the invese of $q_t$. The points $x^\pm$ belong to $m_1\subset\partial\omega_{i_1}$. Since, by Lemma \ref{do-conn}, $\partial\omega_{i_1}$ is connected there is and arc $\eta:[s_1^-,s_1^+]\rightarrow\partial\omega_{i_1}$ that connectes $x^-$ to $x^+$
see Figure \ref{om-dO}.

Consider first the case where $\omega_{i_1}\cap\partial\Omega=\emptyset$. In this case we have $m_1=\partial\omega_{i_1}$ and therefore $\eta([s_1^-,s_1^+])\subset m_1$ and the arc $q_t(\eta([s_1^-,s_1^+]))$
connects $\gamma(s_1^-)$ to $\gamma(s_1^+)$ and, by \eqref{seqs} and the assumed smallness of $t>0$, is fully contained in $\bar{\Omega}\setminus\cup_{i\in I}\omega_i$. It follows that the union of the arc $\gamma([0,s_1^-))$ and  the arc $q_t(\eta([s_1^-,s_1^+])$ is an arc that connects $x_0$ to $\gamma(s_1^+)$ and does not intersect any of the $\omega_i$, $i\in I$. Actually for the above construction it is not necessary that $m_1=\partial\omega_{i_1}$ but it sufficies that  $\eta([s_1^-,s_1^+])\subset m_1$. Therefore it remains to consider the case where
\begin{equation}
\eta([s_1^-,s_1^+])\cap\partial m_1\neq\emptyset.
\label{remain}
\end{equation}
This is possible only if $\omega_{i_1}\cap\partial\Omega=\Gamma_a$ for some $a\in\tilde{A}$ and $\partial m_1=\partial\Gamma_a$, see Figure \ref{om-dO}. Indeed $\Gamma_a\setminus\omega_{i_1}\cap\partial\Omega\neq\emptyset$
would contradict assumption (ii).

\begin{figure}
  \begin{center}
\begin{tikzpicture}[scale=.55]
\draw[fill=black,black] (2.16,1.30) circle [radius=.04];
\draw[fill=black,black] (2.16,-1.30) circle [radius=.04];
\draw[] (0,0) to [out=90,in=180] (3,2);
\draw[] (0,0) to [out=270,in=180] (3,-2);
\draw[] (2,0) to [out=90,in=180] (3,2);
\draw[] (2,0) to [out=270,in=180] (3,-2);
\draw[] (4,0) to [out=90,in=0] (3,2);
\draw[] (4,0) to [out=270,in=0] (3,-2);
\draw[blue,very thick] (2,0) to [out=91,in=251] (2.16,1.30);
\draw[blue,very thick] (2,0) to [out=269,in=109] (2.16,-1.30);
\draw[red,very thick] (1.75,1.30)-- (2.16,1.30);
\draw[red,very thick] (1.75,-1.30)-- (2.16,-1.30);
\draw[blue,very thick] (1.75,1.30) to [out=90,in=230] (2.25,2.5);
\draw[blue,very thick] (1.75,-1.30) to [out=270,in=30] (1.25,-2.5);
\draw[red]  (2.16,-1.30) to [out=60,in=300] (2.16,1.30);
\draw[fill=black,black] (2.16,1.30) circle [radius=.04];
\draw[fill=black,black]  (1.75,1.30) circle [radius=.04];
\draw[fill=black,black]  (1.75,-1.30) circle [radius=.04];
\draw[]  (4.3,0) to [out=90,in=300] (3,4.5);
\draw[]  (4.3,0) to [out=270,in=60] (3,-4.5);
\node[left]at (1.75,1.30){$x^-$};
\node[left]at (1.75,-1.30){$x^+$};
\node[]at(3.3,-.6) {$\Gamma_a$};
\node[]at(1,.3) {$\omega_{i_1}$};
\node[] at (2.26,2.60){$\gamma(s_1^-)$};
\node[] at (1.26,-2.6){$\gamma(s_1^+)$};
\node[] at (1.5,3.5){$\Omega$};
\node[] at (2.6,0){$\eta$};
\node[] at (-.95,0) {$\partial\omega_{i_1}$};
\draw[->](-.6,0)--(0,0);
\end{tikzpicture}
\end{center}
\caption{The points $x_1^\pm=q_t^{-1}(\gamma(s_1^\pm)$ and the arc $\eta:[s_1^-,s_1^+]\rightarrow\partial\omega_{i_1}$. }
\label{om-dO}
\end{figure}
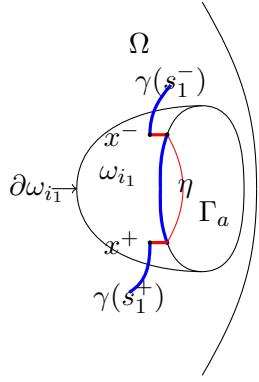

If $n=2$, the connectivity of $\partial\omega_{i_1}$ stated in Lemma \ref{do-conn} implies that $\partial\omega_{i_1}$ is a Jordan curve and, by assumption (ii) $\Gamma_a$ is an arc. It follows that the points $x_1^-$ and $x_1^+$ divide  $\partial\omega_{i_1}$ in two arcs and that one of them contains $\Gamma_a$ while the other is fully contained in $\Omega$. This reduces the analysis to the case discussed above where $\eta([s_1^-,s_1^+])\cap\partial m_1=\emptyset$.
Therefore we assume \eqref{remain} and $n>2$ and observe that, given $\theta\in(0,t_0)$, if $t\in(0,t_0)$ is sufficiently small, $m_{1,t}\cap\partial\Omega_\theta$ is a $n-2$ dimensional manifold  which is diffeomorphic to $\partial\Gamma_a$ and is contained in $\Omega\setminus\cup_{i\in I}\omega_i$  
From \eqref{remain} it follows that, if $t>0$ is sufficiently small, there exist $s_1^-<\tau_1^-\leq\tau_1^+<s_1^+$ such that
\begin{equation}
\begin{split}
&q_t^{-1}(\eta(s))\not\in\partial\Omega_\theta,\;\;s\in[s_1^-,\tau_1^-),\\
&q_t^{-1}(\eta(\tau_1^-))\in m_{1,t}\cap\Omega_\theta,\\
&q_t^{-1}(\eta(\tau_1^+))\in m_{1,t}\cap\Omega_\theta,\\
&q_t^{-1}(\eta(s))\not\in\partial\Omega_\theta,\;\;s\in(\tau_1^+,s_1^+].
\end{split}
\label{taupm}
\end{equation}
From (ii) $\partial\Gamma_a=\partial m_1$ is connected and so is  $m_{1,t}\cap\partial\Omega_\theta$ which is diffeomorphic to  $\partial\Gamma_a$. Therefore there exists an arc $\zeta([\tau_1^-,\tau_1^+])$ which is contained in $m_{1,t}\cap\partial\Omega_\theta$ and connects $q_t^{-1}(\eta(\tau_1^-))$ to $q_t^{-1}(\eta(\tau_1^+))$. From this and \eqref{seqs} we conclude that the arc sum of $\gamma([0,s_1^-))$, $q_t^{-1}(\eta([s_1^-,\tau_1^-))$,  $\zeta([\tau_1^-,\tau_1^+))$ and $q_t^{-1}([\tau_1^+,s_1^+])$ is an arc that connect $\gamma(s_1^-)$ to $\gamma(s_1^+)$ and is contained in $\Omega\setminus\cup_{i\in I}\omega_i$.

By iterating the argument developed above for $\omega_{i_1}$ we successively introduce $s_2^+,\ldots,s_k^+=1$ and corresponding arcs that connect $x_0$ to $\gamma(s_2^+)\ldots$, to $\gamma(1)=x_1$ and are contained in $\Omega\setminus\cup_{i\in I}\omega_i$. The last arc that connects $x_0$ to $x_1$ can be defined to be the arc $\tilde{\gamma}$. The proof of Theorem \ref{connect} is complete.

 \section{Proof of Theorem \ref{network}.}
 \label{3}
 As in the  proof of Theorem \ref{connect} we drop the superscripts $\epsilon$ and $\delta$ and write simply $u$, $\mathscr{I}$, $\mathscr{G}$, $\Omega_a\ldots$ instead of $u^\epsilon$, $\mathscr{I}^{\epsilon,\delta}$, $\mathscr{G}^{\epsilon,\delta}$ $\Omega_a^{\epsilon,\delta}\ldots$.

 From the assumption that $\delta$ is a regular value, the implicit function theorem and the fact that we are in two dimension: $\Omega\subset\R^2$, it follows that the boundary of each connected component $\omega$ of $\Omega\setminus\mathscr{I}$ is a curve which is of class $C^2$ if $\omega\cap\partial\Omega=\emptyset$ or, if $\omega\cap\partial\Omega\neq\emptyset$ is the union of a $C^2$ curve $\partial\omega\cap\Omega$ with one of the arcs  $\Gamma_{a}$, $a\in\tilde{A}$. This and Lemma
  \ref{do-conn} imply that $\partial\omega$ is a Jordan curve  and therefore that $\omega$ is homeomorphic to a closed ball (see \cite{stoker} pp. 43-45). Since by assumption $\partial\Omega$ is connected the same is true for $\Omega$. It follows that $\mathscr{I}$ is homotopically equivalent to a closed ball $\bar{B}$ deprived of $N^\prime$ points $\tilde{N}$ of which on the boundary $\partial B$ of $B$, here $N^\prime\geq N$ is the number of connect components of $\bar{\Omega}\setminus\mathscr{I}$.

  Assume that $\omega_1$ and $\omega_2$ are two distinct connected component of $\Omega_a$, for some $a\in A$. Since $\mathscr{I}$ is connected and $\mathscr{I}\setminus\partial\Omega$ is open there is an arc $\eta:[0,1]\rightarrow\mathscr{I}$ that connects $\omega_1$ to $\omega_2$ and satisfies $\eta(0)\in\omega_1$, $\eta(1)\in\omega_2$ and $\eta((0,1))\subset\mathscr{I}\setminus\partial\Omega$. Actually we can assume that, for $0<r$ sufficiently small, $\eta^r=\cup_{x\in\eta([0,1])}B_r(x)$ is a little channel that connects $\omega_1$ and $\omega_2$ and such that $\eta^r\setminus(\omega_1\cup\omega_2)
\subset\mathscr{I}\setminus\partial\Omega$. The boundary of the set $\omega_1\cup\eta^r\cup\omega_2$ is a Jordan curve. This follows from the construction of $\eta^r$ and from fact that $\omega_i$, $i=1,2$ are homeomorphic to closed balls. It follows that also $\omega_1\cup\eta^r\cup\omega_2$ is homeomorphic to a closed ball. Hence $\mathscr{I}\setminus\eta^r$ is homotopically equivalent to a ball deprived of $N^\prime-1$ points. The process can be iterated and after a finite number of steps we end up with a set of the form $\hat{\mathscr{I}}=\bar{\Omega}\setminus\cup_{a\in A}\hat{\Omega}_a$
 where $\hat{\Omega}_a$, $a\in A$ are non overlapping closed sets that satisfy $\Omega_a\subset\hat{\Omega}_a$, $a\in A$ and are homeomorphic to closed ball.
 \vskip.2cm
\noindent To prove the existence of  $\mathscr{G}$ we use an induction argument. We
 assume:
\vskip.2cm
\noindent (i) there are $\nu,\tilde{\nu}\in\N$ such that
\begin{equation}
2\leq\tilde{\nu}\leq\tilde{N},\quad\quad\tilde{\nu}\leq\nu<N,
\label{IndHyp}
\end{equation}
and a connected network $G_{\nu,\tilde{\nu}}\subset\hat{\mathscr{I}}$ composed of $\nu_s=3(\nu-1)-\tilde{\nu}$ arcs that connect $\nu_b=2(\nu-1)-\tilde{\nu}$ branching points of triple junction type and $\tilde{\nu}$ end points in the intervals $I_a$, $a\in A$ in such a way that each of the $I_a$ contains at most one of the $\tilde{\nu}$ end points.
\vskip.2cm
\noindent(ii)  $\bar{\Omega}\setminus G_{\nu,\tilde{\nu}}=\cup_{j=1}^\nu S_j^{\nu,\tilde{\nu}}$, where $S_j^{\nu,\tilde{\nu}}$, $j=1,\ldots,\nu$ is a relatively open simply connected set such that
\[\text{either}\quad\hat{\Omega}_a\subset S_j^{\nu,\tilde{\nu}}\quad\text{or}\quad\hat{\Omega}_a\cap  S_j^{\nu,\tilde{\nu}}=\emptyset,\;\;a\in A,\;j=1,\ldots,\nu.\]
\vskip.2cm
We show that, if $\nu<N$ this induction hypothesis implies the existence of a network $G_{\nu+1,\tilde{\nu}}$ or $G_{\nu+1,\tilde{\nu}+1}$ which satisfies (i) and (ii) above for $\nu+1,\tilde{\nu}$ or
for $\nu+1,\tilde{\nu}+1$.

From (ii) we have that $S_j^{\nu,\tilde{\nu}}\setminus\cup_{a\in A}\hat{\Omega}_a$, $j=1,\ldots,\nu$ is homotopically equivalent to a closed ball deprived of a certain number of points some of which may be on the boundary. This is the case if $\partial S_j^{\nu,\tilde{\nu}}$ contains some of the arcs $\Gamma_a$. If $\nu<N$ one of the set  $S_j^{\nu,\tilde{\nu}}$ contains at least two of the sets $\hat{\Omega}_a$.

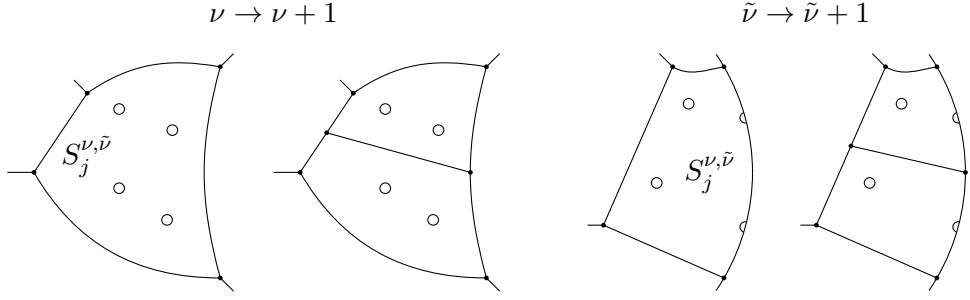
\begin{figure}
  \begin{center}
 \begin{tikzpicture}[scale=.7]
\draw [] (-8.9,-.3) circle [radius=.1];;
\draw [] (-8.,-.9) circle [radius=.1];;
\draw [] (-7.9,.8) circle [radius=.1];;
\draw [] (-8.9,1.2) circle [radius=.1];;
\draw [] (-3.9,-.3) circle [radius=.1];;
\draw [] (-3.,-.9) circle [radius=.1];;
\draw [] (-2.9,.8) circle [radius=.1];;
\draw [] (-3.9,1.2) circle [radius=.1];;
\draw [] (-10.5,0)--(-9.5,1.5);
\draw[] (-7,2) to [out=255,in=105] (-7,-2);
\draw[] (-9.5,1.5) to [out=35,in=170] (-7,2);
\draw[] (-7,-2) to [out=180,in=300]  (-10.5,0);
\draw [] (-10.5,0)-- (-11,0);
\draw [] (-7,2)-- (-6.75,2.25);
\draw [] (-7,-2)-- (-6.75,-2.25);
\draw [] (-9.5,1.5)-- (-9.75,1.75);
\path [fill=black] (-7,2) circle [radius=.045];;
\path [fill=black] (-7,-2) circle [radius=.045];;
\path [fill=black] (-10.5,0) circle [radius=.045];;
\path [fill=black] (-9.5,1.5) circle [radius=.045];;
\path [fill=black] (-5,.75) circle [radius=.045];;
\path [fill=black] (-2.3,0) circle [radius=.045];;
\path [fill=black] (-2,2) circle [radius=.045];;
\path [fill=black] (-2,-2) circle [radius=.045];;
\path [fill=black] (-5.5,0) circle [radius=.045];;
\path [fill=black] (-4.5,1.5) circle [radius=.045];;
\draw [] (-5,.75)-- (-2.3,0);
\draw [] (-5.5,0)-- (-6,0);
\draw [] (-2,2)-- (-1.75,2.25);
\draw [] (-2,-2)-- (-1.75,-2.25);
\draw [] (-4.5,1.5)-- (-4.75,1.75);
\draw [] (-5.5,0)--(-4.5,1.5);
\draw[] (-2,2) to [out=255,in=105] (-2,-2);
\draw[] (-4.5,1.5) to [out=35,in=170] (-2,2);
\draw[] (-2,-2) to [out=180,in=300]  (-5.5,0);
\draw [] (3,0) arc [radius=4, start angle=0, end angle=30];
\draw [] (3,0) arc [radius=4, start angle=0, end angle=-30];
\draw [] (2.4641,2) to [out=190,in=330] (1.5,2);
\draw [] (2.4641,-2)--(.2,-1)-- (1.5,2);
\path [fill=black] (2.4641,2) circle [radius=.045];;
\path [fill=black] (2.4641,-2) circle [radius=.045];;
\path [fill=black] (.2,-1) circle [radius=.045];;
\path [fill=black] (1.5,2) circle [radius=.045];;
\draw [] (2.4641,2) arc [radius=4, start angle=30, end angle=34];
\draw [] (2.4641,-2) arc [radius=4, start angle=-30, end angle=-34];
\draw [] (.2,-1)--(-.1,-1);
\draw [] (1.5,2)--(1.25,2.25);
\draw [] (1.2,-.2) circle [radius=.1];;
\draw [] (1.8,1.3) circle [radius=.1];;
\draw [] (2.8377,1.1319) arc [radius=.1, start angle=105, end angle=285];
\draw [] (2.8897,-.9386) arc [radius=.1, start angle=75, end angle=255];
\draw [] (7,0) arc [radius=4, start angle=0, end angle=30];
\draw [] (7,0) arc [radius=4, start angle=0, end angle=-30];
\draw [] (6.4641,2) to [out=190,in=330] (5.5,2);
\draw [] (6.4641,-2)--(4.2,-1)-- (5.5,2);
\path [fill=black] (6.4641,2) circle [radius=.045];;
\path [fill=black] (6.4641,-2) circle [radius=.045];;
\path [fill=black] (4.2,-1) circle [radius=.045];;
\path [fill=black] (5.5,2) circle [radius=.045];;
\path [fill=black] (7,0) circle [radius=.045];;
\path [fill=black] (4.85,.5) circle [radius=.045];;
\draw [] (6.4641,2) arc [radius=4, start angle=30, end angle=34];
\draw [] (6.4641,-2) arc [radius=4, start angle=-30, end angle=-34];
\draw [] (4.2,-1)--(3.9,-1);
\draw [] (5.5,2)--(5.25,2.25);
\draw [] (4.85,.5)--(7,0);
\draw [] (6.8377,1.1319) arc [radius=.1, start angle=105, end angle=285];
\draw [] (6.8897,-.9386) arc [radius=.1, start angle=75, end angle=255];
\node[] at (-6,3) {$\nu\rightarrow\nu+1$};
\node[] at (4,3) {$\tilde{\nu}\rightarrow\tilde{\nu}+1$};
\node[] at (2.2,0) {$S_j^{\nu,\tilde{\nu}}$};
\node[] at (-9.5,.3) {$S_j^{\nu,\tilde{\nu}}$};
\draw [] (5.2,-.2)  circle [radius=.1];;
\draw [] (5.8,1.3) circle [radius=.1];;
\end{tikzpicture}
\end{center}
\caption{Examples of cases I) and II).}
\label{III}
\end{figure}

We consider two possibilties (see Figure \ref{III}): I) There exists $a\in A$ such that $\hat{\Omega}_a\subset(S_j^{\nu,\tilde{\nu}}\setminus\partial\Omega)$ and II) There exist $a_1, a_2\in A$ such that $S_j^{\nu,\tilde{\nu}}\cap\hat{\Omega}_a=\Gamma_a,\;\;a\in\{a_1,a_2\}$.

I) In this case the connectivity of $S_j^{\nu,\tilde{\nu}}\setminus\cup_{a\in A}\hat{\Omega}_a$ implies the existence of an arc $\eta:[0,1]\rightarrow S_j^{\nu,\tilde{\nu}}\setminus\cup_{a\in A}\hat{\Omega}_a$ with the extreme points on $\partial S_j^{\nu,\tilde{\nu}}\setminus\partial\Omega$ and such that the open set $S_j^\prime$, bounded by the Jordan curve union of $\eta([0,1])$ and of the subarc of $\partial S_j^{\nu,\tilde{\nu}}\setminus\partial\Omega$ determined by $\eta(0)$ and $\eta(1)$, contains at least one of the sets $\hat{\Omega}_a$. Moreover the points $\eta(0)$ and $\eta(1)$ can be chosen so that they do not coincides with any of the triple points belonging to $\partial S_j^{\nu,\tilde{\nu}}\setminus\partial\Omega$. It follows that the network $G_{\nu,\tilde{\nu}}\cup\eta([0,1])$ has $\nu_b+2=2((\nu+1)-1)-\tilde{\nu}$ triple junction branching points and $\nu_s+3=3((\nu+1)-1)-\tilde{\nu}$ arcs. The latter statement follows from the fact the $\eta(0)$ divides the arc to which belongs in two parts and the same is true with for $\eta(1)$. These considerations and the observation that the only difference between $\bar{\Omega}\setminus G_{\nu,\tilde{\nu}}$ and $\bar{\Omega}\setminus(G_{\nu,\tilde{\nu}}\cup\eta([0,1]))$ consists in the fact that $S_j^{\nu,\tilde{\nu}}$ is replaced by $S_j^\prime$ and $S_j^{\nu,\tilde{\nu}}\setminus\bar{S}_j^\prime$ show that we have construct a network
$G_{\nu+1,\tilde{\nu}}=G_{\nu,\tilde{\nu}}\cup\eta([0,1])$ which satisfies the induction hypothesis for $\nu+1,\tilde{\nu}$.

II) In this case we have necessarily $I_a\subset\partial S_j^{\nu,\tilde{\nu}}$ for some $a\in A$ and as a  consequence $S_j^{\nu,\tilde{\nu}}$ can be divided in two parts by an arc $\eta:[0,1]\rightarrow S_j^{\nu,\tilde{\nu}}\setminus\cup_{a\in A}\hat{\Omega}_a$ with the initial point $\eta(0)\in I_a$ and the end point $\eta(1)\in(\partial S_j^{\nu,\tilde{\nu}}\cap\Omega)$. Note that, since from \eqref{IndHyp} we have $\nu\geq 2$, $\partial S_j^{\nu,\tilde{\nu}}\cap\Omega$ is nonempty and $\eta(1)$ can be chosen in the interior of one of the arcs that compose $\partial S_j^{\nu,\tilde{\nu}}\cap\Omega$. The network $G_{\nu,\tilde{\nu}}\cup\eta([0,1])$ has $\tilde{\nu}+1$ end points, $n_b+1=2((\nu+1)-1)-(\tilde{\nu}+1)$ branching points and $n_s+2=3((\nu+1)-1)-(\tilde{\nu}+1)$ arcs. It follows that $G_{\nu+1,\tilde{\nu}+1}=G_{\nu,\tilde{\nu}}\cup\eta([0,1])$ is a network that satisfies the induction hypothesis for $\nu+1,\tilde{\nu}+1$.
\vskip.2cm
To complete the proof it remains to show the existence of a network $G_{\nu,\tilde{\nu}}\subset\hat{\mathscr{I}}$ with the properties described above for $\tilde{\nu}=\nu=2$. With this choice of $\nu$ and $\tilde{\nu}$ we have $\nu_b=0$ and $\nu_s=1$. It follows that $G_{2,2}$ must consist in a unique arc $\eta:[0,1]\rightarrow\hat{\mathscr{I}}$ with end points in two distinct interval $I_a$. The existence of such an arc follows from the assumption $\tilde{N}\geq 2$ and the connectivity of $\hat{\mathscr{I}}$. The proof is complete.

\section{The existence of the optimal network.}
\label{optimal}
In this section we consider the problem, posed at the end of the Introduction, of the existence of an optimal network in the set of networks with the properties listed in Theorem \ref{network}.

We let $\mathcal{G}=\mathcal{G}^{\epsilon,\delta}$ the class of the networks that
satisfies the properties listed in Theorem \ref{network} and the bound \eqref{B-length}.
 $\mathcal{G}$ with the distance
 \begin{equation}
 d(\mathscr{G},\mathscr{G}^\prime)=\sum_{i=1}^{n_s}\|\gamma_i-\gamma_i^\prime\|_{C^0([0,1];\R^2)},
 \label{metric}
 \end{equation}
is a metric space and the closure $(\bar{\mathcal{G}},d)$ of $(\mathcal{G},d)$
is a compact space. This follows from the fact that
 \eqref{B-length} implies the existence of a parametrization of $\gamma\in\mathscr{G}$ that makes the
 map $\gamma:[0,1]\rightarrow\R^2$ a lipschitz map with lipschitz constant $K$. Therefore,
  if $\{\mathscr{G}_j\}_{j=1}^\infty$ is a sequence in $\mathcal{G}$,
  then there is a subsequence  $\{\mathscr{G}_{j_k}\}_{k=1}^\infty$ that converges to some $\mathscr{G}
  \in\bar{\mathcal{G}}$. We denote $S_{a,k}$ the sets associated to $\mathscr{G}_{j_k}$, $k=1,\ldots$ by \eqref{dec}.

  $\mathscr{G}$ as the limit of the sequence  $\{\mathscr{G}_{j_k}\}_{k=1}^\infty$ has the following properties:
  \begin{enumerate}
  \item $\mathscr{G}\subset\bar{\mathscr{I}}$ is a connected network that separate
   the phases in the sense that
   \begin{equation}
   \begin{split}
 &\bar{\Omega}\setminus\mathscr{G}=\cup_{a\in A}S_a,\\
 &\Omega_a^\circ\subset S_a,\;\,a\in A,
 \end{split}
 \label{G0-sep}
   \end{equation}
 where $S_a$, $a\in A$, are relatively open subsets of $\bar{\Omega}$ and $E^\circ$ is the interior of $E$.
\item $\mathscr{G}=\cup_{i=1}^{n_s}\gamma_i([0,1])$ is the union of $n_s=3(N-1)-\tilde{N}$ rectifiable
 arcs and has $n_b=2(N-1)-\tilde{N}$ branching points of triple junction
 type and $\tilde{N}$ end
 points, one for each interval $\bar{I}_a$, $a\in A$.

  Some of the $\gamma_i:[0,1]\rightarrow\bar{\mathscr{I}}$ may reduce to constant maps. The branching
  points may lie on $\partial\mathscr{I}$ and some of them may coincide.

 \item Each arc $\gamma:[0,1]\rightarrow\R^2$ in $\mathscr{G}$ can be parametrized as a lipschitz map with
 lipschitz constant $K$, $K$ the constant in \eqref{B-length}.

  \item Every arc $\gamma_i$ can be tangent to itself without self intersections. Two different arcs $\gamma_i$
   and  $\gamma_j$, $i\neq j$ can be tangent to each other without intersections.

 \item Point (v) in Theorem \ref{network} passes in the limit in the form: the arcs $\gamma\in\mathscr{G}$
  defined
 by
 \[\gamma=\lim_{k\rightarrow\infty}\gamma_k,\;\;\gamma_k\subset\partial S_{a,k},\]
 and the arc $\partial\Omega\cap\partial S_a$
 determine a circuit $\mathscr{C}_a$, $a\in A$ such that
  \begin{equation}
 \begin{split}
 &\mathrm{index}(x,\mathscr{C}_a)=1,\;\;x\in\Omega_a^\circ,\\
 &\mathrm{index}(x,\mathscr{C}_a)=0,\;\;x\in\Omega_{a^\prime}^\circ,\;a^\prime\neq a.
 \end{split}
 \label{index-}
 \end{equation}
 For each $\gamma\in\mathscr{G}$ which does not reduce to a constant there are $a_\gamma\neq a_\gamma^\prime\in A$
  such that
 \begin{equation}
 \gamma([0,1])\subset\mathscr{C}_{a_\gamma}\cap\mathscr{C}_{a_\gamma^\prime}.
 \label{gamma-in-}
 \end{equation}
  \end{enumerate}

  Note that we don't claim that $S_a$, $a\in A$ is connected.
  Indeed we can not a priori exclude that, unless we assume the connectivity of $\Omega_a^\circ$, $a\in A$, some  $S_a$
  is the union of two of more connected components. For $\mathscr{G}\in\partial\mathcal{G}$,
  $\mathscr{C}_a=\lim_{k\rightarrow\infty}\partial S_{a,k}$ may not coincide with $\partial S_a$
  but we have the inclusion
  \[\partial S_a\subset\mathscr{C}_a.\]
    Note also that in \eqref{index-} we have restricted
  $x$ to the interior of $\Omega_a$, $a\in A$.

  From \eqref{gamma-in-} it follows that $\mathscr{F}(\mathscr{G})$ is well
  defined also for $\mathscr{G}\in\partial\mathcal{G}$ and we can write
  \[\mathscr{F}(\mathscr{G})
  =\sum_{\gamma\in\mathscr{G}}\sigma_{a_\gamma a_\gamma^\prime}\vert\gamma\vert,\;\;\mathscr{G}\in\bar{\mathcal{G}}.\]

  \begin{proposition}
  \label{opt-G}
  There exists a minimizer $\hat{\mathscr{G}}\in\bar{\mathcal{G}}$ of the problem
  \[\mathscr{F}(\hat{\mathscr{G}})=\min_{\mathscr{G}\in\bar{\mathcal{G}}}\mathscr{F}(\mathscr{G}).\]
  Moreover $\gamma\in\hat{\mathscr{G}}$ is piecewise $C^2$.
  \end{proposition}
  \begin{proof}
 1. The functional $\bar{\mathcal{G}}\ni\mathscr{G}\rightarrow\mathscr{F}(\mathscr{G})\in\R$ is
   lower semicontinuous with respect to the distance \eqref{metric}.

   The argument is standard: for each
    $\mathscr{G}\in\bar{\mathcal{G}}$, for each $\gamma\in\mathscr{G}$ and for each $n\in\N$, we set
  \begin{equation}
  \begin{split}
 & l_n(\gamma)=\sum_{k=1}^n\vert\gamma(\frac{k-1}{n})-\gamma(\frac{k}{n})\vert,\\
 & \mathscr{F}_n(\mathscr{G})=\sum_{\gamma\in\mathscr{G}}\sigma_{a_\gamma a_\gamma^\prime} l_n(\gamma).
 \end{split}
  \label{elle-n}
  \end{equation}
  Since the map $\bar{\mathcal{G}}\ni\mathscr{G}\rightarrow\mathscr{F}_n(\mathscr{G})\in\R$
   is continuous for each $n\in\N$ the semicontinuity follows from
  \[ \mathscr{F}(\mathscr{G})=\sup_n \mathscr{F}_n(\mathscr{G}).\]

  2. Set $\lambda^*=\inf_{\mathscr{G}\in\bar{\mathcal{G}}}\mathscr{F}(\mathscr{G})$ and let
  $\{\mathscr{G}_j\}_{j=1}^\infty\subset\bar{\mathcal{G}}$ be a minimizing sequence:
  \begin{equation}
  \lim_{j\rightarrow+\infty}\mathscr{F}(\mathscr{G}_j)=\lambda^*.
  \label{min-seq}
  \end{equation}
  Since $\bar{\mathcal{G}}$ is compact there exist $\hat{\mathscr{G}}\in\bar{\mathcal{G}}$ and a subsequence still
  denoted $\{\mathscr{G}_j\}_{j=1}^\infty$ such that
  \[\hat{\mathscr{G}}=\lim_{j\rightarrow+\infty}\mathscr{G}_j.\]
  This \eqref{min-seq} and the semicontinuity in Step 1. imply
  \[\mathscr{F}(\hat{\mathscr{G}})\leq\lim_{j\rightarrow+\infty}\mathscr{F}(\mathscr{G}_j)=\lambda^*.\]
This proves the first statement of the Proposition. The minimizer $\hat{\mathscr{G}}\in\bar{\mathcal{G}}$ may be singular. See Figure \ref{sing} for a simple example.

\begin{figure}
  \begin{center}
 \begin{tikzpicture}[scale=.7]
\draw[fill=gray] (0,0) circle [radius=3];;

\draw[fill=lightgray!75!white!25!] (1.5,.3) circle [radius=.3];;
\path[fill=lightgray!75!white!25!] (0,3)  arc [radius=3, start angle=90, end angle= 180]
 arc [radius=3, start angle=-90, end angle= 0] to (0,3);
\path[fill=lightgray] (.7765,2.898) -- (0,0)--(.5,0) arc [radius=1, start angle=180, end angle= -90] to (-.5,-1)
 arc [radius=.5, start angle=270, end angle=180] arc [radius=.5, start angle=0, end angle=140] to (-2.598,-1.5)
 arc [radius=3, start angle=-150, end angle=75]
  to (.7765,2.898);
  \path[fill=lightgray]  (.4831,.025) arc [radius=.13, start angle=-60.76, end angle=0] to (.8,1.7)-- (.3,1)--(.4831,.025);
  \draw [red, line width=1] (0,3)-- (0,.15)
 arc [radius=.15, start angle=180, end angle=270] to (1.5,0) arc [radius=.3, start angle=-90, end angle=120]
  to (.631,.0733)  arc [radius=.5943 , start angle=299.24, end angle=270] to (-3,0);
\draw [black, line width=1] (.3916,2.974)-- (-.193,.052) arc [radius=.2, start angle=165, end angle=270] to (.5,-.2);
\draw [black, line width=1]  (.5,-.2) arc [radius=.2, start angle=-90, end angle=0] arc [radius=.8, start angle=180, end angle=-90] to (-.5,-.8);
\draw [black, line width=1] (-.5,-.8) arc [radius=.3, start angle=270, end angle=180] arc [radius=.7, start angle=0, end angle=140] to (-2.719,-1.268);
\node[] at (-.8,.5) {$\mathscr{I}$};
\node[] at (-1.7,1.55) {$a_1$};
\node[] at (1.6,.35) {$a_1$};
\node[] at (1,-1.8) {$a_2$};
\draw[->] (-1,2.3)--(-.05,2.3);
\draw[->] (1,2.3)--(.25,2.3);
\node[left] at (-1,2.3) {$\hat{\mathscr{G}}$};
\node[right] at (1,2.3) {$\mathscr{G}$};
 \end{tikzpicture}
  \end{center}
\caption{The picture refers to the case $N=\tilde{N}=2$ where $n_b=0$ and $n_s=1$. The black line $\mathscr{G}$ denotes one of the infinite line in $\mathscr{I}$ that separate $\Omega_{a_1}$ and $\Omega_{a_2}$. The red line denotes $\hat{\mathscr{G}}$, the shortest such line in $\bar{\mathscr{I}}$. Note that $\Omega_{a_1}$ not connected together with the particular structure of $\Omega_{a_2}$ force $\hat{\mathscr{G}}$ to describe twice the same segment. This is one of the singularities that, a priori, can be expected for the minimizing network  $\hat{\mathscr{G}}$.}
\label{sing}
\end{figure}
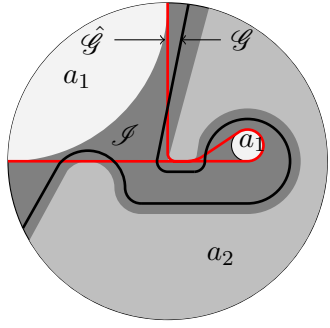

  3. Let $B_r(x)\subset\mathscr{I}\cap\Omega$ be a ball which has positive distance from the set of
  branching points of $\hat{\mathscr{G}}$. Then $B_r\cap\hat{\mathscr{G}}$ is the union of a family of simple arcs
  with extremes on $\partial B_r$. We know from  (iv) above that the arcs in $\hat{\mathscr{G}}$
  can be tangent but cannot cross. Therefore, if $\gamma_i([t,t^\prime])$
   and $\gamma_j([\tau,\tau^\prime])$ are
  two of these arcs we have that the segments $\mathrm{sg}(\gamma_i(t),\gamma_i(t^\prime))$ and
   $\mathrm{sg}(\gamma_j(\tau),\gamma_j(\tau^\prime))$ either coincide or have empty intersection.
  It follows that we can consider the network $\hat{\mathscr{G}}^\prime$ obtained from $\hat{\mathscr{G}}$
   by replacing each arc in $B_r$
   with the segment determined by its extremes. This implies that each arc in  $B_r$ actually coincides with
   the segment determined
    by its extremes. Indeed otherwise
    $\mathscr{F}(\hat{\mathscr{G}}^\prime)< \mathscr{F}(\hat{\mathscr{G}})$ in contradiction with the minimality of
    $\mathscr{F}(\hat{\mathscr{G}})$.
    In particular if $x\in\hat{\mathscr{G}}$, by reducing $r>0$ if necessary, we can assume that $B_r(x)\cap\hat{\mathscr{G}}$ coincides with a diameter of $B_r(x)$ and that, for each $\gamma\in\hat{\mathscr{G}}$, $\gamma\cap B_r(x)$ is either empty or coincides with that diameter.

4. If $x\in\partial\mathscr{I}\cap\Omega$ and, as before, $B_r(x)$ has positive distance from the set of branching points of $\hat{\mathscr{G}}$, then, arguing as in Step.3, we conclude that $B_r(x)\cap\hat{\mathscr{G}}$ is the union
  of a family of simple arcs contained in $B_r(x)\cap\bar{\mathscr{I}}$ and that, if $\gamma((t,t^\prime))=\gamma([0,1])\cap B_r(x)\cap\bar{\mathscr{I}}$, $\gamma\in\hat{\mathscr{G}}$, is one of these arcs,
   then $\gamma((t,t^\prime))$ coincides with the arc of minimal length which is contained in $B_r(x)\cap\bar{\mathscr{I}}$
    and has the extremes in
    $\gamma(t)$ and $\gamma(t^\prime)$.

    Since, as we have observed, $\partial\mathscr{I}\cap\Omega$ is a $C^2$ curve, the minimality of $\gamma((t,t^\prime))$ implies that, in the interval $(t,t^\prime)$, $\gamma$ is at least of class $C^1$.

    If $x=\gamma(t)$ for some $\gamma\in\hat{\mathscr{G}}$ and $t\in(0,1)$, then $\gamma$ can not intersect $\partial\mathscr{I}$ at $x$ and therefore, by taking $r>0$ sufficiently small we can ensure that the diameter of $B_r(x)$ parallel to $\gamma^\prime(t)$ is contained in $\bar{\mathscr{I}}$. Moreover if $\ddot{\gamma}(t)\neq 0$, where $\dot{}$ denotes derivative with respect to arc length, then $\ddot{\gamma}(t)$ points outside $\mathscr{I}$ and coincides with the curvature vector of $\partial\mathscr{I}$ at $x$.

 5. From Step.3 and Step.4 it follows that $\gamma\in\hat{\mathscr{G}}$ is piece wise $C^2$
  unless there is $t\in(0,1)$ such that $\gamma(t)$ coincides with a branching point and $\gamma^\prime$
   is discontinuous at $t$.

   \begin{figure}
  \begin{center}
 \begin{tikzpicture}[scale=.7]
\draw[fill=lightgray] (0,0) circle [radius=3];;
\path[fill=black] (0,0) circle [radius=.035];;
\draw[] (-1.5,-2.598)--(0,0)--(2.589,-1.5);
\draw[] (0,-.15)--(-.7764,-2.897);
\draw[] (0,-.15)--(0,-3);
\draw[] (0,-.15)--(2.1213,-2.1213);
\draw[] (.15,.246)--(-1.5,2.598);
\draw[]  (.15,.246)--(2.1213,2.1213);
\draw[]  (.15,.246)--(3,0);
\draw[blue,<->, line width=1.5] (-1.5,-2.598) arc [radius=3, start angle=240, end angle=330];
\draw[red] (.549,-2.049)--(-1.5,-2.598);
\draw[red] (.549,-2.049)--(2.598,-1.5);
\draw[red] (.549,-2.049)--(-.7764,-2.897);
\draw[red] (.549,-2.049)--(0,-3);
\draw[red] (.549,-2.049)--(2.1213,-2.1213);
\path[fill=black] (.549,-2.049) circle [radius=.035];;
\node[] at (-.13,.1) {$x$};
\node[] at (-1.8,.35) {$B_r(x)$};
\node[] at (1,-3.15) {$\beta$};
\node[] at (1.6,-.4) {$\gamma$};
\node[] at (.549,-1.849) {$x_m$};
\node[] at (-1.1,-1.15) {$\tilde{\mathscr{G}}$};
\draw[->] (1.5,-.45)--(1.3,-.7);
\draw[] (-1.1,-1.45)--(-1.1,-1.7);
\draw[->] (-1.1,-2)--(-1.1,-2.5);
\node[left] at (-1.5,-2.598) {$q_0$};
\node[right] at (2.548,-1.5) {$q_1$};
\node[] at (-.8764,-3.2) {$q^1$};
\node[] at (0,-3.3) {$q^2$};
\node[] at (2.2213,-2.3213) {$q^3$};
 \end{tikzpicture}
  \end{center}
\caption{The construction of $\tilde{\mathscr{G}}$.}
\label{bra}
\end{figure}

   We first consider the case where $x=\gamma(t)$ belongs to $\mathscr{I}\cap\Omega$. In this case we can choose $r>0$ in such a way that $B_r(x)\subset\mathscr{I}\cap\Omega$ has positive distance from the set of the branching points of $\hat{\mathscr{G}}$ which do not coincide with $x$. This and Step.3 imply that, for $r>0$ sufficiently small,
   $B_r(x)\cap\hat{\mathscr{G}}$ is the union of a finite family of radii of $B_r(x)$.
 In particular there are $0<t_0<t<t_1<1$
     such that
     \[\mathrm{sg}[x,\gamma(t_0))\cup\mathrm{sg}[x,\gamma(t_1))=\gamma((t_0,t_1)).\]
   The extremes $q_0=\gamma(t_0)$ and $q_1=\gamma(t_1)$ divide $\partial B_r(x)$ in two closed arcs
    $\beta$ and $\beta_1$ with $\beta$ the shortest one (we do not exclude the case where $\beta$ reduces to a point).
   Note that there may be several arcs through $x$ with properties analogous to $\gamma((t_0,t_1))$ but we can assume that
   the arc $\beta$ is minimal in the sense that, if $\tilde{\beta}$ is the arc associated to some other
    $\tilde{\gamma}\in\hat{\mathscr{G}}$ through $x$ analogous to $\gamma$, then
    \[\beta\subset\tilde{\beta}.\]
    This follows from the fact that, as we have seen, $\gamma_i, \gamma_j\in\hat{\mathscr{G}}$
     can be tangent but cannot cross. For the same reason the extremes $q^i\in \partial B_r(x)$, $i=1,2,3$ of
     the arcs (radii) that originate in a branching point that coincides with $x$ must all three together
     belong to $\beta$ or to $\beta_1$.

     Let $R$ be the set of the subarcs $\gamma((\tau_1,\tau_2))$, $\gamma\in\hat{\mathscr{G}}$, $(\tau_1,\tau_2)\subset(0,1)$ which
     satisfy
    \[\gamma([\tau_1,\tau_2])=\mathrm{sg}[x,q],\;\;\text{for some}\;q\in\beta,\]
    and $\gamma((\tau_1,\tau_2))$ is either a subarc of one of the three branches of a triple junction with
    branching point in $x$ and $q^i\in\beta$, $i=1,2,3$ or is part of a subarc with extremes in $\{q_0,q_1\}$.

    We are now in the position of showing that, in contrast with the minimality of $\hat{\mathscr{G}}$,
     the assumption of the discontinuity of $\gamma^\prime$ at $t\in(0,1)$ implies the existence of a
     network $\tilde{\mathscr{G}}\in\bar{\mathcal{G}}$ that satisfies
     \begin{equation}
     \mathscr{F}(\tilde{\mathscr{G}})<\mathscr{F}(\hat{\mathscr{G}}).
     \label{Fg<fg}
     \end{equation}
     To define $\tilde{\mathscr{G}}$ we only change $\hat{\mathscr{G}}\cap\bar{B}_r(x)$ by replacing each
     $\gamma([\tau_1,\tau_2])=\mathrm{sg}[x,q]$ in $R$ with the arc $\tilde{\gamma}([\tau_1,\tau_2])=\mathrm{sg}[x_m,q]$
     where $x_m=\frac{1}{2}(q_0+q_1)$. From the definition of the set $R$ and from the fact
      that $\gamma_i, \gamma_j\in\hat{\mathscr{G}}$ satisfy the non crossing condition it follows that the same is true
      for $\tilde{\gamma}_i, \tilde{\gamma}_j\in\tilde{\mathscr{G}}$. Therefore the transformation
       $\gamma\rightarrow\tilde{\gamma}$ defines an admissible  network $\tilde{\mathscr{G}}\in\bar{\mathcal{G}}$
       and  \eqref{Fg<fg} follows from $\vert q-x_m\vert<\vert q-x\vert$ for $q\in\beta$. The construction of $\tilde{\mathscr{G}}$ is illustrated in Figure \ref{bra} where the branching points of $\hat{\mathscr{G}}$ that coincide with $x$ are slightly displaced from $x$.

       The discussion of the case where the branching point $x=\gamma(t)$ belongs to $\partial\mathscr{I}\cap\Omega$ is similar. In this case some of the radii may be replaced by arcs of minimal length with extremes $x$ and a point $q\in\bar{\mathscr{I}}\cap \partial B_r(x)$. The point $x_m$ for the construction of $\tilde{\mathscr{G}}$ can be defined as before provided $r>0$ is taken sufficiently small. The proof is complete.
\end{proof}
\begin{corollary}
\label{cor}
Let $\bar{\mathcal{G}_f}$ the closure of  $\mathcal{G}_f$, the set of networks defined at the end of the Introduction. Then there exists an optimal network $\mathscr{G}_f\in\bar{\mathcal{G}_f}$ such that
\begin{equation}
\mathscr{F}(\mathscr{G}_f)=\inf_{\mathscr{G}\in \mathcal{G}_f}\mathscr{F}(\mathscr{G}).
\label{inF}
\end{equation}
\end{corollary}
\begin{proof}
Similarly to $\bar{\mathcal{G}}$, also the closure $\bar{\mathcal{G}_f}$ of $\mathcal{G}_f$ with the distance \eqref{metric} is a compact space. Then, the same arguments developed in the proof of Proposition \ref{opt-G} show the existence of an optimal \emph{free} network $\mathscr{G}_f$ which is not subject to the constraint of being a subset of $\bar{\mathscr{I}}$.
\end{proof}

The determination of $\mathscr{G}_f$ is a purely geometric problem and we can expect that, generically,  $\mathscr{G}_f$ is a known object and there are constants
 $d_0>0$ and $c_0>0$ such that, for $d\in(0,d_0]$ and for each $\mathscr{G}^\prime\in\bar{\mathcal{G}}_f$ with the same end points as $\mathscr{G}_f$ it results
 \begin{equation}
 d(\mathscr{G}^\prime,\mathscr{G}_f)\geq d\quad\Rightarrow\quad  \mathscr{F}(\mathscr{G}^\prime)-\mathscr{F}(\mathscr{G}_f)\geq c_0d^2.
 \label{GIn}
 \end{equation}
 We say that $\mathscr{G}_f$ is nondegenerate if  $\mathscr{G}_f$ satisfies \eqref{GIn} for some positive constants $d_0$ and $c_0$.

 \section{A theorem on the fine structure of $u^\epsilon$.}
 \label{Fine}
 In this Section we assume $h_1$, $h_2$ and
 \begin{description}
 \item[$h_3$] The phases $\Omega_a$, $a\in A$ are connected.
 \item[$h_4$] For each pair $a_i,a_j\in A$, $a_i\neq a_j$, there is a smooth map $u_{a_ia_j}$ that connects $a_i$ to $a_j$, $\sigma_{a_ia_j}=\int_\R\vert u^\prime_{a_ia_j}\vert^2 dt$ and the inequality
\begin{equation}
\sigma_{a_ia_j}+\sigma_{a_ia_k}>\sigma_{a_j a_k},\;\;a_i,a_j,a_k\in A,
\label{triangle}
\end{equation} holds.

\end{description}
\noindent
We prove that, under assumption $h_3$, we can give a quite precise description of the fine structure of $u^\epsilon$ for small $\epsilon>0$.

We expect that, at least in some sense, the connectivity of the phases assumed in $h_3$ holds true but we are not able to prove it. To our knowledge  $h_3$ has not yet be established or disproved.

 \begin{theorem}
\label{fine}
Assume $n=2$ and $h_1-h_4$. Assume that $\delta=C\epsilon^\frac{1}{6}$ is a regular value and that the boundary datum $v_0^\epsilon$ is such that
 \begin{equation}
 \vert I_a\vert\leq C\epsilon^\frac{1}{3},\;\;a\in\tilde{A}.
 \label{h}
 \end{equation}
 Let $u^\epsilon$ be a minimizer of \eqref{min} and let $\mathscr{H}_f\subset\bar{\mathcal{G}}_f$ be the set of the free minimizers of $\mathscr{F}(\mathscr{G})$ given by Corollary \ref{cor}. Assume that each $\mathscr{G}\in\mathscr{H}_f$ is nondegenerate and that the sets $S_a$, $a\in A$ determined by  $\mathscr{G}$ via $\bar{\Omega}\setminus \mathscr{G}=\cup_a S_a$ are connected.

Then there exist $\mathscr{G}_f\in\mathscr{H}_f$ and positive constants $K,k$ and $C$ such that, for small $\epsilon>0$
\begin{equation}
\vert u^\epsilon(x)-a\vert\leq Ke^{-\frac{k}{\epsilon}(d(x,\mathscr{G}_f)-C\epsilon^\frac{1}{6})^+},\;\;x\in S_{a,f},\;a\in A,
\label{Stru}
\end{equation}
where the sets $S_{a,f}$, $a\in A$ are defined by.
\[\bar{\Omega}\setminus\mathscr{G}_f=\cup_{a\in A}S_{a,f}.\]
\end{theorem}
We begin by  sketching a possible approach to the proof of Theorem \ref{fine}. We denote by $e_\epsilon$ a generic quantity which converges to $0$ as $\epsilon\rightarrow 0$.

 In Proposition \ref{opt-G} we have attached to the unknown diffuse interface $\mathscr{I}$ a geometrical object: the optimal network $\hat{\mathscr{G}}\subset\bar{\mathscr{I}}$. $\hat{\mathscr{G}}$ is a kind of \emph{spine}  of class $C^1$ which is intimately related to $\mathscr{I}$ and admits a normal unit vector $\nu_x$ for each $x\in\hat{\mathscr{G}}$ which is not a branching or an end point. Moreover $\mathscr{F}(\hat{\mathscr{G}})$ can be considered as a generalized  \emph{weighed length} of the diffuse interface $\mathscr{I}$ and we can expect a lower bound of the form
 \begin{equation}
 J_\Omega^\epsilon(u^\epsilon)\geq\mathscr{F}(\hat{\mathscr{G}})-e_\epsilon,
 \label{LB}
 \end{equation}
  for the energy of $J_\Omega^\epsilon(u^\epsilon)$ of a minimizer $u^\epsilon$ of problem \eqref{min}.
 This conjecture is based on the fact that, given $\beta\in(0,1)$, due to the smallness of $\vert\mathscr{I}\vert$, for most $x\in\hat{\mathscr{G}}$, the connected component $F_x$ of $\{y=x+s\nu_x,\,\vert s\vert\leq\epsilon^\beta\}\cap\bar{\mathscr{I}}$ that contains $x$ should have the extremes in two different phases $\Omega_a$ and $\Omega_{a^\prime}$.

While a direct study of the structure of $\mathscr{I}$ as dictated by the minimality of $u^\epsilon$ seems an almost impossible task, we can expect that, for small $\epsilon>0$, the shape of $\mathscr{I}$ should be approximately determined by the condition that $\mathscr{F}(\hat{\mathscr{G}})$ be minimum. This means minimizing $\mathscr{F}(\mathscr{G})$ without imposing the constraint $\mathscr{G}\subset\bar{\mathscr{I}}$ as we have discussed in Corollary \ref{cor}. In other words we conjecture that, for small $\epsilon>0$, $\mathscr{I}$ is a \emph{blurring} of the optimal free network $\mathscr{G}_f$ considered in  Corollary \ref{cor} and that we have an  upper bound of the form
 \begin{equation}
 J_\Omega^\epsilon(u^\epsilon)\leq\mathscr{F}(\mathscr{G}_f)+e_\epsilon.
 \label{UB}
 \end{equation}
 Indeed, if $S_{a,f}$, $a\in A$ are the sets defined by the decomposition $\bar{\Omega}\setminus\mathscr{G}_f=\cup_{a\in A}S_{a,f}$, then \eqref{UB} is established by constructing a suitable comparison map which is a perturbation of the step map $v^0=\sum_{a\in A}a\1_{S_{a,f}}$, see Lemma \ref{UBerr} where \eqref{UB} is proved with $e_\epsilon\leq C\epsilon\vert\ln{\epsilon}\vert^2$.

 The idea is then to show that $\hat{\mathscr{G}}$ is actually a small perturbation of $\mathscr{G}_f$.

\noindent  From \eqref{LB} and \eqref{UB} we have
 \begin{equation}
 \mathscr{F}(\hat{\mathscr{G}})-\mathscr{F}(\mathscr{G}_f)\leq e_\epsilon.
 \label{F-F}
 \end{equation}
  $\hat{\mathscr{G}}$ may not have the same end points as $\mathscr{G}_f$ and we can not directly conclude from
 \eqref{GIn} and \eqref{F-F} that
 \begin{equation}
 d(\hat{\mathscr{G}},\mathscr{G}_f)\leq e_\epsilon.
 \label{GnearG}
 \end{equation}
 This inequality follows from \eqref{GIn} and \eqref{F-F} after observing that, due to the smallness of the intervals $I_a$ introduced in \eqref{Bdry-int}, there exists a small deformation  $\hat{\mathscr{G}}^\prime$  of  $\hat{\mathscr{G}}$ that satisfies
 \begin{equation}
 d(\hat{\mathscr{G}}^\prime,\hat{\mathscr{G}})\leq e_\epsilon,\quad\text{and}\quad
\vert \mathscr{F}(\hat{\mathscr{G}}^\prime)-\mathscr{F}(\hat{\mathscr{G}})\vert\leq e_\epsilon,
\label{gprime-g}
\end{equation}
see Lemma \ref{gf-gf} where we establish this estimate with $e_\epsilon\leq C\epsilon\vert\ln{\epsilon}\vert$.
The estimate \eqref{GnearG} is the key step toward the understanding of the fine structure of minimizers $u^\epsilon$ of problem \eqref{min}. Indeed  \eqref{GnearG} determines, modulo a small error $e_\epsilon$, the shape and the location of the sets $\hat{S}_a$, $a\in A$, defined by \eqref{G0-sep} with $\mathscr{G}=\hat{\mathscr{G}}$. This implies the knowledge of the location of the phases $\Omega_a\subset\hat{S}_a$, $a\in A$, that is the knowledge of the regions where the minimizer is near to one or to another of the zeros of $W$ and is the basis for further analysis.
\vskip.3cm

Unfortunately the approach outlined above finds an essential obstruction in the proof of the lower bound \eqref{LB}. We can not a priori exclude pathological situation  like the one illustrated in Figure \ref{sing} and, in particular, we don't have control of the curvature of $\hat{\mathscr{G}}$ and the \emph{fibers} $F_x$ and $F_y$
associated to different point $x,y\in\hat{\mathscr{G}}$ may have nonempty intersection.

In the remaining of this Section we show that, in various nontrivial cases, this obstruction can be removed and the estimate \eqref{GnearG} rigorously established under assumption $h_3$.
\vskip.3cm

\begin{proposition}
\label{LB3}
Assume the same as in Theorem \ref{fine}, and $\delta=\epsilon^\alpha$, $\alpha\in(0,\frac{1}{2})$. If
$h_3$ holds, and \eqref{triangle} holds,
 then there exists $C_1>0$ such that
\begin{equation}
J_\Omega^\epsilon(u^\epsilon)\geq\mathscr{F}(\hat{\mathscr{G}})-C_1\epsilon^\frac{1}{3},
\label{Lbequ}
\end{equation}
where $\hat{\mathscr{G}}\subset\bar{\mathscr{I}}$ is the optimal network given by Proposition \ref{opt-G}.
\end{proposition}
\begin{proof}
We have seen that for each $x\in\hat{\mathscr{G}}$, which does not coincide with an end or with a branching point, there is a normal line $\{x+t\nu_x, t\in\R\}$ to $\hat{\mathscr{G}}$ at $x$. We let $F_x$ be the connected component of $\{x+t\nu_x, \vert t\vert\leq\epsilon^\beta\}\cap\bar{\mathscr{I}}$ that contains $x$.
Given $x_i\in\hat{\mathscr{G}}$, $i=1,2$, we let $\hat{d}(x_1,x_2)$ the length of the arc in $\hat{\mathscr{G}}$ that connects $x_1$ to $x_2$ and has minimal length.
\begin{lemma}
\label{nearBra}
Assume the same as in Proposition \ref{LB3}.
Assume that $x,y\in\hat{\mathscr{G}}$, $t_x,t_y$ and $\gamma^x,\gamma^y\subset\hat{\mathscr{G}}$ satisfy
$x=\gamma^x(t_x)$, $y=\gamma^y(t_y)$ with $(\gamma^x,t_x)\neq(\gamma^y,t_y)$.
Then there exists $e_\epsilon$ such that $F_x\cap F_y\neq\emptyset$ implies
\[\hat{d}(x,p)\leq C\epsilon^{\frac{1}{2}-\alpha},\]
for some branching point $p\in\hat{\mathscr{G}}$.
\end{lemma}
\begin{proof}
We give a detailed proof only for the case where there exists $a\in A$ such that $x,y\in\hat{\mathscr{C}}_a=\partial\hat{S}_a$ and $F_x\cap F_y\cap\hat{S}_a\neq\emptyset$.
Let $\xi$ be a point in $F_x\cap F_y\cap\hat{S}_a$, by definition of $F_x$ and $F_y$ $\mathrm{sg}[x,\xi]\cup\mathrm{sg}[\xi,y]$ is contained in $\bar{\mathscr{I}}$ and therefore divides $\hat{S}_a$ in two parts and $h_3$ implies that $\Omega_a$ is contained in just one of them. We let $E$ be closure of the part that does not contain  $\Omega_a$. $\partial E$ is a Jordan curve which is the union of $\mathrm{sg}[x,\xi]\cup\mathrm{sg}[\xi,y]$, of suitable sub arcs of $\gamma^x$ and $\gamma^y$ and of a number $K$ of arcs $\gamma_1,\ldots\gamma_k\subset\hat{\mathscr{C}}_a$ which successively connect branching points $p_1,\ldots,p_{K+1}$.

We observe that the proof is quite straightforward if
\begin{equation}
\sigma_{aa^\prime}=\sigma,\;\;a\neq a^\prime\in A.
\label{equal}
\end{equation}
Let $l_1,\ldots,l_K$ be the length of $\gamma_1,\ldots,\gamma_K$ and $l_0,l_{K+1}$ the length of the sub arcs of $\gamma^x,\gamma^y$ on $\partial E$. Chose $h\in\{0,\ldots,K+1\}$ so that $l_h=\max\{l_0,\ldots,l_{K+1}\}$. Let $\hat{S}_{a_0},\ldots,\hat{S}_{a_{K+1}}$ be the sets determined by the conditions: $\gamma_i\subset\hat{\mathscr{C}}_{a_i}$, $a_i\neq a$, $i=0,\ldots,K+1$. Let
$\tilde{\mathscr{G}}$ be the network obtained by replacing the arc $\gamma_h$ with $\mathrm{sg}[x,\xi]\cup\mathrm{sg}[\xi,y]$, $\hat{S}_a$ with $\tilde{S}_a=\hat{S}_a\setminus E$ and $\hat{S}_{a_h}$ with $\tilde{S}_{a_h}= (\hat{S}_{a_h}\cup E)^\circ$. Note that the network $\tilde{\mathscr{G}}$ has the same number of arcs and the same number of branching points as $\hat{\mathscr{G}}$. Moreover $E\subset\bar{\mathscr{I}}$ implies $\tilde{\mathscr{G}}\subset\bar{\mathscr{I}}$ and therefore $\tilde{\mathscr{G}}\in\bar{\mathcal{G}}$. Then the minimality of  $\hat{\mathscr{G}}$ implies
\begin{equation}
\sigma_{a_ha}l_h+\sum_{i\neq h}\sigma_{a_ia}l_i\leq\sum_{i\neq h}\sigma_{a_ia_h}l_i+\sigma_{aa_h}(\vert\xi-x\vert+\vert y-\xi\vert).
\label{+opt}
\end{equation}
This $\vert F_x\vert,\vert F_y\vert\leq 2\epsilon^\beta$ and \eqref{equal} yield
\begin{equation}
l_i\leq l_h\leq4\epsilon^\beta.
\label{g*1}
\end{equation}
To extend the proof to the general case where we have merely \eqref{triangle} instead of \eqref{equal} we make essential use of the fact already observed in the proof of Proposition \ref{opt-G}, that $E\subset\bar{\mathscr{I}}$ implies that the arcs $\gamma_i$, $i=0,\ldots,K+1$ turn their convexity inward $E$.
This property of $E$ implies that, if $\mathcal{B}$ is the set of the balls contained in $E$, the subset
$\mathcal{B}_M\subset\mathcal{B}$ of the balls $B\in\mathcal{B}$ such that $\partial B\cap\partial E$
 contains three distinct points is finite and $\sharp\mathcal{B}_M\leq  n_N$ where $n_N$ is a number that depends only $N=\sharp A$. We also note that from \eqref{|I|} and $\delta=\epsilon^\alpha$ it follows that, if $B_r(z)$ is one of the balls in $\mathcal{B}_M$, then
\begin{equation}
r\leq\frac{1}{\sqrt{\pi}}\epsilon^{\frac{1}{2}-\alpha}.
\label{r<}
\end{equation}

\begin{figure}
  \begin{center}
\begin{tikzpicture}[scale=1]
\draw [] (-1.5,0) circle [radius=1];
\draw [] (2,0) circle [radius=1.5];
\draw [blue] (-1,.866) to [out=-30,in=210] (1.25,1.299);
\draw [blue] (-1,-.866) to [out=30,in=150] (1.25,-1.299);
\draw [blue] (-1,.866)--(-1.866,1.366);
\draw [blue] (-1,-.866)--(-1.866,-1.366);
\draw [blue] (1.25,1.299)--(2.116,1.799);
\draw [blue] (1.25,-1.299)--(2.116,-1.799);
\node[left] at (-1.866,-1.366) {$p_1^-$};
\node[left] at (-1.866,1.366) {$p_1^+$};
\node[] at (-1.35,-1.61) {$\gamma_1^{p,-}$};
\node[] at (-1.35,1.61) {$\gamma_1^{p,+}$};
\path[fill] (-1.866,-1.366) circle [radius=.035];
\path[fill] (-1.866,1.366)circle [radius=.035];
\path[fill] (-1,-.866) circle [radius=.035];
\path[fill] (-1,.866)circle [radius=.035];
\path[fill] (1.25,-1.299) circle [radius=.035];
\path[fill] (1.25,1.299) circle [radius=.035];
\node[below] at (0,-.7) {$\gamma_1^-$};
\node[above] at (0,.7) {$\gamma_1^+$};
\node[left] at (-1.5,0) {$B_1$};
\node[right] at (2,0) {$B_2$};
\draw[blue,dotted] (-1,-.866)--(-1,.866);
\draw[blue,dotted] (1.25,-1.299)--(1.25,+1.299);
\node[below] at (-1,-.866) {$z_1^-$};
\node[above] at (-1,.866) {$z_1^+$};
\node[below] at (1.25,-1.299) {$\zeta_1^-$};
\node[above] at (1.25,1.299) {$\zeta_1^+$};
\end{tikzpicture}
\end{center}
\caption{Neighboring balls $B_1, B_2\in\mathcal{B}_M$.}
\label{neigh}
\end{figure}
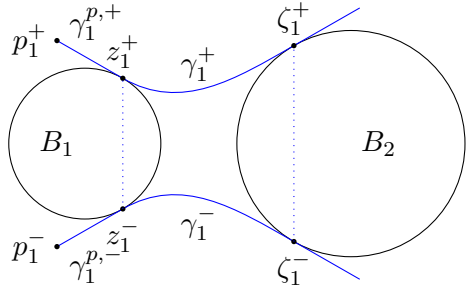

We say that $B_1\in\mathcal{B}_M$ and $B_2\in\mathcal{B}_M$ are neighbors, see Figure \ref{neigh}, if there are arcs $\gamma_1^\pm\subset(\partial E\cap\hat{\mathscr{C}}_a)$ which do not contain branching points, are tangent to $B_1$ and $B_2$ and together with the segments $\mathrm{sg}[z_1^-, z_1^+]$, $\mathrm{sg}[\zeta_1^-, \zeta_1^+]$ $\{z_1^\pm\}=\gamma_1^\pm\cap B_1$, $\{\zeta_1^\pm\}=\gamma_1^\pm\cap B_2$ $i=1,2$, compose a Jordan curve with interior contained in $E$.

The arc $\gamma_1^\pm$ is a sub arc of, say $\hat{\gamma}_1^\pm$, one of the $n_s$ arcs that compose $\hat{\mathscr{G}}$. We let $p_1^\pm$ be the branching point, extreme of $\hat{\gamma}_1^\pm$, such that $z_1^\pm$ belongs to the sub arc of $\hat{\gamma}_1^\pm$ with extremes  $p_1^\pm$ and  $\zeta_1^\pm$ and denote  $\gamma_1^{p,\pm}$ the sub arc of $\hat{\gamma}_1^\pm$ with extremes $p_1^\pm$ and  $z_1^\pm$ .

$B\in\mathcal{B}_M$ can have one, two or more neighbors. It follows that $E$ has the structure of a tree. A branch of the tree is a sequence $\{B_j\}_{j=1}^J\subset\mathcal{B}_M$ such that $B_j$ and $B_{j+1}$ are neighbors for $j=1,\ldots,J-1$; $B_j$ has two  neighbors for $j=2,\ldots,J-1$ and $B_1$ and $B_J$ have either just one or more than two  neighbors.

To begin with we consider a branch such that $B_2$ is the unique neighbor of $B_1$ and show that the length of the sub arc of $\partial E$ which touches $B_1$ and has the extremes in $\zeta_{J-1}^\pm\in\partial B_J$ is small in $\epsilon$.

For each $1\leq j<J$ we let $\psi_j\subset\partial E$ the arc which has the extremes in
 $z_j^\pm\in\partial B_j$ and touches $B_1$.

 The arc $\psi_j$ is the union of:  say $K_j$, arcs $\hat{\gamma}_k\in\hat{\mathscr{G}}$, $k=1,\ldots,K_j$ that touches some of the ball $B_1,\ldots, B_j$ and are different from  $\hat{\gamma}_j^\pm$, and the arcs $\gamma_j^{p,\pm}$.

 Let $a_j^\pm\in A$ and $a_k\in A$ be defined by $\gamma_j^\pm\subset\partial\hat{S}_{a^\pm}$ and $\hat{\gamma}_k\subset\partial\hat{S}_{a_k}$, $k=1,\ldots,K_j$ and set $l_j^\pm=\vert \gamma_j^\pm\vert$, $\lambda_j^\pm=\vert\gamma_j^{p,\pm}\vert$, $l_k=\vert\hat{\gamma}_k\vert$, $k=1,\ldots,K_j$.

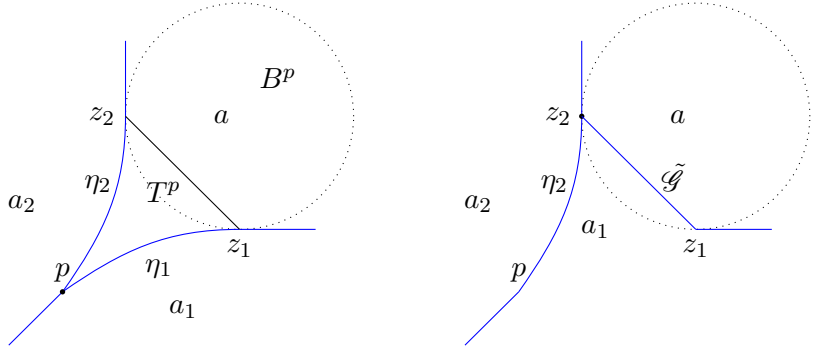
\begin{figure}
  \begin{center}
\begin{tikzpicture}[scale=1]
\draw [dotted] (-3,0) circle [radius=1.5];
\draw [dotted] (3,0) circle [radius=1.5];
\draw [blue] (-5.3284,-2.3284) to [out=35,in=180] (-3,-1.5);
\draw [blue] (-5.3284,-2.3284) to [out=55,in=270] (-4.5,0);
\draw [blue] (-4.5,0)--(-4.5,1);
\draw [blue] (-3,-1.5)--(-2,-1.5);
\draw [blue] (-5.3284,-2.3284)--(-6.0355,-3.0355);
\node[left,above] at (-5.3284,-2.3284) {$p$};
\node[] at (-4,-1) {$T^p$};
\node[] at (-2.5,.5) {$B^p$};
\node[below] at (-4.061,-1.761) {$\eta_1$};
\node[above] at (-4.861,-1.161) {$\eta_2$};
\node[below,right] at (-4.061,-2.561) {$a_1$};
\node[] at (-5.861,-1.161) {$a_2$};
\node[left] at (-3,0) {$a$};
\node[below] at (-3,-1.5) {$z_1$};
\node[left] at (-4.5,0) {$z_2$};
\draw []  (-3,-1.5)--(-4.5,0);
\path [fill] (-5.3284,-2.3284) circle [radius=.035];
\draw [blue] (.6716,-2.3284) to [out=55,in=270] (1.5,0);
\draw [blue] (1.5,0)--(1.5,1);
\draw [blue] (3,-1.5)--(4,-1.5);
\draw [blue] (.6716,-2.3284)--(-.0355,-3.0355);
\node[left,above] at (.6716,-2.3284) {$p$};
\node[above] at (1.139,-1.161) {$\eta_2$};
\node[below] at (3,-1.5) {$z_1$};
\node[left] at (1.5,0) {$z_2$};
\node[left] at (2,-1.5) {$a_1$};
\node[] at (.139,-1.161) {$a_2$};
\node[left] at (3,0) {$a$};
\node[below] at (2.7,-.5) {$\tilde{\mathscr{G}}$};
\path [fill] (1.5,0) circle [radius=.035];
\draw [blue] (3,-1.5)--(1.5,0);
\end{tikzpicture}
\end{center}
\caption{$p$, $B^p$, $T^p$ and the construction of $\tilde{\mathscr{G}}$.}
\label{p}
\end{figure}

 Note that for each branching point $p\in\partial E$, see Figure \ref{p}, there is $B^p\in\mathcal{B}_M$ and arcs $\eta_i\subset\partial E$, $i=1,2$ which originate in $p$, are tangent to $B^p$ and do not contain branching points different from $p$. We let $\{z_i\}=\eta_i\cap\partial B^p$ be the point of tangency of $\eta_i$ with $B^p$, $i=1,2$ and denote $T^p\subset E$ the closure of interior of the Jordan curve determined by  $\mathrm{sg}[z_1,z_2]$ and $\eta_1$ and $\eta_2$. Let $\hat{S}_a$, $\hat{S}_{a_i}$, $i=1,2$, that meet at $p$. Denote $l_i, i=1,2,$ the length of $\eta_i$, assume, for definiteness, that $l_1\geq l_2$ and  consider the network $\tilde{\mathscr{G}}$ obtained from  $\hat{\mathscr{G}}$ by replacing $\eta_1$ with the segment $\mathrm{sg}[z_1,z_2]$ and  $\hat{S}_{a_i}$ with $(\tilde{S}_{a_i}=(\hat{S}_{a_i}\cup T^p)^\circ$ and $\hat{S}_a$ with $\hat{S}_a\setminus T^p$. By construction, see Figure \ref{p}, the network $\tilde{\mathscr{G}}$ has the same number of arcs and of branching points as $\hat{\mathscr{G}}$ and since $T^p$ is a subset of $E\subset\bar{\mathscr{I}}$ is contained $\bar{\mathscr{I}}$. Hence the minimality of $\hat{\mathscr{G}}$ implies
\begin{equation}
\begin{split}
&\sigma_{a_1a}l_1+\sigma_{a_2a}l_2\leq\sigma_{a_1a_2}l_2+\sigma_{a_1,a}\vert z_1-z_2\vert,\\
&\Rightarrow\quad
\sigma_{a_1a}(l_1-l_2)+(\sigma_{a_1a}+\sigma_{a_2a}-\sigma_{a_1a_2})l_2
\leq\sigma_{a_1,a}\frac{2}{\sqrt{\pi}}\epsilon^{\frac{1}{2}-\alpha},\\
&\Rightarrow\quad l_i\leq\rho\epsilon^{\frac{1}{2}-\alpha},\;\;i=1,2,
\end{split}
\label{p-opt}
\end{equation}
where we have used \eqref{triangle}, \eqref{r<} and indicated with $\rho>0$ a constant that depends only on the value of $\sigma_{aa^\prime}$, $a\neq a^\prime\in A$.

\begin{figure}
  \begin{center}
\begin{tikzpicture}[scale=1]
\draw[dotted] (-2,2.5) circle [radius=1];
\draw[dotted] (-2,-2.5) circle [radius=1];
\draw[dotted] (2,2.5) circle [radius=1];
\draw[dotted] (2,-2.5) circle [radius=1];
\draw[dotted] (-7,2.5) circle [radius=1];
\draw[dotted] (-7,-2.5) circle [radius=1];
\draw [blue] (-1.5,3.366) to [out=-30,in=180] (0,3);
\draw [blue] (-1.5,1.634) to [out=30,in=180] (0,2);
\draw [blue] (0,3) to [out=0,in=210] (1.5,3.366);
\draw [blue] (0,2) to [out=0,in=150](1.5,1.634) ;
\draw [blue] (-1.5,-1.634) to [out=-30,in=180] (0,-2);
\draw [blue] (0,-2) to [out=0,in=210] (1.5,-1.634);
\draw [blue] (2.5,3.366) to [out=-30,in=180] (4,3);
\draw [blue] (2.5,1.634) to [out=30,in=180] (4,2);
\draw [blue] (2.5,-1.634) to [out=-30,in=180] (4,-2);
\draw [blue] (2.5,-3.366) to [out=30,in=180] (4,-3);
\draw [blue] (-6.5,3.366) to [out=-30,in=180] (-5,3);
\draw [blue] (-6.5,1.634) to [out=30,in=180] (-5,2);
\draw [blue] (-6.5,-1.634) to [out=-30,in=180] (-5,-2);
\draw [blue] (-6.5,-3.366) to [out=30,in=180] (-5,-3);
\draw [blue] (-4,3) to [out=0,in=210] (-2.5,3.366);
\draw [blue] (-4,2) to [out=0,in=150](-2.5,1.634) ;
\draw [blue] (-4,-2) to [out=0,in=210] (-2.5,-1.634);
\draw [blue] (-4,-3) to [out=0,in=150](-2.5,-3.366) ;
\draw [blue] (-1.5,3.366) to (-2,3.655);
\draw [blue] (-1.5,1.634) to (-2,1.345);
\draw [blue] (-1.5,-1.634) to (-2,-1.345);
\draw [blue] (2.5,3.366) to (2,3.655);
\draw [blue] (2.5,1.634) to (2,1.345);
\draw [blue] (2.5,-3.366) to (2,-3.655);
\draw [blue] (2.5,-1.634) to (2,-1.345);
\draw [blue] (-6.5,3.366) to (-7,3.655);
\draw [blue] (-6.5,1.634) to (-7,1.345);
\draw [blue] (-6.5,-3.366) to (-7,-3.655);
\draw [blue] (-6.5,-1.634) to (-7,-1.345);
\draw [blue] (-2.5,3.366) to (-2,3.655);
\draw [blue] (-2.5,1.634) to (-2,1.345);
\draw [blue] (-2.5,-3.366) to (-2,-3.655);
\draw [blue] (-2.5,-1.634) to (-2,-1.345);
\draw [blue] (1.5,3.366) to (2,3.655);
\draw [blue] (1.5,1.634) to (2,1.345);
\draw [blue] (1.5,-3.366) to (2,-3.655);
\draw [blue] (1.5,-1.634) to (2,-1.345);
\draw [blue] (-7.5,3.366) to (-7,3.655);
\draw [blue] (-7.5,1.634) to (-7,1.345);
\draw [blue] (-7.5,-3.366) to (-7,-3.655);
\draw [blue] (-7.5,-1.634) to (-7,-1.345);
\draw [blue] (-7.5,3.366) to (-9,2.5);
\draw [blue] (-7.5,1.634) to (-9,2.5);
\draw [blue] (-7.5,-3.366) to (-9,-2.5);
\draw [blue] (-7.5,-1.634) to (-9,-2.5);
\draw[dotted](-5,3)-- (-4,3);
\draw[dotted](-5,2)-- (-4,2);
\draw[dotted](-5,-3)-- (-4,-3);
\draw[dotted](-5,-2)-- (-4,-2);
\draw[dotted](-5,3)-- (-4,3);
\draw[dotted](-5,2)-- (-4,2);
\draw[dotted](-5,-3)-- (-4,-3);
\draw[dotted](-5,-2)-- (-4,-2);
\draw[dotted](-5,3)-- (-4,3);
\draw[dotted](-5,2)-- (-4,2);
\draw[dotted](-5,-3)-- (-4,-3);
\draw[dotted](-5,-2)-- (-4,-2);
\draw[dotted](4.5,3)-- (4,3);
\draw[dotted](4.5,2)-- (4,2);
\draw[dotted](4.5,-3)-- (4,-3);
\draw[dotted](4.5,-2)-- (4,-2);
\draw[blue] (-7,3.655)--(-7,4.155);
\draw[blue] (-7,1.345)--(-7,.845);
\draw[blue] (-7,-3.655)--(-7,-4.155);
\draw[blue] (-7,-1.345)--(-7,-.845);
\draw[blue] (-2,3.655)--(-2,4.155);
\draw[blue] (-2,1.345)--(-2,.845);
\draw[blue] (-2,-3.655)--(-2,-4.155);
\draw[blue] (-2,-1.345)--(-2,-.845);
\draw[blue] (2,3.655)--(2,4.155);
\draw[blue] (2,1.345)--(2,.845);
\draw[blue] (2,-3.655)--(2,-4.155);
\draw[blue] (2,-1.345)--(2,-.845);
\draw[blue] (-9.5,2.5)--(-9,2.5);
\draw[blue] (-9.5,-2.5)--(-9,-2.5);
\draw[blue] (1.5,-1.634)--(1.5,-3.366);
\node[] at (0,2.5) {$a$};
\node[] at (-5,2.5) {$a$};
\node[] at (-5,2.5) {$a$};
\node[] at (4,2.5) {$a$};
\node[] at (-8,3.4) {$a_k$};
\node[] at (-8,1.6) {$a_{k+1}$};
\node[] at (-8,-1.6) {$a_k$};
\node[] at (-8,-3.4) {$a_{k+1}$};
\node[] at (0,3.5) {$a_j^+$};
\node[] at (0,1.5) {$a_j^-$};
\node[] at (0,-1.5) {$a_j^+$};
\node[] at (0,-3) {$a_j^-$};
\node[] at (-5,-2.5) {$a_j^-$};
\node[] at (4,-2.5) {$a$};
\node[left] at (-7,2.5) {$B_1$};
\node[left] at (-7,-2.5) {$B_1$};
\node[left] at (-2,2.5) {$B_j$};
\node[left] at (-2,-2.5) {$B_j$};
\node[right] at (2,2.5) {$B_{j+1}$};
\node[right] at (2,-2.5) {$B_{j+1}$};
\path [fill] (-1.5,3.366) circle [radius=.035];
\node[above] at (-1.5,3.366) {$z_j^+$};
\path [fill] (-1.5,1.634) circle [radius=.035];
\node[below] at (-1.5,1.634) {$z_j^-$};
\path [fill] (-1.5,3.366) circle [radius=.035];

\path [fill] (-1.5,-1.634) circle [radius=.035];
\node[above] at (-1.5,-1.634) {$z_j^+$};
\path [fill=red] (-2,1.345) circle [radius=.04];
\node[above] at (-2,1.345) {$p_j^-$};
\path [fill] (1.5,3.366) circle [radius=.035];
\node[above] at (1.5,3.366) {$\zeta_j^+$};
\path [fill] (1.5,1.634) circle [radius=.035];
\node[below] at (1.5,1.634) {$\zeta_j^-$};
\path [fill] (1.5,-3.366) circle [radius=.035];
\node[below] at (1.5,-3.366) {$\zeta_j^-$};
\path [fill=red] (1.5,-1.634) circle [radius=.045];
\node[above] at (1.5,-1.634) {$\zeta_j^+$};
\path [fill] (2.5,3.366) circle [radius=.035];
\node[above] at (2.5,3.366) {$z_{j+1}^+$};
\path [fill] (2.5,1.634) circle [radius=.035];
\node[below] at (2.5,1.634) {$z_{j+1}^-$};
\path [fill] (2.5,-3.366) circle [radius=.035];
\node[below] at (2.5,-3.366) {$z_{j+1}^-$};
\path [fill] (2.5,-1.634) circle [radius=.045];
\node[above] at (2.5,-1.634) {$z_{j+1}^+$};
\end{tikzpicture}
\end{center}
\caption{A branch of $E$ and the replacement of $\lambda_j^{p,-}\cup\gamma_j^-$ with $\mathrm{sg}[\zeta_j^-,\zeta_j^+]$. The old branching point $p_j^-$ and the new branching point $\zeta_j^+$ are marked in red.}
\label{branch}
\end{figure}
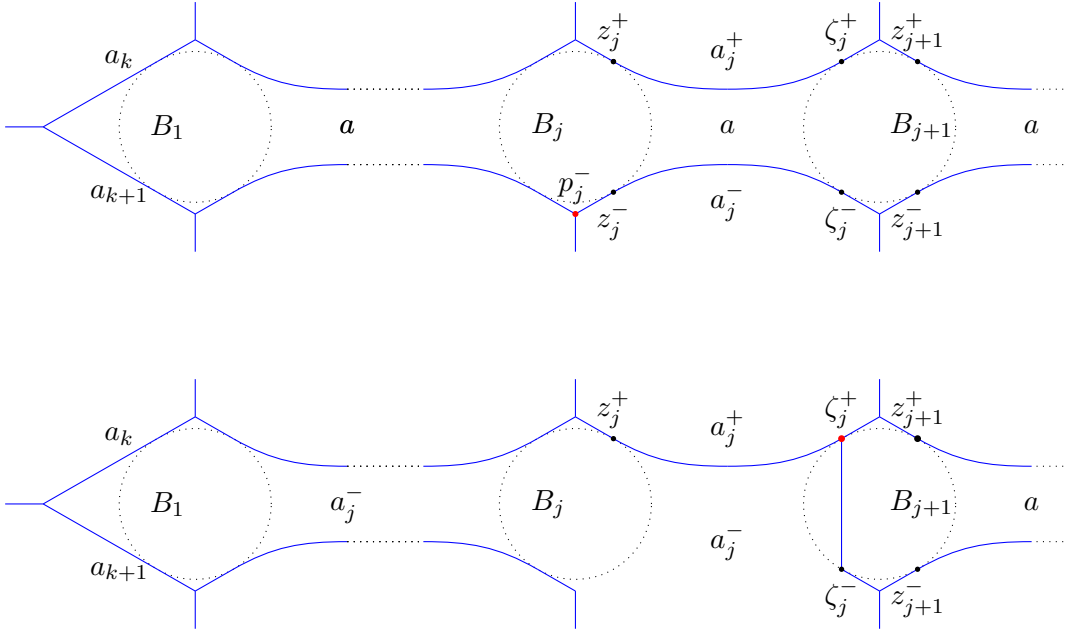

Now we show that, if there are $1\leq j<J-1$ and a constant $C_j>0$ such that $\vert\psi_j\vert\leq C_j\epsilon^{\frac{1}{2}-\alpha}$, then there is a constant $C_{j+1}>0$ such that $\vert\psi_{j+1}\vert\leq C_{j+1}\epsilon^{\frac{1}{2}-\alpha}$. We assume for definiteness that $l_j^-\geq l_j^+$. Then, as before, we
construct a network $\tilde{G}\subset\bar{\mathscr{I}}$ and prove the claim by using the minimality of $\hat{\mathscr{G}}$ which implies $\mathscr{F}(\hat{\mathscr{G}})\leq\mathscr{F}(\tilde{\mathscr{G}})$.
$\tilde{\mathscr{G}}$, see Figure \ref{branch}, is defined by replacing the arc $\gamma_j^{p,-}\cup\gamma_j^-$ with the segment $\mathrm{sg}[\zeta_j^-,\zeta_j^+]$ and defining the sets $\tilde{S}$ accordingly. The minimality of $\hat{\mathscr{G}}$ implies
\begin{equation}
\begin{split}
&\sigma_{a_j^-a}(l_j^-+\lambda_j^-)+\sum_{k=1}^{K_j}\sigma_{aa_k}l_k+\sigma_{a_j^+a}(l_j^++\lambda_j^+)\\
&\leq\sum_{k=1}^{K_j}\sigma_{a_j^-a_k}l_k +\sigma_{a_j^-a_j^+}(l_j^++\lambda_j^+)
+\sigma_{a_j^-a}\vert\zeta_j^--\zeta_j^+\vert,\\
&\Rightarrow\\
&\sigma_{a_j^-a}(l_j^--l_j^+)+(\sigma_{a_j^-a}+\sigma_{a_j^+a}-\sigma_{a_j^+a_j^-})l_j^+\\
&\leq\sum_{k=1}^{K_j}\sigma_{a_j^-a_k}l_k+\sigma_{a_j^-a_j^+}\lambda_j^++\sigma_{a_j^-a}\vert\zeta_j^--\zeta_j^+\vert.
\end{split}
\label{comp}
\end{equation}
From the induction assumption we have $\sum_{k=1}^{K_j}\sigma_{a_j^-a_k}l_k+\lambda_j^+\leq C_j\epsilon^{\frac{1}{2}-\alpha}$. This and \eqref{r<} imply that the right hand side of \eqref{comp} is bounded by $\sigma_M(C_j+\frac{2}{\sqrt{\pi}})\epsilon^{\frac{1}{2}-\alpha}$ where $\sigma_M=\max_{a\neq a^\prime}\sigma_{a,a^\prime}$. From this bound  \eqref{comp} and \eqref{triangle} we finally obtain
\begin{equation}
l_j^\pm\leq \frac{\sigma_M}{\sigma^*}(C_j+\frac{2}{\sqrt{\pi}})\epsilon^{\frac{1}{2}-\alpha},
\label{lpm}
\end{equation}
where $\sigma^*=\min\{\sigma_{aa^\prime},\sigma_{aa^\prime}+\sigma_{aa^{\prime\prime}}-\sigma_{a^\prime a^{\prime\prime}}\}$.
After $J-2$ such steps we conclude that all the arcs on the boundary of the branch satisfy a bound of the form \eqref{lpm}. The same procedure can be applied to all the branch that compose the boundary of $E$. This concludes the proof.
\end{proof}
Continuing with the proof of Proposition \ref{LB3} we observe that an
 important property of $\hat{\mathscr{G}}$ is the following: if $x\in\gamma$, $\gamma$ one of the arcs in $\hat{\mathscr{G}}$, and the curvature of $\gamma$ at $x$ is different from zero, then $F_x$ and the center of curvature of $\gamma$ at $x$ lie on opposite sides with respect to $x$ (see Figure \ref{opps}).

It follows that, if $\eta$ is an subarc of one of the $n_s$ arcs in $\hat{\mathscr{G}}$ and
\[F_x\cap F_y=\emptyset,\quad x\in\eta, y\in\hat{\mathscr{G}}\setminus\eta,\] then we have:
\[\begin{split}
&\vert\cup_{x\in\eta}F_x\vert\geq\int_{x\in\eta}\vert F_x\vert dx,\\
&J^\epsilon_{\cup_{x\in\eta}F_x}(u^\epsilon)\geq\int_\gamma J_x dx,
\end{split}\]
where  $J_x$ is the one-dimensional energy of $u^\epsilon\vert_{F_x}$.
 \begin{figure}
  \begin{center}
    \begin{tikzpicture}[scale=.55]
\draw[] (-.5,-.86602)--(.5,.86602);
\draw[] (.5,.86602) arc  [radius=1.73205, start angle=150, end angle= 120];
\draw[] (-.5,-.86602) arc  [radius=1.73205, start angle=330, end angle= 300];
\draw[] (1.13398,1.5)--(2,2);
\draw[] (-1.13398,-1.5)--(-2,-2);
\draw[red,  line width=2.7] (.56495,.828525) arc  [radius=1.73205, start angle=150, end angle= 90];
\draw[blue,  line width=2.7] (-.56495,-.828525) arc  [radius=1.65705, start angle=330, end angle= 270];
\draw[] (-.86602,.5)--(.86602,-.5);
\draw[] (.77525,1.22474)--(.06815,1.93184);
\draw[] (-.77525,-1.22474)--(-.06815,-1.93184);
\node[right] at (-.86602,.5) {$F_x$};
\node[left] at (.35,1.53184) {$F_x$};
\node[right] at (-.35,-1.53184) {$F_x$};
\node[] at (.2,0) {$x$};
\node[] at (.77525,1.42474) {$x$};
\node[] at (-.77525,-1.42474) {$x$};
\draw[->] (-.7,-.4)--(-.3,-.4);
\node[left] at (-.7,-.4) {$\gamma$};
\node[left] at (-1.5,-1) {$\Omega_a$};
\node[right] at (1.5,1) {$\Omega_{a^\prime}$};
 \end{tikzpicture}
 \end{center}
\caption{If the curvature of $\gamma$ at $x\in\gamma$ is non zero, $F_x$ and the center of curvature lie on opposite sides with respect to $x$.}
\label{opps}
\end{figure}
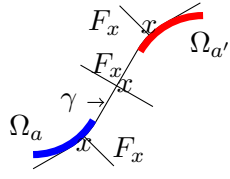
\vskip.1cm
 Define
\[\begin{split}
&g^*=\{x\in\hat{\mathscr{G}}: F_x\cap F_y\neq\emptyset,;\text{for some}\;y\in\hat{\mathscr{G}},\;y\neq x\},\\
&g^-=\{x\in\hat{\mathscr{G}}\setminus g^*:\vert F_x\vert<\epsilon^\beta\},\\
&g^+=\{x\in\hat{\mathscr{G}}\setminus g^*:\vert F_x\vert\geq\epsilon^\beta\},\\\\
&\mathscr{I}^\pm=\{y: y\in F_x,\;\text{for some}\;x\in g^\pm\}.\end{split}\]
\noindent
Lemma \ref{nearBra} and the fact $n_b$ is a finite number imply
\begin{equation}
\vert g^*\vert\leq C\epsilon^{\frac{1}{2}-\alpha}.
\label{g*}
\end{equation}
\noindent
From \eqref{|I|} we obtain
\begin{equation}
\begin{split}
& C\epsilon^{1-2\alpha}\geq\vert\mathscr{I}^+\vert\geq\int_{g^+}\vert F_x\vert dx\geq\epsilon^\beta\vert g^+\vert,\\
&\Rightarrow \vert g^+\vert\leq C\epsilon^{1-2\alpha-\beta}.
\end{split}
\label{g+}
\end{equation}
\noindent
A standard computation, see for instance Lemma 2.3 in \cite{AF+}, implies that,
for $x\in\gamma\cap g^-$ it results
\[J_x\geq\sigma_{a_\gamma a_\gamma^\prime}-C\delta^2\geq\sigma_{a_\gamma a_\gamma^\prime}-C\epsilon^{2\alpha},\]
and therefore
\begin{equation}
\begin{split}
&J_{\mathscr{I}^-}^\epsilon(u^\epsilon)\geq\sum_{\gamma\in\hat{G}}(\sigma_{a_\gamma a_\gamma^\prime}-C\delta^2)(\vert\gamma\vert-\vert g^*\vert-\vert g^+\vert)\\
&\geq\mathscr{F}(\hat{\mathscr{G}})-C(\vert g^*\vert+\vert g^+\vert+\delta^2),\\
&\geq\mathscr{F}(\hat{\mathscr{G}})-C(\epsilon^{\frac{1}{2}-\alpha}+\epsilon^{1-2\alpha-\beta}+\epsilon^{2\alpha}),
\end{split}
\label{LBfirst}
\end{equation}
where we have used \eqref{g*},\eqref{g+} and $\delta=\epsilon^\alpha$.
From \eqref{LBfirst}, $\alpha=\frac{1}{6}$ and $\beta=\frac{1}{3}$ we obtain
\[J_\Omega^\epsilon(u^\epsilon)\geq\mathscr{F}(\hat{\mathscr{G}})-C\epsilon^\frac{1}{3}.\]
This concludes the proof.
\end{proof}
\begin{proposition}
\label{dhatg-gf}
There is a constant $C_2>0$ such that
\[d(\hat{\mathscr{G}},\mathscr{G}_f)\leq C_2\epsilon^\frac{1}{6}.\]
\end{proposition}
\begin{proof}
From Proposition \ref{LB3} and \eqref{UB} with $e_\epsilon\leq C\epsilon\vert\ln{\epsilon}\vert^2$ that follows from Lemma \ref{UBerr} we obtain
\[\mathscr{F}(\hat{\mathscr{G}})-\mathscr{F}(\mathscr{G}_f)\leq C\epsilon^\frac{1}{3}.\]
This inequality and \eqref{gprime-g} with $e_\epsilon\leq C\epsilon^\frac{1}{3}$ (see Lemma \ref{gf-gf}) imply
\begin{equation}
\mathscr{F}(\hat{\mathscr{G}}^\prime)-\mathscr{F}(\mathscr{G}_f)\leq C\epsilon^\frac{1}{3},
\label{Sameend}
\end{equation}
where $\hat{\mathscr{G}}^\prime$ has the same end points as $\mathscr{G}_f$ and satisfies
\begin{equation}
d(\hat{\mathscr{G}}^\prime,\hat{\mathscr{G}})\leq C\epsilon\epsilon^\frac{1}{3}.
\label{dgprime-g}
\end{equation}
From \eqref{Sameend} and \eqref{GIn} we obtain
\[d(\hat{\mathscr{G}}^\prime,\mathscr{G}_f)\leq C\epsilon^\frac{1}{6},\]
which together with \eqref{dgprime-g} conclude the proof.
\end{proof}
We now show that Proposition \ref{dhatg-gf} and the minimality of $u^\epsilon$ imply the exponential estimate in Theorem \ref{fine}. Note that the definition \eqref{diffInt} of $\mathscr{I}$ and $\delta=\epsilon^\frac{1}{6}$
imply
\[\vert u^\epsilon(x)-a^\prime\vert\geq \epsilon^\frac{1}{6},\; \text{for}\; x\in\hat{S}_a,\;\;a^\prime\in A\setminus\{a\}.\]
If $S_{a,f}$ has positive measure, this and Proposition \ref{dhatg-gf} imply
\[\vert u^\epsilon(x)-a^\prime\vert\geq \epsilon^\frac{1}{6},\; \text{for}\; x\in S_{a,f},\;d(x,\mathscr{G}_f)\geq C_2\epsilon^\frac{1}{6},\;a^\prime\in A\setminus\{a\}.\]
Moreover we have
\begin{equation}
\begin{split}
& \vert u^\epsilon(x)-a\vert\leq \epsilon^\frac{1}{6},\; \text{for}\; x\in\Omega_a\cap S_{a,f},\\
&\vert\Omega_a\cap S_{a,f}\vert\geq\vert S_{a,f}\vert-C\epsilon^\frac{1}{6}.
\end{split}
\label{OmegaSaf}
\end{equation}
To see this we observe that $\Omega_a\subset\hat{S}_a$ and $\vert\mathscr{I}\vert\leq C\epsilon^\frac{2}{3}$ that follows from \eqref{|I|} with $\delta=\epsilon^\frac{1}{6}$ imply
\[\vert\Omega_a\vert\geq\vert\hat{S}_a\vert-C\epsilon^\frac{2}{3}.\]
On the other hand Proposition \ref{dhatg-gf} implies
\[\vert\hat{S}_a\cap S_{a,f}\vert\geq\vert S_{a,f}\vert-C\epsilon^\frac{1}{6},\]
and \eqref{OmegaSaf}$_2$ follows. Equation \eqref{OmegaSaf}$_1$ is just the definition of $\Omega_a$.

Let $\delta_0>0$ a number independent from $\epsilon$, for instance the number $\delta_0$ in \eqref{increas}.
\begin{proposition}
\label{<d}
There exists a constant $C_0>0$ such that
\[x\in S_{a,f}\;\;\text{and}\;\;d(x,\mathscr{G}_f)\geq C_0\epsilon^\frac{1}{6}\quad
\Rightarrow\quad\vert u^\epsilon(x)-a\vert\leq\delta_0.\]
\end{proposition}
\begin{proof}
From the proof of Proposition \ref{LB3} we see that the lower bound in \eqref{Lbequ} accounts only for the energy of $u^\epsilon$ in the set $\cup_{x\in\hat{\mathscr{G}}}F_x$. From this, $\vert F_x\vert\leq\epsilon^\frac{1}{3}$ and Proposition \ref{dhatg-gf} it follows that there is a constant $C_3>0$ such that the lower bound in Proposition \ref{LB3} does not consider the energy of $u^\epsilon$ in the set $S_{a,f}(C_3)$ where
\[S_{a,f}(c)=\{x\in S_{a,f}: d(x,\mathscr{G}_f)\geq c\epsilon^\frac{1}{6}\}.\]

Set $C_0=2C_3$ and assume that $x_0\in S_{a,f}$ satisfies $d(x,\mathscr{G}_f)\geq C_0\epsilon^\frac{1}{6}$ and
\[\min_{a^\prime\in A}\vert u^\epsilon(x_0)-a^\prime\vert\geq\delta_0.\]
Then, the bound \eqref{H1bound}  implies
\begin{equation}
\min_{a^\prime\in A}\vert u^\epsilon(x)-a^\prime\vert\geq\frac{\delta_0}{2},\;\;\text{for}\;\;x\in B_{\epsilon\frac{\delta_0}{2M^\prime}}(x_0).
\label{ContrHyp}
\end{equation}
Since $u^\epsilon$ is a minimizer of $J_\Omega^\epsilon(u)$, the map $v(y)=u^\epsilon(x_0+\epsilon y)$, $x_0+\epsilon y\in S_{a,f}(C_3)$ is a minimizer of $\int_{\Sigma_{a,f}(C)}(\frac{\vert\nabla v\vert^2}{2}+W(v))dy $, $\Sigma_{a,f}(C_3)=\{y: x_0+\epsilon y\in S_{a,f}(C_3)\}$ and \eqref{ContrHyp} implies
\[\min_{a^\prime\in A}\vert v(y)-a^\prime\vert\geq\frac{\delta_0}{2},\;\;\text{for}\;\;y\in B_{\frac{\delta_0}{2M^\prime}}(0).\]
From this and the density estimate derived by Caffarelli and Cordoba \cite{CC} that we use in the vector version as stated in Theorem 5.2 in \cite{afs} it follows
\[\vert B_r(0)\cap\{y:\min_{a^\prime\in A}\vert v(y)-a^\prime\vert\geq\frac{\delta_0}{2}\}\vert\geq C^*r^2,\]
as long as $B_r(0)\subset \Sigma_{a,f}(C_3)$. This estimate in terms of $u^\epsilon$ becomes
\begin{equation}
\vert B_{\epsilon r}(x_0)\cap\{x:\min_{a^\prime\in A}\vert u^\epsilon(x)-a^\prime\vert\geq\frac{\delta_0}{2}\}\vert\geq C^*(\epsilon r)^2,\;\;B_{\epsilon r}(x_0)\subset S_{a,f}(C_3).
\label{at-eps}
\end{equation}
This and \eqref{cW} imply
\begin{equation}
J^\epsilon_{ B_{\epsilon r}(x_0)}(u^\epsilon)\geq\frac{1}{\epsilon}\int_{ B_{\epsilon r}(x_0)}W(u^\epsilon)dx\geq \frac{1}{2}c_W\delta_0^2 C^*\epsilon r^2,\;\;B_{\epsilon r}(x_0)\subset S_{a,f}(C_3).
\label{JB}
\end{equation}
We fix $r=r_\epsilon$ by setting
\[\frac{1}{2}c_W\delta_0^2 C^*\epsilon r^2=2C_1\epsilon^\frac{1}{3},\]
where $C_1$ is the constant in Proposition \ref{LB3}. Then from  Proposition \ref{LB3} and \eqref{JB} we obtain
\[\begin{split}
&\epsilon r\leq C\epsilon^\frac{2}{3},\\
&J^\epsilon_{ B_{\epsilon r}(x_0)}\geq \mathscr{G}_f+2C_1\epsilon^\frac{1}{3}.
\end{split}\]
The first of these equation shows that, for small $\epsilon>0$ and $x_0\in S_{a,f}(C_0)$, $B_{\epsilon r}(x_0)\subset S_{a,f}(C_3)$. The second equation contradicts the upper bound \eqref{LB} where on the basis of Lemma \ref{UBerr} we can assume $e_\epsilon\leq C\epsilon\vert\ln{\epsilon}\vert^2$. This contradiction shows that,
for each $x_0\in S_{a,f}(C_0)$, there is $a_0\in A$ such that $\vert u^\epsilon(x_0)-a_0\vert<\delta_0$. the continuity of $u^\epsilon$ yields that $a_0$ is constant in $ S_{a,f}(C_0)$. This and \eqref{OmegaSaf} imply that actually
\[a_0=a,\;\;x\in S_{a,f}(C_0).\]
The proof is complete.
\end{proof}
Once the inequality in Proposition \ref{<d} is established the exponential estimate in Theorem \ref{fine} follows by linear elliptic theory, see for example Lemma 4.5 in \cite{afs}. This concludes the proof of Theorem \ref{fine}.
\subsection{Application of Theorem \ref{fine}.}
\label{app}
Theorem \ref{fine} says that, under suitable assumptions on $W$ and $v_0^\epsilon$ and provided condition $h_3$ is satisfied, the description of the fine structure of a minimizer $u^\epsilon$ of problem \eqref{min} reduces to the characterization of the free minimizer $\mathscr{G}_f\in\bar{\mathcal{G}}_f$ given by Corollary \ref{cor} or of the set $\mathscr{H}_f\subset\bar{\mathcal{G}}_f$ of such minimizers if $\mathscr{G}_f$ is not unique. Since, by assumption the arcs $I_a$, $a\in\tilde{A}$ are small in $\epsilon$, the determination of $\mathscr{G}_f$ is essentially a geometric problem.
We observe that the search of $\mathscr{G}_f$ is not exactly equivalent to the problem considered in $\cite{N}$ of determining the optimal partitioning of $\Omega$ in $N$ parts under prescribed boundary conditions.
Indeed there are a number of geometric properties which are basic for the construction of $\mathscr{G}_f$. First of all we recall from Theorem \ref{network} that  $\mathscr{G}_f$ has $n_b=2(N-1)-\tilde{N}$ branching points and is composed of $n_s=3(N-1)-\tilde{N}$ arcs where $N=\sharp A$ and $\tilde{N}=\sharp\tilde{A}$. For $N>2$ it may be also useful to consider the subnetwork $G_f\subset\mathscr{G}_f$ consisting of the arcs of $\mathscr{G}_f$ which have both extremes coinciding with branching points of $\mathscr{G}_f$. $G_f$ is connected and is composed of $\nu_s=3(N-1)-2\tilde{N}$ arcs and has the property that the set of the extreme of the arcs in $G_f$ coincides with the set of the $n_b$ branching points of  $\mathscr{G}_f$. Moreover if, as we assume, the sets $S_{a,f}$, $a\in A$, are connected, then
\begin{equation}
G_f=\cup_{a\in A\setminus\tilde{A}}\partial S_{a,f}.
\label{unotilde}
\end{equation}
Another important property of  $\mathscr{G}_f$ concerns its behavior in a neighborhood of an isolated branching point
 $p\in\Omega$. If $S_{a_i,f}$, $i=1,2,3$ are the set of the decomposition of $\bar{\Omega}$ defined by  $\mathscr{G}_f$ that meet at $p$ and $\gamma_{a_ia_j}$ is the arc that originates in $p$ and lies on the common boundary of
  $S_{a_i,f}$ and  $S_{a_j,f}$, then
\begin{equation}
\begin{split}
&\sigma_{a_1a_2}\tau_{a_1a_2}+\sigma_{a_2a_3}\tau_{a_2a_3}+\sigma_{a_3a_1}\tau_{a_3a_1}=0,\\
&\Rightarrow\quad\;\;\frac{\sigma_{a_1a_2}}{\sin{\alpha_3}}=\frac{\sigma_{a_2a_3}}{\sin{\alpha_1}}=\frac{\sigma_{a_3a_1}}{\sin{\alpha_2}}.
\end{split}
\label{eq-sig}
\end{equation}
where $\tau_{a_ia_j}$ is the unit tangent to  $\gamma_{a_ia_j}$ at $p$ and $\alpha_1$ is the angle between $\tau_{a_3a_1}$ and $\tau_{a_1a_2}$, $\alpha_2$ is the angle between $\tau_{a_1a_2}$ and $\tau_{a_2a_3}$, $\alpha_3$ is the angle between $\tau_{a_2a_3}$ and $\tau_{a_3a_1}$.

Finally we mention that, if $\gamma\in\mathscr{G}_f$ is an arc with extremes $p,p^\prime$ and $\mathrm{sg}(p,p^\prime)\subset\Omega$, then $\gamma$ coincides with $\mathrm{sg}(p,p^\prime)$. In particular, if $\Omega$ is convex, all the arcs in $\mathscr{G}_f$ with positive measure coincide with segments.

Our first application of Theorem \ref{fine} concerns the case where there exists $\sigma>0$ such that
\begin{equation}
\sigma_{aa^\prime}=\sigma,\;\;a,a^\prime\in A,\;a\neq a^\prime.
\label{sigma=}
\end{equation}
An example of potential that satisfies \eqref{sigma=} and has the property that any $a, a^\prime\in A$, $a\neq a^\prime$ are connected is given by
\[W(z)=\Pi_{a\in A}\vert z-a\vert^2,\;\;z\in\R^m,\;m\geq\sharp A-1,\]
where $A$ is the set of the vertices of a regular hypertetrahedron.
Note that from \eqref{sigma=} and \eqref{eq-sig} it follows
\begin{equation}
\alpha_i=\frac{2}{3}\pi,\;\; i=1,2,3.
\label{alpha=}
\end{equation}
\begin{proposition}
\label{only-boun}
Assume \eqref{sigma=} and, for $a\in A$, let $\mathscr{C}_a$ be the circuit defined in Section \ref{optimal}. Then
\[\vert\mathscr{C}_a\vert=0,\;\;\text{for}\;a\in A\setminus\tilde{A}.\] In particular $S_{a,f}$ has zero measure for each $a\in A\setminus\tilde{A}$.
\end{proposition}
\begin{proof}
Assume $\vert\mathscr{C}_a\vert>0$ for some $a\in A\setminus\tilde{A}$. Then there exists an arc $\gamma\in\mathscr{G}_f\cap\mathscr{C}_a$ with length $\ell>0$. Let $a^\prime\in A$ be determined by the condition that
\begin{equation}
\gamma\cap\partial S_{a,f}\cap\partial S_{a^\prime,f}\neq\emptyset.
\label{non-uniq}
\end{equation}and let $p_i$, $i=1,2$ be the extremes of $\gamma$. We assume that $\gamma$ is the unique arc satisfying these properties and that, if $r>0$ is sufficiently small, then $\partial (B_r(p_1)\cap S_{a,f})$ is a Jordan curve  composed by $\gamma\cap B_r(p_1)$, $\partial B_r (p_1)
\cap S_{a,f}$ and $\gamma_1\cap B_r(p_1)$ where $\gamma_1\in\mathscr{G}_f$ is the arc originating in $p_1$. More general situation like, for example the case where two or more arcs satisfy \eqref{non-uniq} with the same $a^\prime\in A$ can be discussing via similar arguments.
Let $q\in\gamma$ and $q_1\in\gamma_1$ be the extremes of the arc $\partial B_r (p_1)
\cap S_{a,f}$ and let $\tilde{\gamma}\subset\gamma$ the arc with extremes $q$ and $p_2$. We define a network $\tilde{\mathscr{G}}\in\bar{\mathcal{G}}_f$ with corresponding decomposition $\tilde{S}_a$, $a\in A$ as follows. $\tilde{\mathscr{G}}$ is obtained from $\mathscr{G}_f$ by removing the arc $\gamma$ and by adding the arc $\eta$ union of $\gamma_1\cap B_r(p_1)$ and $\partial B_r(p_1)\cap S_{a,f}$. Moreover we set $\tilde{S}_a=S_{a,f}\cap B_r(p_1)$, $\tilde{S}_{a^\prime}=(S_{a,f}\cup S_{a^\prime,f}\cup\tilde{\gamma})\setminus\tilde{S}_a$. We do not change the set $S_{a^{\prime\prime},f}$ for $a^{\prime\prime}\neq a, a^\prime$. By construction $\tilde{\mathscr{G}}$ and $\mathscr{G}_f$ have the same number of arcs and of branching points. Moreover, if $r>0$ is sufficiently small, it results
\[\mathscr{F}(\tilde{\mathscr{G}})-\mathscr{F}(\mathscr{G}_f)=\sigma(\vert\eta\vert-\ell)<0,\]
in contradiction with the minimality of $\mathscr{G}_f$. This concludes the proof.
\end{proof}

In view of Proposition \ref{only-boun} we assume
\[\tilde{N}=N.\]
It follows $n_b=N-2$, $n_s=2N-3$ and $\nu_s=N-3$. In the table that follows we list the values of $n_n,n_s$ and $\nu_s$ for $N=2,\ldots, 6$
\vskip.3cm

\begin{tabular}{c|c|c|c|c|c|}
N & 2 & 3 & 4 & 5 & 6\\
\hline
$n_b$ & 0 & 1 & 2& 3 & 4 \\
\hline
$n_s$ & 1 & 3 & 5& 7 & 9\\
\hline
$\nu_s$ &  & 0 & 1& 2 & 3
\end{tabular}
\vskip.3cm
We also assume that $\Omega$ is convex and that the boundary datum $v_0^\epsilon$ divides $\partial\Omega$ in $N$ arcs determined by the vertices $q_i\in\partial\Omega$, $i=1,\ldots,N$ of a  closed  polygon $\mathcal{P}_N$ with $N$ sides. Here we treat the small arcs $I_a$, $a\in A$, as points and identify the points $q_i$ with the end points of $\mathscr{G}_f$.

Note that, if $\eta\subset\mathscr{G}_f$ is an arc, not necessarily the union of some of the $n_s$ arcs of $\mathscr{G}_f$, with both extremes $q,q^\prime$ on the same side of $\mathcal{P}_N$ and $\eta\setminus\mathcal{P}_N$ is nonempty, then the length of the segment $\mathrm{sg}[q,q^\prime]$ is less than $\vert\eta\vert$. This implies
\begin{equation}
\mathscr{G}_f\subset\mathcal{P}_N.
\label{inP}
\end{equation}
 We also observe that the polygonal $\mathrm{Pol}_N$ obtained by removing one of the sides of $\mathcal{P}_N$, say $\mathrm{sg}(q_1,q_2)$, from $\partial\mathcal{P}_N$ can be interpreted as an element of $\bar{\mathcal{G}}_f$. Indeed the $n_b=N-2$ branching points can be identified with the vertices of $\mathrm{Pol}_N$ different from $q_1$ and $q_2$, $N-1$ of the $n_s=2N-3$ arcs coincide with the sides of $\mathrm{Pol}_N$. The remaining $N-2$ arcs have zero length and are the arcs that connect each branching point with the corresponding vertex of $\mathrm{Pol}_N$. Note that this implies the upper bound
\begin{equation}
\mathscr{F}(\mathscr{G}_f)\leq\sigma\pi D,
\label{UBF}
\end{equation}
where $D$ is the diameter of $\Omega$.

Let $q_i$ be one of the branching points of  $\mathrm{Pol}_N$ and denote $\theta_i$ the angle between the sides of   $\mathrm{Pol}_N$ that join at $q_i$. Let $\mathscr{G}^\prime\in\bar{\mathcal{G}}_f$ be the network obtained from  $\mathrm{Pol}_N$ via a small displacement of $q_i$ into a branching point $q_i^\prime\in\mathcal{P}_N\setminus\{q_i\}$. A classical argument shows that
\[\begin{split}
& \theta_i\geq\frac{2}{3}\pi\quad\Rightarrow\quad\mathscr{F}(\mathscr{G}^\prime)>\mathscr{F}(\mathrm{Pol}_N),
\:\:\forall\;q_i^\prime,\\
& \theta_i<\frac{2}{3}\pi\quad\Rightarrow\quad\exists\, q_i^\prime\;\text{such that}\;\;\mathscr{F}(\mathscr{G}^\prime)<\mathscr{F}(\mathrm{Pol}_N).
\end{split}\]
This and similar arguments show that
\begin{equation}
\min\theta_i\geq\frac{2}{3}\pi,
\label{minphi}
\end{equation}
implies that $\mathrm{Pol}_N$ is a nondegenerate local minimizer of $\mathscr{F}$ on $\bar{\mathcal{G}}_f$.

Now we further restrict $\Omega$ to be a ball $B$ and the boundary datum to be such that $\mathcal{P}_N$ is a regular polygon. In this case it results  $\theta_i=\phi_N=(1-\frac{2}{N})\pi$ and \eqref{minphi} is satisfied for only for $N\geq 6$ and the above discussion implies that, for $N=3,4,5$ the branching points belong to the interior of $\mathcal{P}_N$. In a similar way we see that the branching points are also isolated. From these observations and the fact that $G_f$ has, for $N=3,4,5$, $\nu_s=0,1,2$ segments respectively.

 On the basis of this information and using the fact that the network $G_f$ has a simple well determined structure: $G_f$ consists of $\nu_s=0,1,2$ segments for $N=3,4,5$ we conclude that the shape of  $\mathscr{G}_f$ is uniquely determined as shown in Figure \ref{345} where $G_f$ is highlighted in black.

\begin{figure}
  \begin{center}
\begin{tikzpicture}[scale=2]
\draw[thin] (0,0) circle [radius=1];
\draw[thin] (3,0) circle [radius=1];
\draw[thin] (-3,0) circle [radius=1];
\draw[blue] (-3,0)--(-2.13397,-.5);
\draw[blue] (-3,0)--(-3.866,-.5);
\draw[blue] (-3,0)--(-3,1);
\path[fill=red] (-3,0) circle [radius=.03];
\draw[blue] (0,0)--(0,.29886);
\draw[blue] (0,0)--(0,-.29886);
\draw[blue] (0,.29886)--(.707107,.707107);
\draw[blue] (0,.29886)--(-.707107,.707107);
\draw[blue] (0,-.29886)--(.707107,-.707107);
\draw[blue] (0,-.29886)--(-.707107,-.707107);
\path[fill=red] (0,.29886) circle [radius=.03];
\path[fill=red] (0,-.29886) circle [radius=.03];
\draw[blue] (3.5878,-.809)--(3.5878,.0993);
\draw[blue] (3.5878,.0993)--(3.951,.309);
\draw[blue] (2.4122,-.809)--(2.4122,.0993);
\draw[blue] (2.4122,.0993)--(2.049,.309);
\draw[blue] (3.5878,.0993)--(3,.438666);
\draw[blue] (2.4122,.0993)--(3,.438666);
\draw[blue] (3,.438666)--(3,1);
\path[fill=red] (3.5878,.0993) circle [radius=.03];
\path[fill=red] (3,.438666) circle [radius=.03];
\path[fill=red] (2.4122,.0993) circle [radius=.03];
\node[below] at (-3,-1) {$N=3$};
\node[below] at (0,-1) {$N=4$};
\node[below] at (3,-1) {$N=5$};
\node[right] at (.165,0) {$G_f$};
\node[below] at (3,.150666) {$G_f$};
\draw[->] (.15,0)--(.02,0);
\draw[->] (3,.150666)--(3.25,.238666);
\draw[->] (3,.150666)--(2.75,.238666);
\end{tikzpicture}
\end{center}
\caption{The minimizer $\mathscr{G}_f$ for $N=3,4,5$ and the network $G_f$ for $N=4,5$.
The red spots correspond to the branching points. }
\label{345}
\end{figure}
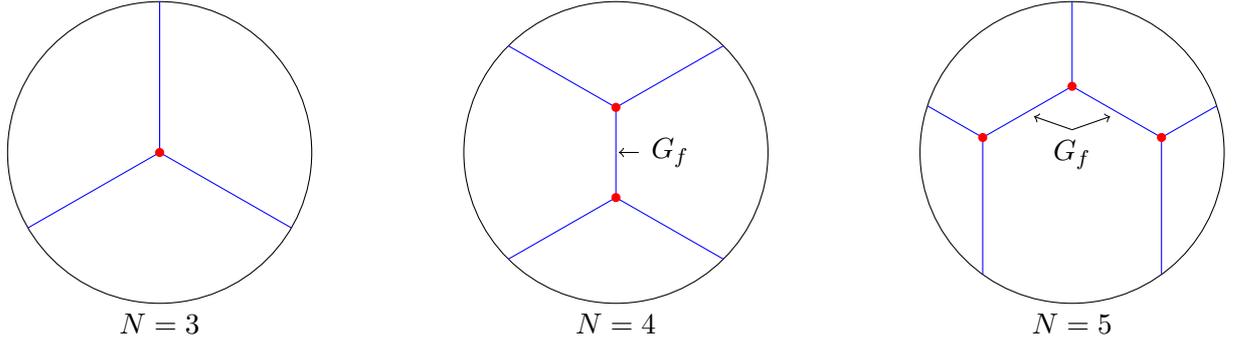
We observe that, for $N=3$, $\mathscr{G}_f$ is unique, for $N=4$, there are two minimizers, the one shown in Figure \ref{345} and the one obtained by a rotation of $\frac{\pi}{2}$  of $\mathscr{G}_f$ around the center of $B$, for $N=5$ there are five minimizers obtained from  $\mathscr{G}_f$ by successive rotations of $\frac{2\pi}{5}$.

For $N=6$, we have $\theta_i=\phi_6=\frac{2}{3}\pi$. We claim that this implies that either the four branching points of  $\mathscr{G}_f$ lie in the interior of $\mathcal{P}_6$ or all coincide with vertices of  $\mathcal{P}_6$ and we have the polygonal $\mathrm{Pol}_6$. Indeed, if this in not the case, there are two branching points $q$ and $p$, extremes of the segment $\mathrm{sg}[q,p]$, and such that $q$ is one of the vertices of $\mathcal{P}_6$ while $p$ is in the interior of  $\mathcal{P}_6$. It follows that the angle $\phi$ between  $\mathrm{sg}[q,p]$ and any other segment that connects $q$ to some $p_1\in\mathcal{P}_6\setminus\{q\}$ is $<\frac{2}{3}\pi$. This, as discussed above, implies the existence of a small perturbation which reduces $\mathscr{F}$ and proves the claim.

The network $G_f$ consists of $\nu_6=3$ segments. Hence for $G_f$ there are two possible configurations: either the three segments have an extreme in common or form a polygonal with three sides. This and the above claim imply that there are exactly three well determined geometric structures that candidate to be a minimizer $\mathscr{G}_f$.
Namely the polygonal $\mathrm{Pol}_6$ and the two structure that  correspond to the two possible configuration of $G_f$ and have all the branching points in the interior of $\mathcal{P}_6$, see Figure \ref{N6}.
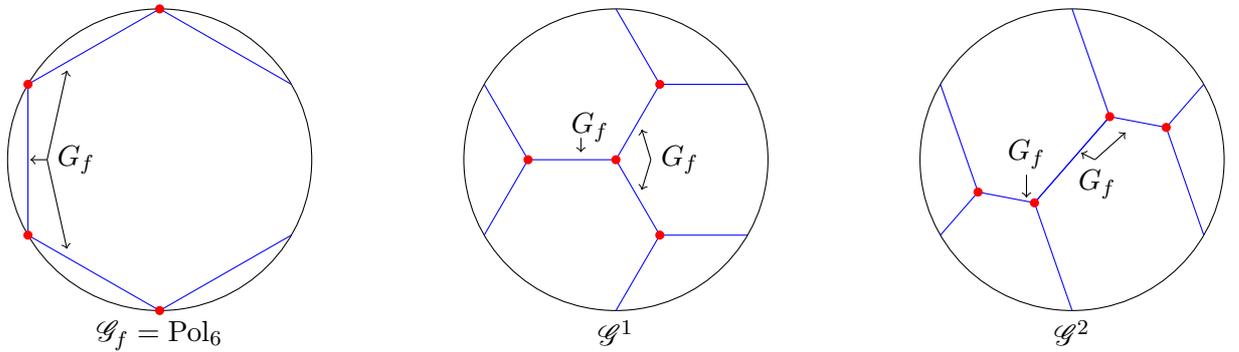
\begin{figure}
  \begin{center}
\begin{tikzpicture}[scale=2]
\draw[thin] (0,0) circle [radius=1];
\draw[thin] (3,0) circle [radius=1];
\draw[thin] (-3,0) circle [radius=1];
\draw[blue] (.28867,.5)--(0,0);
\draw[blue] (.28867,.5)--(0,1);
\draw[blue] (.28867,.5)--(.866,.5);
\draw[blue] (.28867,-.5)--(0,0);
\draw[blue] (.28867,-.5)--(0,-1);
\draw[blue] (.28867,-.5)--(.866,-.5);
\draw[blue] (-.57735,0)--(0,0);
\draw[blue] (-.57735,0)--(-.866,.5);
\draw[blue] (-.57735,0)--(-.866,-.5);
\draw[blue] (3.247436,.285714)--(2.752564,-.285714);
\draw[blue] (3.247436,.285714)--(3,0);
\draw[blue] (3.247436,.285714)--(3,1);
\draw[blue] (2.752564,-.285714)--(3,0);
\draw[blue] (2.752564,-.285714)--(3,-1);
\draw[blue] (3.247436,.285714)--(3.6185894,.214286);
\draw[blue] (3.6185894,.214286)--(3.866,.5);
\draw[blue] (3.6185894,.214286)--(3.866,-.5);
\draw[blue] (2.752564,-.285714)--(2.381406,-.214286);
\draw[blue] (2.38141,-.214286)--(2.134,-.5);
\draw[blue] (2.38141,-.214286)--(2.134,.5);
\draw[blue] (-3.866,.5)--(-3,1);
\draw[blue] (-3.866,-.5)--(-3,-1);
\draw[blue] (-3.866,-.5)--(-3.866,.5);
\draw[blue] (-2.134,-.5)--(-3,-1);
\draw[blue] (-2.134,.5)--(-3,1);
\node[below] at (-3,-1) {$\mathscr{G}_f=\mathrm{Pol}_6$};
\node[below] at (0,-1) {$\mathscr{G}^1$};
\node[below] at (3,-1) {$\mathscr{G}^2$};
\draw[->] (-3.74,0)--(-3.85,0);
\draw[->] (-3.74,0)--(-3.61,.59);
\draw[->] (-3.74,0)--(-3.61,-.59);
\node[right] at (-3.74,0) {$G_f$};
\draw[->] (.23,0)--(.17,.2);
\draw[->] (.23,0)--(.17,-.2);
\draw[->] (-.23,.145)--(-.23,.05);
\node[right] at (.23,0) {$G_f$};
\node[above] at  (-.17,.05) {$G_f$};
\draw[->] (3.15,0)--(3.06,.045);
\draw[->] (3.15,0)--(3.35,.18);
\draw[->] (2.7,-.1)--(2.7,-.25);
\node[below] at (3.165,0) {$G_f$};
\node[above] at  (2.7,-.125) {$G_f$};
\path[fill=red] (.28867,.5) circle [radius=.03];
\path[fill=red] (0,0) circle [radius=.03];
\path[fill=red] (.28867,-.5) circle [radius=.03];
\path[fill=red] (-.57735,0) circle [radius=.03];
\path[fill=red] (3.247436,.285714) circle [radius=.03];
\path[fill=red] (2.752564,-.285714) circle [radius=.03];
\path[fill=red] (3.6185894,.214286) circle [radius=.03];
\path[fill=red] (2.38141,-.214286) circle [radius=.03];
\path[fill=red] (-3,1) circle [radius=.03];
\path[fill=red] (-3,-1) circle [radius=.03];
\path[fill=red] (-3.866,.5) circle [radius=.03];
\path[fill=red] (-3.866,-.5) circle [radius=.03];
\end{tikzpicture}
\end{center}
\caption{The minimizer $\mathscr{G}_f=\mathrm{Pol}_6$, $\mathscr{G}^i$, $i=1,2$ and the set $G_f$.
The red spots correspond to the branching points. }
\label{N6}
\end{figure}

If $R$ is the radius of $B$, then we have $\mathscr{F}(\mathrm{Pol}_6)=5\sigma R$ while, if we let $\mathscr{G}^1$ and $\mathscr{G}^2$ the networks associated to the two different structures of $G_f$, with $\mathscr{G}^2$ the one that corresponds to the case where $G_f$ is a polygonal, it results
\[5\sigma R<\mathscr{F}(\mathscr{G}^1)<\mathscr{F}(\mathscr{G}^1),\]
and we concludes that $\mathscr{G}_f=\mathrm{Pol}_6$ and the same is true for all the six polygonal obtained by successive rotation of $\frac{\pi}{3}$.
\vskip.2cm
For $N>6$ the number of possible configurations of $G_f$ and the number of networks candidate to be a minimizer $\mathscr{G}_f$ given in Corollary \ref{cor} rapidly increases and the analysis becomes more and more involved and it is outside the scope of this paper. However the discussion of the case $N=6$ suggests the conjecture that $N=6$ is a critical value and that for $N\geq 6$ it results $\mathscr{G}_f=\mathrm{Pol}_N$.
\vskip.2cm
Our next application of Theorem \ref{fine} is to show that the structure of a minimizer $u^\epsilon$ of problem \eqref{min} inside the domain may depend strongly from the values of the surface tensions $\sigma_{aa^\prime}$. We show that $u^\epsilon$ may exhibit, in the interior of $\Omega$, one or more phases different from the ones imposed on $\partial\Omega$ by the boundary datum. We assume that $\Omega=B$ is a ball and that the potential $W$ has $N\geq 4$ zeros and is invariant under the symmetry group of the equilateral triangle. Moreover we assume that the boundary datum divides $\partial B$ in three equal arcs determined by the vertices of $\mathcal{P}_3$.

We start with the case $N=4$ where the symmetry of $W$ implies that $A=\{0,a,b,c\}$ where $a\neq 0$, $b=ra$, $c=r^2 a$ with $r$ the rotation of $\frac{2}{3}\pi$ around the origin (see Figure \ref{W4}).
\begin{figure}
  \begin{center}
\begin{tikzpicture}[scale=1]
\draw[] (-6,0) circle [radius=2];
\draw[] (0,0) circle [radius=2];
\draw[] (6,0) circle [radius=2];
\draw[fill=black,black] (-6,6) circle [radius=.045];
\draw[fill=black,black] (-6,8) circle [radius=.045];
\draw[fill=black,black] (-7.73205,5) circle [radius=.045];
\draw[fill=black,black] (-4.26795,5) circle [radius=.045];
\draw[] (-6,6)--(-6,8);
\draw[] (-6,6)--(-7.73205,5);
\draw[] (-6,6)--(-4.26795,5);
\draw[] (-7.73205,5) to [out=45,in=255] (-6,8);
\draw[] (-4.26795,5) to [out=165,in=15]  (-7.73205,5);
\draw[] (-6,8) to [out=285,in=135]  (-4.26795,5);
\draw[blue] (-6,1)--(-6,2);
\draw[blue] (-6,1)--(-6.866,-.5);
\draw[blue] (-6,1)--(-5.134,-.5);
\draw[blue] (-5.134,-.5)--(-6.866,-.5);
\draw[blue] (-7.73205,-1)--(-6.866,-.5);
\draw[blue] (-4.26795,-1)--(-5.134,-.5);
\draw[fill=red,red] (-6,1) circle [radius=.045];
\draw[fill=red,red] (-6.866,-.5) circle [radius=.045];
\draw[fill=red,red] (-5.134,-.5) circle [radius=.045];
\draw[blue] (0,0)--(0,2);
\draw[blue] (0,0)--(1.73205,-1);
\draw[blue] (0,0)--(-1.73205,-1);
\draw[fill=red,red] (0,0) circle [radius=.045];
\draw[blue] (6,2)--(4.26795,-1);
\draw[blue] (7.73205,-1)--(4.26795,-1);
\draw[blue] (7.73205,-1)--(6,2);
\draw[fill=red,red] (6,2) circle [radius=.045];
\draw[fill=red,red] (7.73205,-1) circle [radius=.045];
\draw[fill=red,red] (4.26795,-1) circle [radius=.045];
\node[] at (1.1,.5) {$a$};
\node[] at (-1.1,.5) {$b$};
\node[] at (0,-1.3) {$c$};
\node[] at (6,0) {$0$};
\node[] at (7.1,.5) {$a$};
\node[] at (4.9,.5) {$b$};
\node[] at (6,-1.3) {$c$};
\node[] at (-4.9,.5) {$a$};
\node[] at (-7.1,.5) {$b$};
\node[] at (-6,-1.3) {$c$};
\node[] at (-6,0) {$0$};
\node[right] at (-4.26795,5) {$a$};
\node[left] at (-7.73205,5) {$c$};
\node[above] at (-6,8) {$b$};
\node[below] at (-6,6) {$0$};
\node[below] at (-5.134,6.5) {$\sigma$};
\node[] at (-6.25,6.5) {$\sigma_0$};
\node[below] at (0,-2) {$\mathscr{G}_f^0,\;\sigma\in(0,\sqrt{3}\sigma_0)$};
\node[below] at (-6,-2) {$\mathscr{G}_f^\ell,\;\ell\in(0,R),\;\sigma=\sqrt{3}\sigma_0$};
\node[below] at (6,-2) {$\mathscr{G}_f^R,\;\sigma\in(\sqrt{3}\sigma_0,2\sigma_0)$};
\end{tikzpicture}
\end{center}
\caption{The potential $W$ with the four zeros and the connections. The minimizers $\mathscr{G}_f^\ell$, $\ell\in(0,R)$,  $\mathscr{G}_f^0$, $\sigma\in(0,\sqrt{3}\sigma_0)$, $\mathscr{G}_f^R$, $\sigma\in(\sqrt{3}\sigma_0,2\sigma_0)$.
}
\label{W4}
\end{figure}

We assume that $a,b$ and $c$ are the phases that appear on $\partial B$. General results, see for example Theorem 2.1 in \cite{afs},\cite{Rab},\cite{MS}, and the symmetry of $W$ imply that $0$ connects to $a,b,c$ with a common value $\sigma_0$ for the energy of these connections. As required for the validity of Theorem \ref{fine} we also assume that $a$ is connected to $b$ and $c$ and that $\sigma<2\sigma_0$, where
\begin{equation}
\sigma=\sigma_{ab}=\sigma_{bc}=\sigma_{ca}.
\label{someSig=}
\end{equation}
We first consider the possibility that $\mathscr{G}_f\in\mathcal{G}_f$ and then the case where  $\mathscr{G}_f\in\partial\mathcal{G}_f$. If $\mathscr{G}_f$ is in the interior of $\bar{\mathcal{G}}_f$, then is composed of $n_s=6$ arcs, actually segments, of positive length with $n_b=3$ branching points of triple junction type that lie in the interior of $B$. The set $G_f$ has $\nu_s=3$ sides and, since the set of extreme points of $G_f$ coincides with the set of the branching points of $\mathscr{G}_f$, has $3$ extreme points. It follows that  $G_f$ is the boundary of a triangle which, in view of \eqref{unotilde} is necessarily associated to the phase $0$. If $p$ is one of the vertices of $G_f$, there is a segment $\mathrm{sg}[p,q]\in\mathscr{G}_f$ where $q\in\partial B$ is one of the end points of $\mathscr{G}_f$. From \eqref{eq-sig} and \eqref{someSig=} it follows that  $\mathrm{sg}[p,q]$ forms equal angles with the two sides of $G_f$ that join at $p$. Let $\theta$ be the value of these angles and let $\phi=2\pi-2\theta$ be the angle between the two sides of $G_f$ that join at $p$. Since the same argument applies to all three vertices of $G_f$ we conclude that $G_f$ is the boundary of an equilateral triangle with center at the centre of $B$ and vertices on the radii that join the end points with the center of $B$. Moreover we have $\phi=\frac{1}{3}\pi$, $\theta=\frac{5}{6}\pi$ and  \eqref{eq-sig} implies
\begin{equation}
\sigma=\sqrt{3}\sigma_0.
\label{sigsig0}
\end{equation}
Under the previous assumptions on $W$ and on the boundary datum, this condition is necessary and sufficient for the existence of a minimizer $\mathscr{G}_f$ with the three branching points in the interior of $B$.
If $\ell\in(0,R)$ is the distance of the vertices of $G_f$ from the centre of $B$, we have from \eqref{sigsig0}
\[\mathscr{F}(\mathscr{G}_f)=3\sqrt{3}\ell\sigma_0+3(R-\ell)\sigma=3R\sigma.\]
That is: $\mathscr{F}(\mathscr{G}_f)$ is independent of $\ell\in(0,R)$ and we have a one-parameter family of minimizers.
 It follows that $\mathscr{G}_f$ is not isolated and Theorem \ref{fine} does not apply if  \eqref{sigsig0} holds. Hence we examine under what conditions there exists a minimizer  $\mathscr{G}_f\in\partial\mathcal{G}_f$. We first consider the configurations $\mathscr{G}_f^0$ and of $\mathscr{G}_f^R$ of $\mathscr{G}_f$ corresponding to $\ell=0$ and $\ell=R$. In the first case $G_f$ reduces to a point $O$, the center of $B$. $O$ coincides with the three branching points and with the three sides of $G_f$ of zero length. A local analysis around $O$ shows that $\sigma\in(0,\sqrt{3}\sigma_0)$ implies that $\mathscr{G}_f^0$ is an isolated minimizer that satisfies \eqref{GIn}. For  $\mathscr{G}_f^R$ the vertices of $G_f$ , the branching points, coincide with the end points and the three arcs with one extreme in a branching point and the other in an end point have zero length. A local analysis around the branching points shows that $\sigma\in(\sqrt{3}\sigma_0,2\sigma_0)$ implies that $\mathscr{G}_f^R$ is an isolated minimizer that satisfies \eqref{GIn}.
 It remains to consider the case where one or two of the vertices of $G_f$ coincide with the vertices of $\mathcal{P}_3$. Assume that two vertices $p_1$ and $p_2$ coincide with the vertices $q_1$ and $q_2$ of $\mathcal{P}_3$ and let $p$ and $q$, $p\neq q$, the third vertex of $G_f$ and the third vertex of $\mathcal{P}_3$ respectively. As before, from \eqref{eq-sig} and \eqref{someSig=}, we have that the angles determined by $\mathrm{sg}[p,q]$ and $\mathrm{sg}[p,q_i]$, $i=1,2$ are equal. Hence the triangle $p,p_1,p_2$ has the angles in $p_1$ and $p_2$ equal and thus is an isosceles triangle. If $\phi\in(\frac{1}{3}\pi,\pi)$ is the angle in $p$, the value $\psi$ of the angles in $p_1$ and $p_2$ is $\psi=\frac{1}{2}(\pi-\phi)<\frac{1}{3}\pi<\phi$. The existence of the branching point $p$ in the interior of $B$ and \eqref{eq-sig} imply that $\sigma=2\cos{\frac{\phi}{2}}\sigma_0$. Let $\mathscr{G}$ be the network defined by the branching points $p$ and $p_i=q_i$, $i=1,2$ and $\mathscr{G}^\prime$
  the network determined by the perturbed branching points $p^\prime=p$ and $p_i^\prime=q_i+h\tau_i$, $i=1,2$, where $h>0$ is a small number and $\tau_i\in\SF^1$ is directed as the bisectrix of the angle in $p_i$ and points inside $B$. Then, using also that $\sigma=2\cos{\frac{\phi}{2}}\sigma_0$, we obtain
  \begin{equation}
  \mathscr{F}(\mathscr{G}^\prime)-\mathscr{F}(\mathscr{G})=2h\sigma-4h\cos(\frac{\psi}{2})\sigma_0+o(h)
  =2\sigma h(1-\frac{\cos(\frac{1}{4}(\pi-\phi))}{\cos{\frac{\phi}{2}}})+o(h).
  \label{psi<phi}
  \end{equation}
  Since $\phi>\frac{1}{3}\pi$ implies $\frac{\cos(\frac{1}{4}(\pi-\phi))}{\cos{\frac{\phi}{2}}}>1$, \eqref{psi<phi} shows that $\mathscr{G}\neq\mathscr{G}_f$. In a similar way one checks that any $\mathscr{G}$ with a unique branching point that coincides with one vertices of $\mathcal{P}_3$ is not a minimizer. Therefore we conclude that, for $0<\sigma<\sqrt{3}\sigma_0$, $\mathscr{G}_f^0$ is the unique minimizer while $\mathscr{G}_f^R$ is the unique minimizer for $\sigma\in(\sqrt{3}\sigma_0, 2\sigma_0)$.
  Then Theorem \ref{fine} yields a detailed description of the fine structure of $u^\epsilon$ for $\sigma\neq\sqrt{3}\sigma_0$. This confirm the results in \cite{f4} where the existence of the solutions associated to $\mathscr{G}_f^0$ and  $\mathscr{G}_f^R$ by Theorem \ref{fine} is established in the class of symmetric solutions.
  
  Before moving to the next example, we can ask what could be the counterpart, both for the parabolic dynamics associated to $J_\Omega^\epsilon$ and for the geometric curvature flow, of the manifold of the equal-energetic networks described above. We expect that it corresponds to invariant manifolds. In particular, for the curvature flow, there should exit manifolds of self-similar solutions with the property the the triangle of the network shrinks to zero or expands to infinite depending on wether $\sigma<\sqrt{3}\sigma_0$ or $\sigma>\sqrt{3}\sigma_0$. For the general theory of network dynamics we refer to \cite{MNP}, see also 
 \cite{BHM} and \cite{CG}.
 \vskip.2cm  
  Under the standing assumptions on $W$ and the boundary datum $v_0^\epsilon$ we now assume that $W$ has $N=7$ zeros: $0,a,b=ra,c=r^2a$, $a^\prime,b^\prime=ra^\prime,c^\prime=r^2a^\prime$ for some $a\neq 0$, $a^\prime=-\lambda a$, $\lambda>0$ (see Figure \ref{N7}).
\begin{figure}
  \begin{center}
\begin{tikzpicture}[scale=.73]
\draw[fill=black,black] (0,0) circle [radius=.045];
\draw[fill=black,black] (0,3) circle [radius=.045];
\draw[fill=black,black] (2.598,-1.5) circle [radius=.045];
\draw[fill=black,black] (-2.598,-1.5) circle [radius=.045];
\draw[fill=black,black] (0,-.5) circle [radius=.045];
\draw[fill=black,black] (.433,.25) circle [radius=.045];
\draw[fill=black,black] (-.433,.25) circle [radius=.045];
\draw[]  (-2.598,-1.5) to [out=45,in=255] (0,3);
\draw[] (2.598,-1.5) to [out=165,in=15]  (-2.598,-1.5);
\draw[] (0,3) to [out=285,in=135]  (2.598,-1.5);
\draw[] (-2.598,-1.5) to [out=35,in=230] (-.433,.25);
\draw[] (0,3) to [out=265,in=70]  (-.433,.25);
\draw[] (2.598,-1.5) to [out=155,in=350]  (0,-.5);
\draw[](-2.598,-1.5) to [out=25,in=190]  (0,-.5);
\draw[] (2.598,-1.5) to [out=145,in=310]  (.433,.25);
\draw[] (0,3) to [out=275,in=110]  (.433,.25);
\draw[] (0,0)-- (0,-.5);
\draw[] (0,0)-- (.433,.25);
\draw[] (0,0)-- (-.433,.25);
\draw[] (-.433,.25) to [out=355,in=185](.433,.25);
\draw[] (0,-.5) to [out=65,in=235](.433,.25);
\draw[] (0,-.5) to [out=115,in=315](-.433,.25);
\node[right] at(2.598,-1.5) {$a$};
\node[left] at (-2.598,-1.5)  {$c$};
\node[above] at (0,3) {$b$};

\node[right] at (.433,.25) {$c^\prime$};
\node[left] at (-.433,.25)  {$a^\prime$};
\node[below] at (0,-.5) {$b^\prime$};
\end{tikzpicture}
\end{center}
\caption{A schematic illustration of the connections for $W$ with $N=7$.}
\label{N7}
\end{figure}
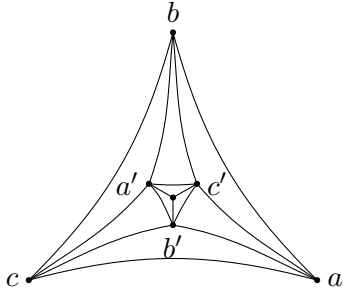
We may think that $W$ is obtained from the previous potential with $N=4$ via a bifurcation of $0$ into the four zeros $0,a^\prime,b^\prime,c^\prime$. 
Since $W$ is invariant under the group $Z_3$, the surface tensions of all the connections that lie on the same group orbit of  $Z_3$ have the same value and we set $\sigma=\sigma_{ab}$, $\sigma_0=\sigma_{ac^\prime}$,
 $\tau=\sigma_{a^\prime b^\prime}$,  $\tau_0=\sigma_{0a^\prime}$, $\sigma^\prime=\sigma_{aa^\prime}$  and $\sigma_{00}=\sigma_{0a}$ (the connections $a,a^\prime$ and $0,a$ are not indicated in Figure \ref{N7}). We assume:
 \begin{equation}
 \sigma^\prime>\sigma_0.
 \label{sigma-pr}
 \end{equation}
Note that $N=7$ and Theorem \ref{network} and the definition of $G_f$ imply $n_s=15$, $n_b=9$  $\nu_s=12$.
  To keep keep things as simple as possible, we restrict to the case of minimizers $u^\epsilon$ equivariant under $Z_3$ and therefore we only consider networks invariant under $Z_3$. We first analyze the case where all the $n_b=9$ branching points are isolated and lie in the interior of $B$. Let $\mathrm{sg}[p,q]$ be the segment that separates $S_{a,f}$ and $S_{b,f}$,  $p$ is a branching point and $q$ is one of the vertices of $\mathcal{P}_3$.
  Moreover $p$ belongs to the ray of $B$ determined by $q$. Let $2\psi$, $\psi\in(0,\frac{\pi}{6})$, be the angle between the segments $\mathrm{sg}[p,p_i]$, $i=1,2$, that join at $p$ with $\mathrm{sg}[p,q]$. The angle between  $\mathrm{sg}[p,q]$ and  $\mathrm{sg}[p,p_i]$, is equal to $\pi-\psi$, $i=1,2$. $G_f$ is composed of $\nu_s=9$ segments six of which are $\mathrm{sg}[p,p_i]$, $i=1,2$ and their images under $r$ and $r^2$. Hence there are two possibilities: $I$: the union of these six segments is the boundary of a star with three tip and we can assume that $p_1=rp_2$. $II$: $p_1\neq rp_2$ and the union of $\mathrm{sg}[p_1,p_2]$, $\mathrm{sg}[p_1,rp_2]$ and their images under $r$ and $r^2$ is the boundary of a hexagonal region (see Figure \ref{1and2}).
  In case $I$ $p_1=rp_2$ is the extreme of a segment $\mathrm{sg}[p_1,p_3]$ which forms an angle of $\frac{\pi}{6}$ with $\mathrm{sg}[p,q]$ and $p_3$ is a vertex of an equilateral triangle $T$ with center in the center of $B$. The three sides of this triangle together with $\mathrm{sg}[p_1,p_3]$ and its images under $r$ and $r^2$ complete the network $\mathscr{G}_f^I$ in case $I$. The networks  $\mathscr{G}_f^I$ and the network  $\mathscr{G}_f^{II}$ of case $II$ are illustrated in Figure \ref{1and2}.
  From equivariance it follows that in both cases the region that contains the center of $B$, a triangle in case $I$ and a hexagon in case $II$ must be identified with $S_0,f$. Moreover \eqref{sigma-pr} implies that, if  $S_{\tilde{a},f}$ is the set that joins at $p$ with $S_{a,f}$ and  $S_{b,f}$, then necessarily $\tilde{a}=c^\prime$ as indicated in Figure \ref{1and2}. 
  
  We assume that $B$ is the unit ball and set $d=1-\vert q-p\vert$ where $p,q$ are the vertices of $\mathrm{sg}[p,q]$. We denote by $2\ell$ the length of the side of the triangle $T$ in case $I$ and set $2\ell=\vert p_1-p_2\vert$  in case $II$ where $p_1,p_2$ are the vertices of $\mathrm{sg}[p_1,p_2]$ considered above. We observe that both in case $I$ and $II$, for each $\psi\in(0,\frac{\pi}{6})$ there is a two parameter family of networks of type  $\mathscr{G}_f^I$ and $\mathscr{G}_f^{II}$  with parameters $d\in(0,1)$ and $\ell\in(0,g(\psi)d)$, where we have set $g(\psi)=\frac{\sqrt{3}}{2}\frac{\sin\psi}{\cos(\frac{\pi}{6}-\psi)}$.
  
 \begin{figure}
  \begin{center}
\begin{tikzpicture}[scale=.8]
\draw[] (0,0) circle [radius=3];
\draw[] (10,0) circle [radius=3];
\draw[blue] (0,2.5)--(0,3);
\draw[blue](2.165,-1.25)-- (2.598,-1.5);
\draw[blue](-2.165,-1.25)-- (-2.598,-1.5);
\draw[blue] (10,2.5)--(10,3);
\draw[blue](12.165,-1.25)-- (12.598,-1.5);
\draw[blue](7.835,-1.25)-- (7.402,-1.5);
\draw[blue] (.433,.25)--(.6495,.375);
\draw[blue] (-.433,.25)--(-.6495,.375);
\draw[blue] (0,-.5)--(0,-.75);
\draw[blue] (.433,.25)--(-.433,.25);
\draw[blue] (0,-.5)--(.433,.25);
\draw[blue] (0,-.5)--(-.433,.25);
\draw[blue] (0,2.5)--(.6495,.375);
\draw[blue] (0,2.5)--(-.6495,.375);
\draw[blue] (2.165,-1.25)--(.6495,.375);
\draw[blue] (2.165,-1.25)--(0,-.75);
\draw[blue] (-2.165,-1.25)--(-.6495,.375);
\draw[blue] (-2.165,-1.25)--(0,-.75);
\draw[fill=red,red] (0,2.5) circle [radius=.045];
\draw[fill=red,red] (0,-.5) circle [radius=.045];
\draw[fill=red,red] (0,-.75) circle [radius=.045];
\draw[fill=red,red] (2.165,-1.25) circle [radius=.045];
\draw[fill=red,red] (-2.165,-1.25) circle [radius=.045];
\draw[fill=red,red] (.433,.25) circle [radius=.045];
\draw[fill=red,red] (-.433,.25) circle [radius=.045];
\draw[fill=red,red] (.6495,.375) circle [radius=.045];
\draw[fill=red,red] (-.6495,.375) circle [radius=.045];
\draw[dotted] (0,2.5)--(0,1);
\draw[dotted] (-.6495,.375)--(-1.9485,1.125);
\draw[blue] (10,2.5)--(10.45859,1);
\draw[blue] (10,2.5)--(9.54141,1);
\draw[blue](9.54141,1)--(10.45859,1);
\draw[blue] (7.835,-1.25)--(8.89046,-.10825);
\draw[blue] (8.89046,-.10825)--(9.36327,-.99247);
\draw[blue] (7.835,-1.25)--(9.36327,-.99247);
\draw[blue] (12.1650,-1.25)--(11.09532,-.10825);
\draw[blue] (12.1650,-1.25)--(10.63673,-.99247);
\draw[blue] (11.09532,-.10825)--(10.63673,-.99247);
\draw[blue] (10.45859,1)--(11.09532,-.10825);
\draw[blue] (9.54141,1)--(8.89046,-.10825);
\draw[blue] (9.36327,-.99247)--(10.63673,-.99247);
\draw[fill=red,red] (10,2.5) circle [radius=.045];
\draw[fill=red,red] (7.835,-1.25) circle [radius=.045];
\draw[fill=red,red] (12.1650,-1.25) circle [radius=.045];
\draw[fill=red,red] (10.63673,-.99247) circle [radius=.045];
\draw[fill=red,red] (9.36327,-.99247) circle [radius=.045];
\draw[fill=red,red] (8.89046,-.10825) circle [radius=.045];
\draw[fill=red,red] (9.54141,1) circle [radius=.045];
\draw[fill=red,red] (11.09532,-.10825) circle [radius=.045];
\draw[fill=red,red] (10.45859,1) circle [radius=.045];
\node[below,left] at (.1,1.150) {$\psi$};
\node[above] at (-1.1,.85) {$\frac{\pi}{3}+\psi$};
\node[] at (1.5155,.875) {$a$};
\node[left] at (-1.6155,.775) {$b$};
\node[] at (0,-1.8) {$c$};
\node[] at (0,.85) {$c^\prime$};
\node[] at (-.866,-.5) {$a^\prime$};
\node[] at (.866,-.5) {$b^\prime$};
\node[] at (0,0) {$0$};
\node[] at (11.5155,.875) {$a$};
\node[left] at (8.3845,.775) {$b$};
\node[] at (10,-1.8) {$c$};
\node[] at (10,1.5) {$c^\prime$};
\node[] at (8.7009,-.75) {$a^\prime$};
\node[] at (11.299,-.75) {$b^\prime$};
\node[] at (10,0) {$0$};
\end{tikzpicture}

\end{center}
\caption{The geometry of networks $\mathscr{G}_f^I$ and $\mathscr{G}_f^{II}$. }
\label{1and2}
\end{figure}
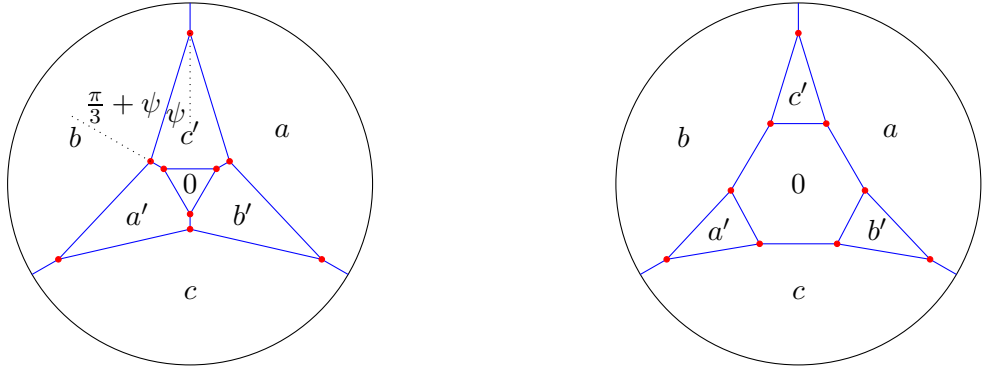

  In order that $\mathscr{G}_f^I$ and $\mathscr{G}_f^{II}$ can be identified with solutions of \eqref{inF} in Corollary \ref{cor} the $\sigma_{aa^\prime}$ need to have precise expressions in term of $\psi\in(0,\frac{\pi}{6})$ that are consequences of \eqref{eq-sig}. We can assume $\sigma_0=1$. With this normalization we find from \eqref{eq-sig} the following conditions for the minimality of  $\mathscr{G}_f^I$ and $\mathscr{G}_f^{II}$
  \begin{equation}
  \text{Case}\;\;I:\;\;\;\left.\begin{array}{l}
  \sigma_0=1,\\
   \sigma=2\cos\psi,\\
   \tau=2\sin(\frac{\pi}{6}-\psi),\\
   \tau_0=\frac{2}{\sqrt{3}}\sin(\frac{\pi}{6}-\psi).
  \end{array}\right.
  \quad\quad\;
  \text{Case}\;\;II:\;\;\; \left.\begin{array}{l}
   \sigma_0=1,\\
   \sigma=2\cos\psi,\\
   \sigma_{00}=\frac{2}{\sqrt{3}}\cos\psi,\\
   \tau_0=\frac{2}{\sqrt{3}}\sin(\frac{\pi}{6}-\psi).
  \end{array}\right.
  \label{sig-minphi}
  \end{equation}
  One can verify that, as expected, $\mathscr{F}(\mathscr{G}_f^I)=\mathscr{F}(\mathscr{G}_f^{II})$ is a constant: \begin{equation}
  \mathscr{F}(\mathscr{G}_f^I)=\mathscr{F}(\mathscr{G}_f^{II})=6\cos\psi,\;\;\,d\in(0,1),\ell
 \in(0,g(\psi)d),\;\psi\in(0,\frac{\pi}{6}).
  \label{equalF}
  \end{equation}
  This implies that the minimizers $\mathscr{G}_f^I$ and $\mathscr{G}_f^{II}$ are not isolated and do not satisfy the condition \eqref{GIn} required by Theorem \ref{fine}.
  
  In the search of a network $\mathscr{G}_f$ which is isolated and satisfies \eqref{GIn}, it remains to consider networks with branching points on the boundary of $\mathcal{P}_3$ or with branching points that coincide in the interior of $B$. In the class of networks invariant under $Z_3$, the networks to be considered are the ones that correspond to the boundary value of $(d,\ell)$ and precisely: the one-parameter families defined either by $(d,\ell)=(d,g(\psi)d)$, $d\in(0,1)$ or by $(d,\ell)=(1,\ell)$, $\ell\in(0,g(\psi))$. These families can be disregarded since, for each member $\mathscr{G}$ of either family, we have again $\mathscr{F}(\mathscr{G})=6\cos\psi$ and  $\mathscr{G}$ is not an isolated solution of \eqref{inF}. Hence we focus on the networks $\mathscr{G}^{0}$, $\mathscr{G}^*$ and $\mathscr{G}^\triangledown$ corresponding to $(d,\ell)=(0,0),(1,0),(1,g(\psi))$, see Figure \ref{01g}.

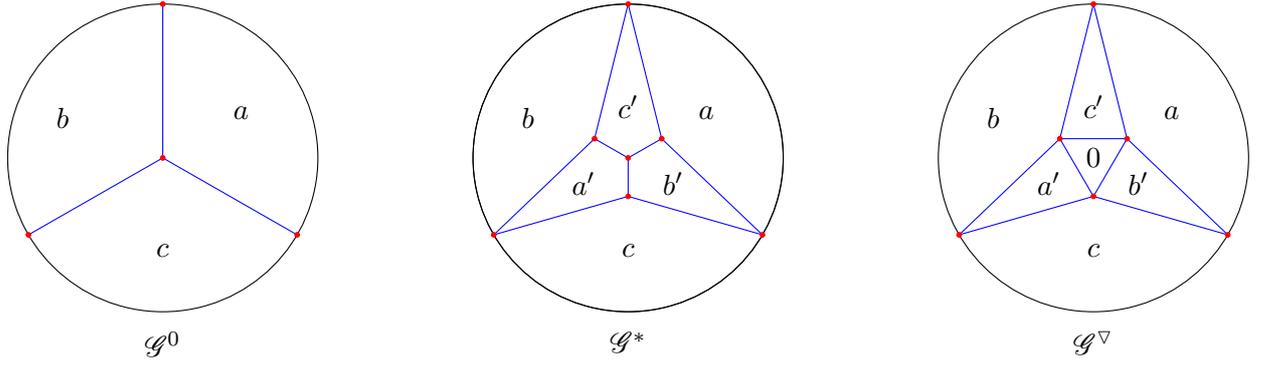
\begin{figure}
  \begin{center}
\begin{tikzpicture}[scale=.68]
\draw[] (0,0) circle [radius=3];
\draw[] (-9,0) circle [radius=3];
\draw[] (9,0) circle [radius=3];
\draw[blue] (-9,0)-- (-9,3);
\draw[blue] (-9,0)-- (-11.598,-1.5);
\draw[blue] (-9,0)-- (-6.402,-1.5);
\draw[fill=red,red] (-9,0) circle [radius=.045];
\draw[fill=red,red] (-9,3) circle [radius=.045];
\draw[fill=red,red] (-11.598,-1.5) circle [radius=.045];
\draw[fill=red,red]  (-6.402,-1.5) circle [radius=.045];
\node[] at (-7.4845,.875) {$a$};
\node[left] at (-10.6155,.775) {$b$};
\node[] at (-9,-1.8) {$c$};
\draw[blue] (9,3)-- (9.6495,.375);
\draw[blue] (9,3)-- (8.3504,.375);
\draw[blue] (11.598,-1.5)--(9.6495,.375);
\draw[blue] (11.598,-1.5)--(9,-.75);
\draw[blue] (6.402,-1.5)--(8.3504,.375);
\draw[blue] (6.402,-1.5)--(9,-.75);
\draw[blue] (9.6495,.375)--(9,-.75);
\draw[blue] (8.3504,.375)--(9.6495,.375);
\draw[blue] (9,-.75)--(8.3504,.375);
\draw[fill=red,red] (9,3) circle [radius=.045];
\draw[fill=red,red] (6.402,-1.5) circle [radius=.045];
\draw[fill=red,red] (9,-.75) circle [radius=.045];
\draw[fill=red,red] (11.598,-1.5) circle [radius=.045];
\draw[fill=red,red] (9.6495,.375) circle [radius=.045];
\draw[fill=red,red] (8.3505,.375) circle [radius=.045];
\node[] at (1.5155,.875) {$a$};
\node[left] at (-1.6155,.775) {$b$};
\node[] at (0,-1.8) {$c$};
\node[] at (0,1) {$c^\prime$};
\node[] at (-.866,-.5) {$a^\prime$};
\node[] at (.866,-.5) {$b^\prime$};

\draw[] (0,0) circle [radius=3];
\draw[blue] (0,3)-- (.6495,.375);
\draw[blue] (0,3)-- (-.6495,.375);
\draw[blue] (2.598,-1.5)--(.6495,.375);
\draw[blue] (2.598,-1.5)--(0,-.75);
\draw[blue] (-2.598,-1.5)--(-.6495,.375);
\draw[blue] (-2.598,-1.5)--(0,-.75);
\draw[blue] (0,0)--(0,-.75);
\draw[blue] (0,0)--(.6495,.375);
\draw[blue] (0,0)--(-.6495,.375);
\draw[fill=red,red] (0,0) circle [radius=.045];
\draw[fill=red,red] (0,3) circle [radius=.045];
\draw[fill=red,red] (-2.598,-1.5) circle [radius=.045];
\draw[fill=red,red] (0,-.75) circle [radius=.045];
\draw[fill=red,red] (2.598,-1.5) circle [radius=.045];
\draw[fill=red,red] (.6495,.375) circle [radius=.045];
\draw[fill=red,red] (-.6495,.375) circle [radius=.045];
\node[] at (10.5155,.875) {$a$};
\node[left] at (7.3845,.775) {$b$};
\node[] at (9,-1.8) {$c$};
\node[] at (9,1) {$c^\prime$};
\node[] at (8.134,-.5) {$a^\prime$};
\node[] at (9.866,-.5) {$b^\prime$};
\node[] at (9,0) {$0$};
\node[below] at (-9,-3.2) {$\mathscr{G}^0$};
\node[below] at (0,-3.2) {$\mathscr{G}^*$};
\node[below] at (9,-3.2) {$\mathscr{G}^\triangledown$};

\end{tikzpicture}

\end{center}
\caption{The networks $\mathscr{G}^0, \mathscr{G}^*$ and $\mathscr{G}^\triangledown$
corresponding respectively to $d=0$, $(d,\ell)=(1,0)$ and $(d,\ell)=(1,g(\psi))$. }
\label{01g}
\end{figure}

  For each of these networks there are branching points that coincide and arcs that reduce to a point. For instance, for  $\mathscr{G}^{0}$, the entire $G_f$ coincides with the center $O$ of $B$. We plan to show that the $\sigma_{aa^\prime}$ can be chosen in such way that  $\mathscr{G}^\triangledown$ satisfies \eqref{inF}. In a similar way one can find condition ensuring the same for   $\mathscr{G}^{0}$ or $\mathscr{G}^*$.
   \begin{proposition}
   \label{ET}
   Given $\psi\in(0,\frac{\pi}{6})$ assume that the $\sigma_{aa^\prime}$ satisfy \eqref{triangle}, \eqref{sigma-pr} and the conditions
   \begin{equation}
   \begin{split}
   &\sigma_0=1,\\
   &\tau_0=\frac{2}{\sqrt{3}}\sin(\frac{\pi}{6}-\psi),\\
   &\sigma>2\cos\psi,\\
   &\sigma_{00}>\frac{2}{\sqrt{3}}\cos\psi,\\
   &\tau>2\sin(\frac{\pi}{6}-\psi).\\
   \end{split}
   \label{Sig-Cond}
   \end{equation}
   Then  $\mathscr{G}^\triangledown\in\partial\mathcal{G}_f$ is the unique network invariant under $Z_3$ that satisfies \eqref{inF} and the geometric inequality \eqref{GIn}. Moreover the decomposition $\bar{\Omega}\setminus\mathscr{G}^\triangledown$ is uniquely determined.
   \end{proposition}
   \begin{proof}
  $\mathscr{G}^\triangledown$ is a limit of both families of networks of type $I$ and $II$. Therefore, if $\tilde{\mathscr{G}}$ is a small perturbation of $\mathscr{G}^\triangledown$, in estimating the difference $\delta\mathscr{F}=\mathscr{F}(\tilde{\mathscr{G}})-\mathscr{F}(\mathscr{G}^\triangledown)$, we need distinguish if $\tilde{\mathscr{G}}$ is of type $I$ or $II$. In both case we change $d=1$ to $\tilde{d}=1-h$ for $h>0$ small.
  Then, if we consider $\mathscr{G}^\triangledown$ as a limit point of the networks of type $I$ we have $p_3=p_1$ and we set $s=\vert\tilde{p}_3-p_1\vert>0$ where $\tilde{p}_3 \in\mathrm{sg}(p_1,O)$, $s$ small. If, instead, $\mathscr{G}^\triangledown$ is seen as a limit element of the family of type $II$, we set $\rho=\vert\tilde{p}_1-p_1\vert>0$ where $\tilde{p}_1$ is a point in the closed half-plane that contains $q$ and is bounded by the line through $p_1$ and $O$. We set $\tilde{p}_1=(\rho\cos\beta,\rho\sin\beta)$ where $\beta\in[0,\pi]$ and $(\rho,\beta)$ are the polar coordinates with origin in $p_1$ and polar axis the half-line determined by $p_1$ and $O$ (see Figure \ref{pert}).

 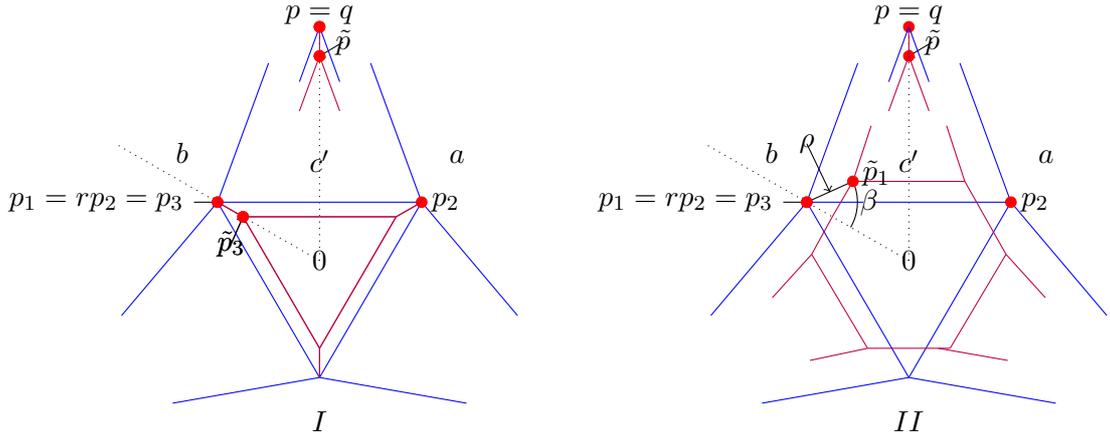
\begin{figure}
  \begin{center}
\begin{tikzpicture}[scale=1.55]
\draw[dotted] (0,0)--(0,2);
\draw[dotted] (0,0)--(-1.732,1);
\node[] at (0,2.1) {$p=q$};
\draw[->](.2,1.87)--(0,1.75);
\node[] at (.2,1.87) {$\tilde{p}$};
\draw[fill=red,red] (0,2) circle [radius=.045];
\draw[fill=red,red] (0,1.75) circle [radius=.045];
\draw[blue] (0,2)-- (-.171,1.53);
\draw[blue] (0,2)-- (.171,1.53);
\draw[purple] (0,1.75)-- (-.171,1.28);
\draw[purple] (0,1.75)-- (.171,1.28);
\draw[purple] (0,1.75)-- (0,2);
\draw[->](-1.066,.5)--(-.866,.5);
\draw[fill=red,red] (-.866,.5) circle [radius=.045];
\node[left] at (-1.066,.5) {$p_1=rp_2=p_3$};
\draw[blue] (-.866,.5)-- (0,.5);
\draw[purple] (-.6495,.375)-- (0,.375);
\draw[blue] (-.866,.5)-- (-.433,1.6896);
\draw[purple] (-.866,.5)-- (-.6495,.375);
\draw[blue] (-.866,.5)--(0,-1);
\draw[purple] (-.6495,.375)--(0,-.75);
\draw[blue] (-.866,.5)--(-1.6797,-.4698);
\draw[blue] (.866,.5)--(1.6797,-.4698);
\draw[blue] (-1.2467,-1.2198)--(0,-1);
\draw[blue] (1.2467,-1.2198)--(0,-1);
\draw[blue] (.866,.5)-- (0,.5);
\draw[purple] (.6495,.375)-- (0,.375);
\draw[blue] (.866,.5)-- (.433,1.6896);
\draw[purple] (.866,.5)-- (.6495,.375);
\draw[blue] (.866,.5)--(0,-1);
\draw[purple] (.6495,.375)--(0,-.75);
\draw[purple] (0,-1)--(0,-.75);
\node[right] at  (.866,.5) {$p_2$};
\draw[->] (-.7495,.145)--(-.6495,.375);
\node[] at (-.7495,.145){$\tilde{p}_3$};
\node[] at  (0,.85) {$c^\prime$};
\node[right,above] at  (1.166,.75) {$a$};
\node[left,above] at  (-1.166,.75) {$b$};
\node[] at (0,0){$0$};
\draw[fill=red,red] (-.6495,.375) circle [radius=.045];
\draw[fill=red,red] (-.866,.5) circle [radius=.045];
\draw[fill=red,red] (0,2) circle [radius=.045];
\draw[fill=red,red] (0,1.75) circle [radius=.045];
\draw[fill=red,red] (.866,.5) circle [radius=.045];
\draw[dotted] (5,0)--(5,2);
\draw[dotted] (5,0)--(3.268,1);
\node[] at (5,2.1) {$p=q$};
\draw[->](5.2,1.87)--(5,1.75);
\node[] at (5.2,1.87) {$\tilde{p}$};
\draw[fill=red,red] (5,2) circle [radius=.045];
\draw[fill=red,red] (5,1.75) circle [radius=.045];
\draw[blue] (5,2)-- (4.829,1.53);
\draw[blue] (5,2)-- (5.171,1.53);
\draw[purple] (5,1.75)-- (4.829,1.28);
\draw[purple] (5,1.75)-- (5.171,1.28);
\draw[purple] (5,1.75)-- (5,2);
\draw[->](3.934,.5)--(4.134,.5);
\draw[fill=red,red] (4.134,.5) circle [radius=.045];
\node[left] at (3.934,.5) {$p_1=rp_2=p_3$};
\draw[blue] (4.134,.5)-- (5,.5);
\draw[purple] (-.6495,.375)-- (0,.375);
\draw[blue] (4.134,.5)-- (4.567,1.6896);
\draw[purple] (-.866,.5)-- (-.6495,.375);
\draw[blue] (4.134,.5)--(5,-1);
\draw[purple] (-.6495,.375)--(0,-.75);
\draw[blue] (4.134,.5)--(3.3203,-.4698);
\draw[blue] (5.866,.5)--(6.6797,-.4698);
\draw[blue] (3.7533,-1.2198)--(5,-1);
\draw[blue] (6.2467,-1.2198)--(5,-1);
\draw[blue] (5.866,.5)-- (5,.5);
\draw[purple] (.6495,.375)-- (0,.375);
\draw[blue] (5.866,.5)-- (5.433,1.6896);
\draw[purple] (.866,.5)-- (.6495,.375);
\draw[blue] (5.866,.5)--(5,-1);
\draw[purple] (.6495,.375)--(0,-.75);
\draw[purple] (0,-1)--(0,-.75);
\draw[purple] (4.3505,.375)--(4.5255,.6781);
\draw[purple] (4.3505,.375)--(4.1755,.0539);
\draw[purple] (4.65,-.75)--(5.35,-.75);
\draw[purple] (4.65,-.75)--(4.1755,.0539);
\draw[purple](5.4745,.6781)--(4.5255,.6781);
\draw[purple](5.4745,.6781)--(5.8245,.0539);
\draw[purple](5.35,-.75)--(5.8245,.0539);
\draw[purple](4.5255,.6781)--(4.68,1.1536);
\draw[purple](5.4745,.6781)--(5.32,1.1536);
\draw[purple](4.65,-.75)--(4.161,-.8539);
\draw[purple](5.25,-.75)--(5.839,-.8539);
\draw[purple] (4.1755,.0539)--(3.8415,-.3176);
\draw[purple] (5.8245,.0539)--(6.1585,-.3176);
\draw[] (4.134,.5)--(4.5255,.6781);
\draw [] (4.5064,.285) arc [radius=.43, start angle=-30, end angle=24.4622];
\draw[->] (4.134,1) --(4.329,.589);
\draw[fill=white,white] (4.664,.5) circle [radius=.045];
\node[] at (4.134,1) {$\rho$};
\node[]at (4.664,.5) {$\beta$};
\node[]at (4.7255,.7481) {$\tilde{p}_1$};
\node[right] at  (5.866,.5) {$p_2$};
\draw[->] (-.7495,.145)--(-.6495,.375);
\node[] at (-.7495,.145){$\tilde{p}_3$};
\node[] at  (5,.85) {$c^\prime$};
\node[right,above] at  (6.166,.75) {$a$};
\node[left,above] at  (3.834,.75) {$b$};
\node[] at (5,0){$0$};
\draw[fill=red,red] (-.6495,.375) circle [radius=.045];
\draw[fill=red,red] (4.134,.5) circle [radius=.045];
\draw[fill=red,red] (0,2) circle [radius=.045];
\draw[fill=red,red] (0,1.75) circle [radius=.045];
\draw[fill=red,red] (5.866,.5) circle [radius=.045];
\draw[fill=red,red] (4.5255,.6781) circle [radius=.045];
\node[below] at (0,-1.2){$I$};
\node[below] at (5,-1.2){$II$};
\end{tikzpicture}
\end{center}
\caption{Schematic illustration of the perturbation $\tilde{\mathscr{G}}$. }
\label{pert}
\end{figure}

   With these definitions we have, using also the symmetry,
  \begin{equation}
  \begin{split}
  &\delta\mathscr{F}=6\Big((\frac{1}{2}\sigma-\cos\psi\sigma_0)h+(\frac{1}{2}\tau-\cos(\frac{\pi}{6})\tau_0)s\Big)+\ldots\geq 0,\;\;
  \tilde{\mathscr{G}}\;\text{of type}\;I,\\
  &\delta\mathscr{F}=6\Big((\frac{1}{2}\sigma-\cos\psi\sigma_0)h+(\sigma_{00}\sin\beta-\cos(\beta-\frac{\pi}{6})\tau_0\\
 &-\cos(\frac{2}{3}\pi-\psi-\beta)\sigma_0)\rho
  \Big)+\ldots\geq 0,\;\;
  \tilde{\mathscr{G}}\;\text{of type}\;II.
  \end{split}
  \label{Stab-Cond}
  \end{equation}
  These inequalities must hold for $\beta\in(0,\pi)$, $h\geq 0$, $s\geq 0$  and $\rho\geq 0$. It follows:
  \begin{equation}
  \begin{split}
  &\sigma\geq 2\cos\psi\sigma_0,\\
  &\tau\geq\sqrt{3}\tau_0,\\
  &-\frac{\sqrt{3}}{2}\tau_0-\cos(\frac{2}{3}\pi-\psi)\sigma_0\geq 0,\;\;\text{if}\,\;\beta=0,\\
  &\frac{\sqrt{3}}{2}\tau_0-\cos(\frac{\pi}{3}+\psi)\sigma_0\geq 0,\;\;\text{if}\,\;\beta=\pi,\\
  &\sin\beta\sigma_{00}-\cos(\beta-\frac{\pi}{6})\tau_0-\cos(\frac{2}{3}\pi-\psi-\beta)\sigma_0\geq 0,\;\text{for}\,\;\beta\in(0,\pi).
  \end{split}
  \label{cond-hs}
  \end{equation}
  Since $-\cos(\frac{2}{3}\pi-\psi)=\cos(\frac{\pi}{3}+\psi)=\sin(\frac{\pi}{6}-\psi)$, \eqref{cond-hs}$_{3,4}$ imply 
  \begin{equation}
  \tau_0=\frac{2}{\sqrt{3}}\sin(\frac{\pi}{6}-\psi)\sigma_0.
  \label{tau0-as}
  \end{equation}
  From \eqref{cond-hs}$_5$ and this expression of $\tau_0$ we obtain
  \begin{equation}
  \begin{split}
  &\sin\beta\sigma_{00}-(\frac{2}{\sqrt{3}}\cos(\beta-\frac{\pi}{6})\sin(\frac{\pi}{6}-\psi)
  +\cos(\frac{2}{3}\pi-\psi-\beta))\sigma_0\\
  &=\sin\beta(\sigma_{00}-(\frac{1}{\sqrt{3}}\sin(\frac{\pi}{6}-\psi)+\sin(\frac{2}{3}\pi-\psi))\sigma_0)
  -\cos\beta(\sin(\frac{\pi}{6}-\psi)+\cos(\frac{2}{3}\pi-\psi))\sigma_0\\
  &=\sin\beta(\sigma_{00}-(\frac{1}{\sqrt{3}}\sin(\frac{\pi}{6}-\psi)+\sin(\frac{2}{3}\pi-\psi))\sigma_0)\geq 0,\;\text{for}\,\;\beta\in(0,\pi),\\
  &\Rightarrow\\
  &\sigma_{00}\geq\big(\frac{1}{\sqrt{3}}\sin(\frac{\pi}{6}-\psi)+\cos(\sin(\frac{\pi}{6}-\psi))\big)\sigma_0
  =\frac{2}{\sqrt{3}}\cos\psi\sigma_0.
  \end{split}
  \label{sigma>}
  \end{equation}
  From \eqref{Stab-Cond} and \eqref{sigma>} and some analysis of the higher order terms in \eqref{Stab-Cond} it follows that with the strict inequalities for $\sigma, \sigma_{00}, \tau$ the network $\mathscr{G}^\triangledown$
  is an isolated local minimizer of $\mathscr{F}$ in the class of networks invariant under $Z_3$. To conclude that indeed, under the conditions \eqref{Stab-Cond}, $\mathscr{G}^\triangledown$ is the unique solution of \eqref{inF} we need to show that the energy $\mathscr{F}(\mathscr{G}^\triangledown)$ is strictly less than the energy of the competitors $\mathscr{G}^0$ and $\mathscr{G}^*$. Since, in spite of the inequalities for $\sigma, \sigma_{00}, \tau$ we still have $\mathscr{F}(\mathscr{G}^\triangledown)=6\cos\psi$, we have
 \[ \begin{split}
 &\mathscr{F}(\mathscr{G}^{0,0})-\mathscr{F}(\mathscr{G}^\triangledown)=3\sigma-6\cos\psi=3(\sigma-2cos{\psi}>0,\\
 &\mathscr{F}(\mathscr{G}^{1,0})-\mathscr{F}(\mathscr{G}^\triangledown)=3(\tau-2\sin(\frac{\pi}{6}-\psi))\lambda
 >0,\\
 \end{split}\]
 where $\lambda=\vert p_1-O\vert$. This concludes the proof.
\end{proof}
On the basis of Proposition \ref{ET} we have that, under the assumptions of Theorem \ref{fine}, if $W$ has $N=7$ and the $\sigma_{aa^\prime}$  satisfy \eqref{Sig-Cond}, then an equivariant minimizer $u^\epsilon$ of problem \eqref{min} has a well determined structure which via \eqref{Stru} is intimately related to the network $\mathscr{G}^\triangledown$ and since $\mathscr{G}^\triangledown$ is unique all such minimizers have a similar structure. In the proof of Proposition \ref{ET} the fact that $(1,g(\psi))$ is a corner point of the set of the parameter $(d,\ell)$ plays an essential role. Indeed this fact implies that not all perturbations are reversible leading to the inequalities in \eqref{Sig-Cond} and in turn to the possibility of applying Theorem \ref{fine}.
On the other hand, when the $\sigma_{aa^\prime}$ have the exact values given in \eqref{sig-minphi}, there exist continuous families of minimizers of $\mathscr{F}(\mathscr{G})$ and it is natural to question about the relationship of minimizers of problem \eqref{min} with these families of geometric objects. A possibility could be that, for the parabolic dynamics associated to the functional $J_B^\epsilon$, there exists a phenomenon of slow motion along an invariant manifold $\mathscr{G}\rightarrow v^\epsilon\in H^1(B;\R^m)$ where $\mathscr{G}$ is the generic element in the family and $v^\epsilon$ is a suitable map with diffuse interfaces related to $\mathscr{G}$. The minimizers $u^\epsilon$ of \eqref{min} should be stationary points in this manifold and an interesting problem could be the characterization of the  $\mathscr{G}$ that correspond to these stationary points.

If this heuristic description is correct, one can also rise the question of the existence of entire solutions $u\in H^1(\R^2;\R^m)$ of \eqref{AC} with a nontrivial complex structure in a compact set.

\section{Appendix}
\label{App}
\begin{lemma}
\label{UBerr}
There is a constant $C>0$ such that in the upper bound \eqref{UB} we have
\[e_\epsilon\leq C\epsilon\vert\ln{\epsilon}\vert^2.\]
\end{lemma}
\begin{proof}
Let $\gamma\in\mathscr{G}_f$  be one of the $n_s$ arcs in $\mathscr{G}_f$. We assume that $\gamma$ coincides with a segment $\mathrm{sg}[p_1,p_2]$ with $p_1$ and $p_2$ branching points of $\mathscr{G}_f$. The general case can be treated by similar arguments. We have $\gamma\subset\partial S_{a,f}\cap\partial S_{a^\prime,f}$ for some $a,a^\prime$. We let $\nu$ be a unit vector orthogonal to $\tau=\frac{p_2-p_1}{\vert p_2-p_1\vert}$ and oriented from $S_{a^\prime,f}$ to $S_{a,f}$.
Let $\alpha_0>0$ be the minimum value of the angle $\alpha$ between the tangents at $p$ to two distinct arcs of $\mathscr{G}_f$ that originate at a branching point $p$.

Set $h=\kappa_0\epsilon\vert\ln{\epsilon}\vert$ and $\rho=\frac{4 h}{\sin{\frac{\alpha_0}{2}}}$ with $\kappa_0>$ a constants to be chosen later. Define
\[\begin{array}{l}
\tilde{E}_\gamma=\{x: x=p_1+s\tau+ t\nu,\;s\in[0,\vert\gamma\vert]\;t\in[- h,h]\},\\\\
\tilde{E}_\gamma^\pm=\{x: x=p_1+s\tau+ t\nu,\;s\in[0,\vert\gamma\vert]\;t\in[\pm h,\pm 2h]\},\\\\
E_\gamma=\tilde{E}_\gamma\setminus(B_\rho(p_1)\cup B_\rho(p_2)),\\\\
E_\gamma^\pm=\tilde{E}_\gamma^\pm\setminus(B_\rho(p_1)\cup B_\rho(p_2)).
\end{array}\]
We define a test map $u_{\mathrm{test}}$ by setting, for each $\gamma\in\mathscr{G}_f$
\begin{equation}
\label{utest}
\begin{split}
&u_{\mathrm{test}}(x)=u_{a^\prime a}(\frac{t}{\epsilon}),\;\;x=p_1+s\tau+t\nu\in E_\gamma,\\\\
&u_{\mathrm{test}}(x)=u_{a^\prime a}(\frac{h}{\epsilon})(2-\frac{t}{h})+(\frac{t}{h}-1)a,\;\;x=p_1+s\tau+t\nu\in E_\gamma^+,\\\\
&u_{\mathrm{test}}(x)=u_{a^\prime a}(\frac{-h}{\epsilon})(2-\frac{\vert t\vert}{h})+(\frac{\vert t\vert}{h}-1)a^\prime,\;\;x=p_1+s\tau+t\nu\in E_\gamma^-.
\end{split}
\end{equation}
This defines $u_{\mathrm{test}}$ in the set $E=\cup_{\gamma\in\mathscr{G}_f}(E_\gamma\cup E_\gamma^+\cup E_\gamma^-)$. To complete the definition of $u_{\mathrm{test}}$ set $B=\cup_{p\in P}B_\rho(p)$ where $P$ denote the set of the branching point of $\mathscr{G}_f$ and define
\begin{equation}
\label{utest1}
\begin{split}
&u_{\mathrm{test}}(x)=a,\;\;\text{for}\;x\in S_{a,f}\setminus(E\cup B),\;a\in A.\\\\
\end{split}
\end{equation}
From \eqref{utest} and \eqref{utest1} the map $u_{\mathrm{test}}$ is defined on $\Omega\setminus B$ and in particular on $\partial B_\rho(p)$ for each $p\in P$. Therefore we can define
\begin{equation}
\label{utest2}
u_{\mathrm{test}}(x)=\frac{r}{\rho}u_{\mathrm{test}}(p+\rho\mu),\;\;\text{for}\;x=p+r\mu,\;r\in(0,\rho],\;\mu\in\SF.
\end{equation}
From $h_1$ there are constants $C_W>0$ and $\delta_0>0$ such that
\begin{equation}
\label{ww}
W(a+z)\leq\frac{1}{2}C_W\vert z\vert^2,\;\;\vert z\vert\leq\delta_0,\;\;a\in A.
\end{equation}
We also have, for some constants $c^0>0$, $C^0>0$ independent from $a,a^\prime\in A$,
\begin{equation}
\label{uaa}
\begin{split}
&\vert u_{a^\prime a}(\frac{t}{\epsilon})-a\vert\leq C^0e^{-c^0\frac{t}{\epsilon}},\;\;t>0,\\\\
&\vert u^\prime_{a^\prime a}(\frac{t}{\epsilon})\vert\leq C^0e^{-c^0\frac{\vert t\vert}{\epsilon}},\;\;t\in\R.
\end{split}
\end{equation}
It results
\begin{equation}
\label{JE}
J_{E_\gamma}^\epsilon(u_{\mathrm{test}})\leq\vert\gamma\vert\int_{-\frac{h}{\epsilon}}^{\frac{h}{\epsilon}}
\vert u^\prime_{a^\prime a}(s)\vert^2ds\leq\sigma_{a^\prime a}\vert\gamma\vert.
\end{equation}
By adding and subtracting $a$ we rewrite \eqref{utest}$_2$ in the form
\begin{equation}
u_{\mathrm{test}}(x)=a+(u_{a^\prime a}(\frac{h}{\epsilon})-a)(2-\frac{t}{h}),\;\;x=p+s\tau+t\nu,\;t\in[h,2h],
\label{newform}
\end{equation}
which implies
\[\frac{\partial}{\partial t} u_{\mathrm{test}}(x)=-\frac{1}{h}(u_{a^\prime a}(\frac{h}{\epsilon})-a).\]
From these expressions \eqref{ww} and \eqref{uaa} we obtain
\begin{equation}
\begin{split}
&\frac{1}{\epsilon}\int_{E_\gamma^+}W(u_{\mathrm{test}})dx\leq\frac{\vert\gamma\vert}{2\epsilon}C_W(C^0)^2\int_h^{2h}
e^{-2c^0\frac{h}{\epsilon}}dt\leq\frac{\vert\gamma\vert}{2}C_W(C^0)^2\frac{h}{\epsilon}e^{-2c^0\frac{h}\epsilon}
\leq C\epsilon^{2c^0\kappa_0}\vert\ln{\epsilon}\vert,\\\\
&\epsilon\int_{E_\gamma^+}\vert\frac{\partial}{\partial t}u_{\mathrm{test}}\vert^2dx\leq\vert\gamma\vert(C^0)^2\frac{\epsilon}{h}e^{-2c^0\frac{h}{\epsilon}}\leq
C\frac{\epsilon^{2c^0\kappa_0}}{\vert\ln{\epsilon}\vert}.
\end{split}
\label{JE+}
\end{equation}
To estimate $J^\epsilon_{B_{\rho}(p)}( u_{\mathrm{test}})$, $p\in P$ we focus on $B_{\rho}(p_1)$. The same estimate applies to all $p\in P$. By definition $u_{\mathrm{test}}$ is bounded in  $B_{\rho}(p_1)$ and the same is true for $W(u_{\mathrm{test}})$. It follows
\begin{equation}
\frac{1}{\epsilon}\int_{B_{\rho}(p_1)}W(u_{\mathrm{test}})dx\leq \frac{C}{\epsilon}\rho^2\leq C\epsilon\vert\ln{\epsilon}\vert^2.
\label{JWB}
\end{equation}
To estimate the contribution of $\nabla u_{\mathrm{test}}$ to $J^\epsilon_{B_{\rho}(p_1)}( u_{\mathrm{test}})$ we  divide $B_{\rho}(p_1)$ in sectors and focus on the sectors $T_i\subset B_{\rho}(p_1)$, $i=1,2,3$ defined by
\begin{equation}
\begin{split}
& T_1=\{x=p_1+r\mu: r\in[0,\rho],\;\mu\cdot\nu=\sin{\theta}\in[0,h],\\
& T_2=\{x=p_1+r\mu: r\in[0,\rho],\;\mu\cdot\nu=\sin{\theta}\in[h,2h],\\
& T_3=\{x=p_1+r\mu: r\in[0,\rho],\;\mu\cdot\nu=\sin{\theta}\in[2h,\frac{\alpha}{2}],\\
\end{split}
\label{sect3}
\end{equation}
where $\alpha\geq\alpha_0$ is the angle determined by the two arcs on the boundary of $S_{a,f}$ that originate in $p_1$.
The estimates that we derive for $T_i$, $i=1,2,3$ apply to all the other similar sectors.

For $x\in T_1$ we have $u_{\mathrm{test}}(x)=\frac{r}{\rho}u_{a^\prime a}(\frac{\rho\sin{\theta}}{\epsilon})$. Since $u_{a^\prime a}$ and $u^\prime_{a^\prime a}$ are bounded functions it follows
\begin{equation}\vert\frac{\partial}{\partial r}u_{\mathrm{test}}\vert\leq\frac{C}{\rho},\quad
\vert\frac{\partial}{\partial\theta}u_{\mathrm{test}}\vert\leq C\frac{r}{\epsilon}.
\label{derEst}
\end{equation}
Hence
\begin{equation}
\frac{\epsilon}{2}\int_{T_1}\vert\nabla u_{\mathrm{test}}\vert^2dx\leq C\epsilon\int_0^{\sin^{-1}(\frac{h}{\rho})}
(\frac{1}{\rho^2}+\frac{1}{\epsilon^2})rdrd\theta\leq C\epsilon(1+\frac{\rho^2}{\epsilon^2})\leq C\epsilon\vert\ln{\epsilon}\vert^2.
\label{JGradT1}
\end{equation}
From \eqref{newform} and $x\in T_2$ it follows $ u_{\mathrm{test}}(x)=
\frac{r}{\rho}\Big(a+(u_{a^\prime a}(\frac{h}{\epsilon})-a)(2-\frac{\rho\sin{\theta}}{h})\Big).$ From this and
$\vert u_{a^\prime a}(\frac{h}{\epsilon})-a\vert\leq C^0e^{-c^0\frac{h}{\epsilon}}\leq \epsilon^{c^0\kappa_0}$
we have
\[\vert\frac{\partial}{\partial r}u_{\mathrm{test}}\vert\leq\frac{C}{\rho},\quad
\vert\frac{\partial}{\partial\theta}u_{\mathrm{test}}\vert\leq C \epsilon^{c^0\kappa_0}\frac{r}{h},\]
and in turn
\begin{equation}
\frac{\epsilon}{2}\int_{T_2}\vert\nabla u_{\mathrm{test}}\vert^2dx\leq
C\epsilon(\frac{1}{\rho^2}+\epsilon^{2c^0\kappa_0}\frac{1}{h^2})\int_{\sin^{-1}(\frac{h}{\rho})}
^{\sin^{-1}(\frac{2h}{\rho})}
\int_0^\rho rdrd\theta\leq C\epsilon,
\label{JGradT2}
\end{equation}
where we have also used $\rho=\frac{2h}{\sin{\frac{\alpha_0}{2}}}$.
For $x\in T_3$ we have $u_{\mathrm{test}}(x)=\frac{r}{\rho}a$ and we obtain
\begin{equation}
\frac{\epsilon}{2}\int_{T_3}\vert\nabla u_{\mathrm{test}}\vert^2dx\leq C\epsilon.
\label{JGradT3}
\end{equation}
From \eqref{JWB},\eqref{JGradT1},\eqref{JGradT2},\eqref{JGradT3} we conclude
\begin{equation}
J^\epsilon_{B_\rho(p_1)}(u_{\mathrm{test}})\leq C\epsilon\vert\ln{\epsilon}\vert^2.
\label{JB1}
\end{equation}
We have assumed that $p_1$ is a branching point. A similar estimate applies to  $J^\epsilon_{B_\rho(p)}\cap\Omega$ when $p\partial\Omega$ is an end point. Therefore from \eqref{JB}, \eqref{JE}, \eqref{JE+} it follows
\begin{equation}
J_\Omega^\epsilon(u_{\mathrm{test}})\leq\sum_{\gamma\in\mathscr{G}_f}\sigma_{a_\gamma a_\gamma^\prime}\vert\gamma\vert+C\epsilon\vert\ln{\epsilon}\vert^2.
\label{UBerr1}
\end{equation}
The proof is complete.
\end{proof}
\begin{lemma}
\label{gf-gf}
Assume $\vert I_a\vert\leq$, $a\in\tilde{A}$. Then there exist $\hat{\mathscr{G}}^\prime\in\bar{\mathcal{G}}$ which has the same end points as $\mathscr{G}_f$ and satisfies \eqref{gprime-g} with
\[e_\epsilon=\]
\end{lemma}
\begin{proof}
To define $\hat{\mathscr{G}}^\prime$ we only change the $\tilde{N}$ arcs $\hat{\gamma}_j$, $j=1,\ldots,\tilde{N}$ which have one of the extremes in $I_{a_j}$, $a_j\in\tilde{A}$. Let $\hat{\gamma}\in\hat{\mathscr{G}}$ be one of these arcs and let $\gamma_f\in\mathscr{G}_f$ the corresponding arc in $\mathscr{G}_f$. The arcs $\hat{\gamma}$ and $\gamma_f$ are determined by the condition that both have one of the extremes in the same arc $I_a$, for some $a\in\tilde{A}$:
\[\hat{\gamma}(1)\in I_a,\quad\gamma_f(1)\in I_a.\]
Let $\mathrm{arc}(\hat{\gamma}(1),\gamma_f(1))\subset I_a$ be the arc with extremes $\hat{\gamma}(1)$ and $\gamma_f(1)$. Set
\begin{equation}
\ell=\vert\mathrm{arc}(\hat{\gamma}(1),\gamma_f(1))\vert,
 \label{ell}
 \end{equation}
 and let $[0,\ell]\ni s\rightarrow\eta(s)\in\partial\Omega$, $\eta(0)=\hat{\gamma}(1)$, $\eta(\ell)=\gamma_f(1)$, $s$ arclength, the representation of $\mathrm{arc}(\hat{\gamma}(1),\gamma_f(1))$. For constructing $\hat{\mathscr{G}}^\prime$ we replace $\hat{\gamma}$ with the arc $\gamma$ defined by
\begin{equation}
\begin{split}
&\gamma(t)=\hat{\gamma}((1+h)t),\;\;t\in[0,\frac{1}{1+h}],\\
&\gamma(t)=\eta(\frac{\ell}{h}((1+h)t-1)),\;\;t\in(\frac{1}{1+h},1],
\end{split}
\label{gam-def}
\end{equation}
where we have set $h=$.
Note in particular that $\gamma(1)=\gamma_f(1)$.

We can assume that $\hat{\gamma}$ is Lipshitz with Lipshitz constant $K>0$ given in \eqref{B-length}, therefore
\begin{equation}
\begin{split}
&\vert\gamma(t)-\hat{\gamma}(t)\vert=\vert\hat{\gamma}((1+h)t)-\hat{\gamma}(t)\vert\leq Kh,\;\;t\in[0,\frac{1}{1+h}],\\\\
&\vert\gamma(t)-\hat{\gamma}(t)\vert\leq\vert\eta(\frac{\ell}{h}((1+h)t-1))-\hat{\gamma}(1)\vert
+\vert\hat{\gamma}(1)-\hat{\gamma}(t)\vert\\&\leq\ell+K\frac{h}{1+h}
,\;\;\quad\;t\in(\frac{1}{1+h},1],
\end{split}
\label{Est-gam}
\end{equation}
where we have also used $\eta(0)=\hat{\gamma}(1)$.
From \eqref{Est-gam} and $\ell\leq h$
we conclude $\|\gamma-\hat{\gamma}\|_{C^0[0,1]}\leq Ch$. Hence
\begin{equation}
d(\hat{\mathscr{G}}^\prime,\hat{\mathscr{G}})\leq Ch.
\label{dgatgg}
\end{equation}
The inequality $\mathscr{F}(\hat{\mathscr{G}}^\prime),\mathscr{F}(\hat{\mathscr{G}})\leq Ch$ follows from
\eqref{ell} and $\ell\leq h$.
The proof is complete.
\end{proof}


\begin{thebibliography}{99}






\bibitem{af6}
\newblock N.~D.~Alikakos, G.~Fusco.
\newblock \emph{Entire solutions to equivariant elliptic system with variational structure},
\newblock  Arch.\ Rational\ Mech.\ Anal. \textbf{202} (2011), pp.~567--597.


\bibitem{AF+}
\newblock N.~D.~Alikakos, G.~Fusco.
\newblock \emph{Sharp lower bounds for the vector Allen-Cahn energy and qualitative properties of minimizers under no symmetry hypothesis},
\newblock  Bull.\ Hellenic\ Math.\ Soc. \textbf{} (2022), pp.~--.


\bibitem{afs}
N.D.~Alikakos, G.~Fusco and P.~Smyrnelis
\newblock {\em Elliptic systems of phase transition type.}
\newblock  Progress in Nonlinear Differential Equations and their applications n.91, Birkhuser,\, (2018).

\bibitem{BHM}
P.~Baldi, E.~Haus and C.~Mantegazza. 
\newblock Networks self-similarly moving by
curvature with two triple junctions. 
\newblock {\em Atti Accad. \ Naz. \ Lincei \ Rend. \ Lincei Mat. \ Appl.} \textbf{28}  no. 2
(2017), pp.~323–-338 . MR 3649351


\bibitem{B}
S.~Baldo
\newblock Minimal interface criterion for phase transitions in mixtures of Cahn-Hilliard fluids.
\newblock {\em Ann.\ Inst.\ Henri Poincaré} \textbf{7} No.~2 (1990),pp.~67–-90.






\bibitem{CC}
L.~Caffarelli  and A.~Cordoba.
\newblock Uniform convergence of a singular perturbation problem.
\newblock {\em Commun.\ Pure \ Appl. \ Math.} \textbf{48} No.~ (1995), pp.~1--12.

\bibitem{CGS}
J.~Carr, M.~Gurtin and M.~Slemrod.
\newblock Structured phase transitions on a finite interval.
\newblock {\em Arch.\ Rat.\ Mech.\ Anal.} \textbf{86} (1984), pp.~317--351.


\bibitem{CH}
R.~Casten  and C.~Holland.
\newblock Instability results for reaction-diffusion equations with Neumann boundary conditions.
\newblock {\em J.\ Diff. \ Eqns. } \textbf{27} No.~ (1978), pp.~266--273.

\bibitem{CG}
X. Chen and J.-S. Guo, Self–similar solutions of a 2–D multiple–phase curvature flow.
\newblock {\em Phys. D} \textbf{229} No.~ 1 (2007), pp.~22--34.

\bibitem{KL}
\newblock A.~Czarnecki, M.~Kulczychi, W.~Lubawski.
\newblock \emph{On the connectedness of boundary and complement for domains},
\newblock  Ann.\ Polin.\ Math. \textbf{103}  (2011), pp.~189--191.




\bibitem{Dieu}
J.~Dieudonn\'e,
\newblock{\em Treatise on Analysis III}
\newblock{Pure\ Appl.\ Math.} Vol.10-III Academic press 1972





\bibitem{f4}
G.~Fusco.
\newblock On the structure of minimizers of the Allen-Cahn energy for
$Z_N$ symmetric nonnegative potentials with $N + 1$ zeros.
\newblock {\em Calc.\ Var.}  \textbf{61} No.~206 (2022), pp.~








\bibitem{GM}
M.~Gurtin, H.~Matano.
\newblock On the structure of equilibrium phase transitions within the gradient theory of fluids.
\newblock {\em Quart.\ Appl.\ Math.\ J.} \textbf{46} (1088), pp.~301–-317.





\bibitem{MNP}
C.~Mantegazza, M.~Novaga and A.~Pluda.
\newblock Lectures on Curvature Flow of Networks. 
\newblock In: Dipierro, S. (eds) Contemporary Research in Elliptic PDEs and Related Topics. Springer INdAM Series, vol 33. (2019). Springer, Cham. $https://doi.org/10.1007/978-3-030-18921-1_9$


\bibitem{Ma}
 H.~Matano.
\newblock Asymptotic behavior and stability of solutions to semilinear diffusion equations.
\newblock {\em Publ.\ Res.\ Inst.\ Math.\ Sci.} \textbf{15} (1979), pp.~401–-424.



\bibitem{M}
L.~Modica,
\newblock The gradient theory of phase transitions and the minimal interface criterion.
\newblock {\em Arch.\ Rat.\ Mech.\ Anal.}, \textbf{98} (1987), pp.~123–-142.

\bibitem{N}
M.~Novak
\newblock Regularity for minimizers of a planar partitioning problem with cusps.
\newblock axXiv:2305.11865v2

\bibitem{MS}
 A.~Monteil, F.~Santambrogio,
\newblock Metric methods for heteroclinic connections in infinite dimen
sional spaces.
\newblock To appear. arXiv: 1602.05487v1

\bibitem{Rab}
P.~H.~Rabinowitz,
\newblock Solutions of heteroclinic type for some classes of semilinear elliptic
partial differential equations.
\newblock {\em J.\ Math.\ Sci.\ Univ.\ Tokio}, \textbf{1} (1994), pp.~525–-550.






\bibitem{S}
P.~Sternberg,
\newblock The effect of singular perturbation on nonconvex variational problems.
\newblock {\em Arch.\ Rat.\ Mech.\ Anal.}, \textbf{101} (1988), pp.~209–-260.




\bibitem{SZu}
P.~Sternberg and K.~Zumbrun.
\newblock Connectivity of phase boundaries in strictly convex domains.
\newblock{\em Arch.\ Rational.\ Mech.\ Anal.}  \textbf{141} (1998), pp.~375--400.

\bibitem{stoker}
J.~J.~Stoker.
\newblock Dfferential Geometry,.
\newblock{Wiley-Interscience, New York}  (1969).



 \end{thebibliography}
\end{document}